\documentclass[12pt]{article}
\usepackage{amsmath}
\usepackage{amssymb}
\usepackage{amsthm}
\usepackage{mathrsfs}
\usepackage{comment}
\usepackage[capitalise]{cleveref}
\crefformat{equation}{(#2#1#3)}
\usepackage{enumerate}
\usepackage{graphicx}
\usepackage[margin=1.5in]{geometry}

\newtheorem{thm}{Theorem}
\newtheorem*{thm*}{Theorem}
\newtheorem*{prp*}{Proposition}
\newtheorem*{oq*}{Open Question 3.10}
\newtheorem{prp}[thm]{Proposition}
\newtheorem{lma}[thm]{Lemma}
\newtheorem{cor}[thm]{Corollary}
\newtheorem*{cor*}{Corollary}
\newtheorem{conj}[thm]{Conjecture}
\newtheorem{clm}{Claim}

\newcommand{\mo}{<_{\textnormal{M}}}

\newcommand{\wo}{\leq_S}
\newcommand{\swo}{<_S}
\newcommand{\lwo}{\geq_S}
\newcommand{\slwo}{>_S}

\newcommand{\I}{\sqsubset}
\newcommand{\D}{\leq_\textnormal{D}}

\theoremstyle{definition}
\newtheorem{defn}[thm]{Definition}
\newtheorem{qst}[thm]{Question}
\newtheorem*{qst*}{Question}
\newtheorem*{defn*}{Definition}
\newtheorem*{ua}{Ultrapower Axiom}

\newtheorem{case}{Case}
\makeatletter\@addtoreset{case}{thm}\makeatother
\makeatletter\@addtoreset{clm}{thm}\makeatother
\makeatletter\@addtoreset{sclm}{thm}\makeatother
\begin{document}
\title{Compactness and Comparison}
\author{Gabriel Goldberg}
\maketitle
\section{Introduction}
Strongly compact cardinals were introduced by Tarski to generalize the fundamental compactness theorem of first-order logic to infinitary logic.

\begin{defn*}
An uncountable cardinal \(\kappa\) is {\it strongly compact} if for any \(\mathcal L_{\kappa,\omega}\) theory \(\Sigma\), if every \({<}\kappa\)-sized subtheory of \(\Sigma\) has a model, then \(\Sigma\) itself has a model.
\end{defn*}

It was soon established that every strongly compact cardinal is measurable. An immediate question is whether the least strongly compact cardinal is strictly larger than the least measurable cardinal. To motivate our work we describe  Solovay's ultimately unsuccessful attack on this problem, Menas's refutation of Solovay's conjecture, and Magidor's remarkable independence result.

In the late 1960s, Solovay initiated a program to answer this question positively. He began by reducing the notion of strong compactness to its pure set theoretic content.

\begin{defn*}
Suppose \(\kappa\) is a cardinal and \(A\) is a set. Then \(P_\kappa(A)\) denotes the collection of subsets of \(A\) of cardinality less than \(\kappa\).

An ultrafilter \(\mathcal U\) on \(P_\kappa(A)\) is {\it fine} if for all \(a\in A\), for \(\mathcal U\)-almost all \(\sigma\in P_\kappa(A)\), \(a\in \sigma\).
\end{defn*}

\begin{thm*}[Solovay]
A cardinal \(\kappa\) is strongly compact if and only if \(P_\kappa(A)\) carries a fine \(\kappa\)-complete ultrafilter for every set \(A\). 
\end{thm*}

Solovay's approach to Tarski's problem was motivated by the well-known fact that if \(\kappa\) carries a \(\kappa\)-complete uniform ultrafilter (or equivalently if \(P_\kappa(\kappa)\) carries a \(\kappa\)-complete fine ultrafilter), then \(\kappa\) carries a {\it normal} ultrafilter. 

Solovay observed that many of the notions of infinitary combinatorics generalize from their classical context on a regular uncountable cardinal \(\kappa\) to a higher order analog on \(P_\kappa(A)\) for any set \(A\). Take for example the notion of a normal ultrafilter:

\begin{defn*}
Suppose \(\kappa\) is a cardinal and \(A\) is a set. A fine ultrafilter \(\mathcal U\) on \(P_\kappa(A)\) is {\it normal} if any function \(f: P_\kappa(A)\to A\) such that \(f(\sigma)\in \sigma\) for \(\mathcal U\)-almost all \(\sigma\) assumes a constant value for \(\mathcal U\)-almost all \(\sigma\). 
\end{defn*}

Given the analogy between \(\kappa\) and \(P_\kappa(A)\), it was natural to conjecture that an arbitrary fine \(\kappa\)-complete ultrafilter on \(P_\kappa(A)\) could be massaged into a {\it normal} fine \(\kappa\)-complete ultrafilter on \(P_\kappa(A)\). Solovay therefore introduced the notion of a supercompact cardinal.

\begin{defn*}
A cardinal \(\kappa\) is {\it supercompact} if \(P_\kappa(A)\) carries a normal fine \(\kappa\)-complete ultrafilter for every set \(A\).  
\end{defn*}

Solovay conjectured that every strongly compact cardinal is supercompact. His conjecture would have implied a positive answer to Tarski's question, since the set of measurable cardinals less than a supercompact cardinal \(\kappa\) is stationary in \(\kappa\). But Solovay's conjecture is false assuming large cardinal hypotheses.

\begin{thm*}[Menas]
Every measurable limit of strongly compact cardinals is strongly compact.
\end{thm*}

Using Menas's theorem, it is easy to see that if \(\kappa\) is the least measurable limit of strongly compact cardinals, the set of measurable cardinals below \(\kappa\) is nonstationary, and therefore \(\kappa\) is not supercompact, even though \(\kappa\) {\it is} strongly compact, again by Menas's theorem.

While this result implies that Solovay's program cannot be carried out naively, the true subtlety of Tarski's question was not understood until Magidor's independence results.

\begin{thm*}[Magidor]
Suppose \(\kappa\) is strongly compact. Then there is a partial order \(\mathbb P\subseteq V_{\kappa+1}\) that preserves the strong compactness of \(\kappa\) while forcing that \(\kappa\) is the least measurable cardinal.
\end{thm*} 

Therefore one cannot prove that the least strongly compact cardinal is strictly larger than the least measurable cardinal. On the other hand, Magidor also showed that one cannot refute Solovay's conjecture at the least strongly compact cardinal. An extension of this result that is relevant to us is due to Kimchi-Magidor:

\begin{thm*}[Kimchi-Magidor]
There is a definable class partial order \(\mathbb P\) that preserves all supercompact cardinals while forcing that every strongly compact cardinal is either supercompact or a measurable limit of supercompact cardinals.
\end{thm*}

In other words, Menas's theorem provides the only provable counterexamples to Solovay's conjecture.

The approach to Tarski's problem and Solovay's conjecture taken here is inspired by Solovay's, except that we are informed by Menas's counterexamples, and we are forced by Magidor's theorems to adopt a new principle.

The Ultrapower Axiom (UA) is a structural principle in the combinatorics of countably complete ultrafilters that holds in all known canonical inner models. If the inner model program reaches an inner model with a supercompact cardinal, then work of Woodin \cite{Woodin} suggests that UA must be consistent with {\it all large cardinal axioms}. 

The axiom itself is simple enough, even though the reasons for believing it to be consistent with a strongly compact cardinal are subtle.

\begin{ua}
For all countably complete ultrafilters \(U\) and \(W\), there exist \(W_*\in M_U\) and \(U_*\in M_W\), countably complete ultrafilters in their respective models, such that \(M^{M_U}_{W_*}= M^{M_W}_{U_*}\) and \(j^{M_U}_{W_*}\circ j_U = j^{M_W}_{U_*}\circ j_W\).
\end{ua}

The Ultrapower Axiom is motivated by the Comparison Lemma of inner model theory. In fact the Comparison Lemma implies UA by a general argument. The Comparison Lemma is the central feature of modern inner model theory, so if one could rule out the Ultrapower Axiom from a supercompact cardinal, one would in fact rule out any sort of inner model theory for supercompact cardinals.

The theory of countably complete ultrafilters under ZFC alone is buried in independence results (for example, see \cite{KunenParis}, \cite{Gitik}, \cite{Magidor}, \cite{MagidorBagaria}, \cite{FriedmanMagidor}), but also contains hints of a deeper underlying structure (\cite{Ketonen}, \cite{Solovay}). Assuming UA, this structure comes to the surface, and there is quite a bit one can prove. We sketch some general facts in the theory of countably complete ultrafilters assuming UA in \cref{BasicTheory}. The main ingredient is a wellfounded partial order on the class of countably complete uniform ultrafilters on ordinals generalizing the Mitchell order. This order is called the {\it seed order} and denoted \(\swo\). The linearity of the seed order is equivalent to UA.

The focus of the rest of the paper is the theory of strong compactness and supercompactness under UA. Assuming UA, there is in fact a generalization of the argument producing a normal ultrafilter on \(\kappa\) from a \(\kappa\)-complete ultrafilter on \(\kappa\) that brings Solovay's ideas described above to fruition. This is the subject of \cref{FirstStronglyCompact}, whose main theorem is below:

\begin{thm*}[UA]
Suppose \(\delta\) is a regular cardinal and \(\kappa \leq \delta\) is the least \(\delta\)-strongly compact cardinal. Then the \(\swo\)-least uniform countably complete ultrafilter on \(\delta\) witnesses that \(\kappa\) is \({<}\delta\)-supercompact.
\end{thm*}

If \(\delta\) is not strongly inaccessible then in fact the \(\swo\)-least ultrafilter \(U\) must be fully \(\delta\)-supercompact. An interesting question is whether this holds in general. We conjecture that the answer is no. This discussed in \cref{LastWord}.

In particular, we have the following theorem:

\begin{thm*}[UA]
The least strongly compact cardinal is supercompact.
\end{thm*}

What about the other strongly compact cardinals? This is the subject of \cref{SecondStronglyCompact} whose main result is that under the Ultrapower Axiom, the Kimchi-Magidor consistency result becomes a theorem.

\begin{thm*}[UA]
Every strongly compact cardinal is either supercompact or a measurable limit of supercompact cardinals.
\end{thm*}

The full characterization of strongly compact cardinals above cannot be proved without first characterizing the least one by a completely different argument. One essentially propagates the supercompactness of the first strongly compact cardinal to the others. This dynamic hints at the special nature of the first supercompact cardinal in inner model theory, identified first by Woodin in \cite{SEM} and \cite{Woodin}.

\section{Preliminaries}\label{Preliminaries}
\subsection{Notation and Conventions}
We set up some conventions for discussing countably complete ultrafilters.

\begin{defn}
An ultrafilter \(U\) on an ordinal \(\delta\) is {\it uniform} if for all \(\xi < \delta\), \(\delta\setminus \xi\in U\). The set of uniform countably complete ultrafilters on \(\delta\) is denoted by \(\textnormal{Un}_\delta\). The class of uniform countably complete ultrafilters is denoted by \(\textnormal{Un}\).
\end{defn}

\begin{defn}
For \(U\in \textnormal{Un}\), we denote by \(\textsc{sp}(U)\) the {\it space} of \(U\), which is unique ordinal \(\delta\) such that \(U\in \textnormal{Un}_\delta\).
\end{defn}

Equivalently, \(\textsc{sp}(U)\) is the unique \(\delta\) such that \(\delta\in U\).

In the case that \(\delta\) is regular, the notion of uniformity defined above is the usual one. In the case that \(\delta\) is a singular cardinal, there are two definitions of uniformity in use. We have chosen the weaker one, but we occasionally use the stronger one.

\begin{defn}
An ultrafilter \(U\) on a set \(X\) is {\it strongly uniform} if for all \(A\in U\), \(|A| = |X|\).
\end{defn}

According to our official definition of uniformity, a principal ultrafilter can be uniform.

\begin{defn}
For any ordinal \(\alpha\), \(P_\alpha\) denotes the {\it uniform principal ultrafilter concentrated at \(\alpha\)}.
\end{defn}

Thus \(\textsc{sp}(P_\alpha) = \alpha + 1\).

We use standard notation for ultrapowers.

\begin{defn}
If \(U\) is an ultrafilter, then \(j_U: V\to M_U\) denotes the ultrapower of \(V\) by \(U\).

More generally, if \(N\) is an inner model and \(U\) is an \(N\)-ultrafilter, then \(j^N_U : N \to M_U^N\) denotes the ultrapower of \(N\) by \(U\) using only functions in \(N\).
\end{defn}

We will use the notation \(j^N_U\) and \(M^N_U\) even when \(U\notin N\).

Every ultrapower considered in this paper will be wellfounded and therefore identified with its transitive collapse.

We will be interested in limits of ultrafilters, occasionally in a slightly more general sense than the usual one.

\begin{defn}\label{Limit}
Suppose \(U\) is an ultrafilter and \(W_*\) is an \(M_U\)-ultrafilter on \(X_*\). Suppose \(X\) is such that \(X_*\subseteq j_U(X)\). The {\it \(U\)-limit of \(W_*\) on \(X\)} is the ultrafilter 
\[U^-(W_*, X) = \{A\subseteq X : j_U(A)\cap X_*\in W_*\}\]
\end{defn}

The main novelty of this definition is that we do not require \(W_*\in M_U\). This is useful in the proof of \cref{IrredGCH}.

We single out two special cases of \cref{Limit} in which there is a canonical choice for the underlying set of the limit ultrafilter. These are the only cases we will really consider in this paper.

\begin{defn}
Suppose \(U\) is an ultrafilter and \(W_*\) is a uniform \(M_U\)-ultrafilter on an ordinal \(\delta_*\). Then the {\it \(U\)-limit of \(W_*\)}, denoted \(U^-(W_*)\), is the \(U\)-limit of \(W_*\) on \(\delta\) where \(\delta\) is least such that \(\delta_*\leq j_U(\delta)\).
\end{defn}

Note that \(\delta\) is chosen so that \(U^-(W_*)\) is uniform. 

The second special case of \cref{Limit} generalizes the first one, but we will use it much less (only in the proof of \cref{IrredGCH}).

\begin{defn}
Suppose \(U\) is an ultrafilter and \(\mathcal W_*\) is a fine \(M_U\)-ultrafilter on \(P^{M_U}_{\delta_*}(\delta_*)\). Then the {\it \(U\)-limit of \(\mathcal W_*\)}, denoted \(U^-(\mathcal W_*)\), is the \(U\)-limit of \(\mathcal W_*\) on \(P_\delta(\delta)\) where \(\delta\) is least such that \(\delta_*\leq j_U(\delta)\).
\end{defn}

Note that \(\delta\) is chosen so that \(U^-(\mathcal W_*)\) is a fine ultrafilter on \(P_{\delta}(\delta)\).

It is usually easier to think about limits in terms of ultrapower embeddings, which is possible by the following lemma.

\begin{lma}\label{LimitEmbeddings}
Suppose \(U,W\in \textnormal{Un}\) and \(U_*\in \textnormal{Un}^{M_W}\). Then \(U = W^-(U_*)\) if and only if there is an elementary embedding \(k : M_U\to M^{M_W}_{U_*}\) such that \(k\circ j_U = j^{M_W}_{U_*}\circ j_W\) and \(k([\textnormal{id}]_U) = [\textnormal{id}]_{U_*}\).
\end{lma}
\subsection{The Seed Order}
The key to this work is a new order on the class of uniform countably complete ultrafilters.
\begin{defn}
The {\it seed order} is defined on \(U,W\in \textnormal{Un}\) by setting \(U\swo W\) if there is some \(U_*\in \textnormal{Un}^{M_W}\) with \(\textsc{sp}(U_*)\leq [\text{id}]_W\) such that \(U = W^-(U_*)\).
\end{defn}

There is a useful characterization of the seed order in terms of elementary embeddings using \cref{LimitEmbeddings}.

\begin{cor}\label{SeedOrderChar}
Suppose \(U,W\in \textnormal{Un}\). Then \(U \swo W\) if and only if there is some \(U_*\in\textnormal{Un}^{M_W}\) and an elementary embedding \(k : M_U\to M^{M_W}_{U_*}\) with \(k\circ j_U = j^{M_W}_{U_*}\circ j_W\) and \(k([\textnormal{id}]_U) < j^{M_W}_{U_*}([\textnormal{id}]_W)\).
\end{cor}

When \(U\) is an ultrafilter, the point \([\text{id}]_U\) is sometimes called the {\it seed} of \(U\), which along with \cref{SeedOrderChar} should explain the name ``seed order." 

A more detailed exposition of the seed order will appear elsewhere.

The following technical lemma allows the structure of the seed order in \(V\) to be copied into its ultrapowers.
\begin{lma}\label{TechnicalLemma}
Suppose \(U,W,Z\in \textnormal{Un}\) and \(U\swo W\). Suppose \(W_*\in \textnormal{Un}^{M_Z}\) is such that \(Z^-(W_*) = W\). Then there is some \(U_*\swo^{M_Z} W_*\) with \(Z^-(U_*) = U\).
\end{lma}

We state without proof a very basic fact characterizing the simple relationship between the space of an ultrafilter and the seed order.

\begin{lma}\label{SpaceLemma}
Suppose \(U,W\in \textnormal{Un}\). If \(\textsc{sp}(U) < \textsc{sp}(W)\) then \(U\swo W\).
\end{lma}

As an easy corollary of \cref{TechnicalLemma} and \cref{SpaceLemma} one can prove the main structural fact about the seed order.

\begin{thm}
The seed order is a set-like wellfounded strict partial order.
\end{thm}

\subsection{The Ultrapower Axiom}
The Ultrapower Axiom is an abstract comparison principle motivated by the comparison process of inner model theory.

\begin{ua}
For any countably complete ultrafilters \(U\) and \(W\), there exist countably complete ultrafilters \(W_*\) and \(U_*\) of \(M_U\) and \(M_W\) respectively such that the following hold:
\begin{align*}
M^{M_U}_{W_*} &= M^{M_W}_{U_*}\\
j^{M_U}_{W_*}\circ j_U &= j^{M_W}_{U_*}\circ j_W 
\end{align*}
\end{ua}

The Ultrapower Axiom holds in all known canonical inner models and is expected to hold in canonical inner models with supercompact cardinals if such models exist.

The key consequence of the Ultrapower Axiom is the linearity of the seed order, which is an immediate consequence of \cref{SeedOrderChar}. Moreover the characterization of the seed order in terms of elementary embeddings becomes fully symmetric:

\begin{prp}[UA]\label{Linearity}
The seed order is linear. In fact if \(U,W\in \textnormal{Un}\), then \(U \swo W\) if and only if there are countably complete ultrafilters \(W_*\) and \(U_*\) of \(M_U\) and \(M_W\) respectively such that the following hold:
\begin{align*}
M^{M_U}_{W_*} &= M^{M_W}_{U_*}\\
j^{M_U}_{W_*}\circ j_U &= j^{M_W}_{U_*}\circ j_W\\
j^{M_U}_{W_*}([\textnormal{id}]_U) &< j^{M_W}_{U_*}([\textnormal{id}]_W) 
\end{align*}
\end{prp}
By \cref{Commutativity} below, one can further simplify the definition of the seed order under UA by removing the commutativity requirement.
\section{Ultrafilter Theory under UA}\label{BasicTheory}
\subsection{Reciprocity}
In this section we prove the converse of \cref{Linearity}: the linearity of the seed order implies the Ultrapower Axiom. This shows the equivalence between comparison for ultrafilters on the one hand and a combinatorial generalization of the linearity of the Mitchell order on the other.

The proof introduces the very useful concept of a {\it translation function}.

\begin{defn}
Assume the seed order is linear. We associate to each countably complete ultrafilter \(U\) a {\it translation function} \(t_U: \textnormal{Un}\to \textnormal{Un}^{M_U}\) as follows:

For any \(W\in \textnormal{Un}\), \(t_U(W)\) denotes the \(\swo^{M_U}\)-least \(W_*\in \textnormal{Un}^{M_U}\) such that \(U^-(W_*) = W\).
\end{defn}

It is convenient to define an operation \(\oplus\) with the property that for any \(U\in \textnormal{Un}\) and \(W\in \textnormal{Un}^{M_W}\), \(U\oplus W\in \textnormal{Un}\) and \(j_{U\oplus W} = j^{M_U}_W\circ j_U\). (The usual ultrafilter sum operation does not have range contained in \(\textnormal{Un}\).) There are various ways in which one could do this, and our choice is motivated mostly by the desire that this operation work smoothly with the seed order; see for example \cref{SumSeed}.

\begin{defn}
For \(\alpha,\beta\in \text{Ord}\), \(\alpha\oplus \beta\) denotes the {\it natural sum} of \(\alpha\) and \(\beta\), which is obtained as follows:

First write \(\alpha\) and \(\beta\) in Cantor normal form: \begin{align*}\alpha &= \sum_{\xi\in \text{Ord}} \omega^{\xi}\cdot m_\xi\\ \beta &= \sum_{\xi\in \text{Ord}} \omega^{\xi}\cdot n_\xi\end{align*} where \(m_\xi,n_\xi < \omega\) are equal to \(0\) for all but finitely many \(\xi \in \text{Ord}\). Then 
\begin{align*}\alpha \oplus \beta &= \sum_{\xi\in \text{Ord}} \omega^{\xi}\cdot (m_\xi + n_\xi)\end{align*}
\end{defn}

In other words one adds the Cantor normal forms of \(\alpha\) and \(\beta\) as polynomials.

The fact that natural addition is commutative and associative follows easily from the corresponding facts for addition of natural numbers. We mostly need the following triviality:
\begin{lma}\label{OrderSum}
If \(\alpha_0< \alpha_1\) and \(\beta\) are ordinals, then \(\alpha_0 \oplus \beta < \alpha_1 \oplus \beta\).
\end{lma}

\begin{defn}
If \(U\in \textnormal{Un}\) and \(W_*\in \textnormal{Un}^{M_U}\) then the {\it natural sum of \(U\) and \(W_*\)}, denoted \(U\oplus W_*\), is the uniform ultrafilter derived from \(j^{M_U}_{W_*}\circ j_U\) using \([\text{id}]^{M_U}_{W_*}\oplus j^{M_U}_{W_*}([\text{id}]_U)\). 
\end{defn}

The next lemma says that the natural sum of ultrafilters is Rudin-Keisler equivalent to the usual sum of ultrafilters.

\begin{lma}\label{UltrafilterSum} For any \(U\in \textnormal{Un}\) and \(W_*\in \textnormal{Un}^{M_U}\), \(M_{U\oplus W_*} = M^{M_U}_{W_*}\) and \(j_{U\oplus W_*} = j^{M_U}_{W_*}\circ j_U\).\end{lma}

Natural sums also interact quite simply with the seed order:

\begin{lma}\label{SumSeed}
Suppose \(U\in \textnormal{Un}\). Suppose \(W_0,W_1\in \textnormal{Un}^{M_U}\). Then \(W_0\swo^{M_U} W_1\) if and only if \(U\oplus W_0 \swo U\oplus W_1\).
\end{lma}

We will use the following theorem whose proof we defer to another paper.

\begin{thm}[Minimality of Definable Embeddings]\label{mindef}
Suppose \(M\) and \(N\) are inner models and \(j,i : M\to N\) are elementary embeddings. If \(j\) is definable from parameters over \(M\), then \(j(\alpha) \leq i(\alpha)\) for all ordinals \(\alpha\).
\end{thm}

The proof generalizes the Dodd-Jensen Lemma from inner model theory. It is inspired by a similar theorem in \cite{Woodin}.

\begin{thm}[Reciprocity Theorem]\label{Reciprocity}
Assume the seed order is linear. Then for any uniform countably complete ultrafilters \(U\) and \(W\), \[U\oplus t_U(W) = W\oplus t_W(U)\]
\begin{proof}
Assume towards a contradiction that \(U\oplus t_U(W) \swo W \oplus t_W(U)\). 

By \cref{LimitEmbeddings}, there is an inner model \(N\) admitting an elementary embedding \(k : M_{U\oplus t_U(W)}\to N\) and an ultrapower embedding \(i: M_{W \oplus t_W(U)}\to N\) such that \(k\circ j_{U\oplus t_U(W)} = i \circ j_{W \oplus t_W(U)}\) and \[k\left([\text{id}]_{U\oplus t_U(W)}\right) < i\left([\text{id}]_{W \oplus t_W(U)}\right)\]
In other words,
\begin{equation}\label{WrongOrder}k\left([\text{id}]^{M_U}_{t_U(W)}\oplus j^{M_U}_{t_U(W)}([\text{id}]_U)\right) < i\left([\text{id}]^{M_W}_{t_W(U)}\oplus j^{M_W}_{t_W(U)}([\text{id}]_W)\right)\end{equation}
\begin{clm}\label{MinimalityTranslates}
\(i\left([\textnormal{id}]^{M_W}_{t_W(U)}\right) \leq k\left(j^{M_U}_{t_U(W)}([\textnormal{id}]_U)\right) \)
\end{clm}
\begin{clm}\label{MinimalityDefinable}
\(i\left(j^{M_W}_{t_W(U)}([\textnormal{id}]_W)\right) \leq k\left([\textnormal{id}]^{M_U}_{t_U(W)}\right)\)
\end{clm}
Using \cref{OrderSum}, these two claims contradict \cref{WrongOrder}, so the assumption that \(U\oplus t_U(W) \swo W \oplus t_W(U)\) was false.
\begin{proof}[Proof of \cref{MinimalityTranslates}]
Let \(U_*\) be the \(M_W\)-ultrafilter derived from \(i\circ j^{M_W}_{t_W(U)}\) using \(k\left(j^{M_U}_{t_U(W)}([\textnormal{id}]_U)\right)\). Let \(h : M^{M_W}_{U_*}\to N\) be the factor embedding. Note that \(W^-(U_*) = U\): this is an easy calculation using \cref{LimitEmbeddings}, noting that there is an elementary embedding \(M_U\to M^{M_W}_{U_*}\) witnessing the hypotheses of \cref{LimitEmbeddings}, namely \(h^{-1}\circ k\circ j^{M_U}_{t_U(W)}\).

If \(h([\text{id}]^{M_W}_{U_*}) < i([\text{id}]^{M_W}_{t_W(U)})\), then \(U_*\swo^{M_W} t_W(U)\) by \cref{SeedOrderChar}, contrary to the minimality of \(t_W(U)\). Thus \(i([\text{id}]^{M_W}_{t_W(U)})\leq  h([\text{id}]^{M_W}_{U_*}) = k(j^{M_U}_{t_U(W)}([\textnormal{id}]_U))\), as desired.
\end{proof}
\begin{proof}[Proof of \cref{MinimalityDefinable}]
Let \(h : M_W\to M^{M_U}_{t_U(W)}\) be the elementary embedding given by \cref{LimitEmbeddings}. Then \(k\circ h([\text{id}]_W) = k([\textnormal{id}]^{M_U}_{t_U(W)})\). Since \(i\circ  j^{M_W}_{t_W(U)}\) and \(k\circ h\) are elementary embeddings \(M_W\to N\), and since \(i\circ  j^{M_W}_{t_W(U)}\) is definable from parameters over \(M_W\), \[i\circ j^{M_W}_{t_W(U)}\restriction \text{Ord} \leq k\circ h\restriction \text{Ord}\] by \cref{mindef}. In particular, \(i\circ j^{M_W}_{t_W(U)}([\text{id}]_W)\leq k\circ h([\text{id}]_W) = k([\textnormal{id}]^{M_U}_{t_U(W)})\).
\end{proof}
Similarly we cannot have \(W \oplus t_W(U)\swo U\oplus t_U(W)\). By the linearity of the seed order, \(U\oplus t_U(W) = W\oplus t_W(U)\), finishing the proof of the Reciprocity Theorem.
\end{proof}
\end{thm}
An immediate corollary of  \cref{UltrafilterSum} and \cref{Reciprocity} is the following:
\begin{cor}
The following are equivalent:
\begin{enumerate}[(1)]
\item The seed order is linear.
\item The Ultrapower Axiom holds.
\end{enumerate}
\end{cor}

Regarding the translation functions \(t_U\), we also have:
\begin{prp}[UA]\label{op}
For any countably complete ultrafilter \(Z\), the function \(t_Z : (\textnormal{Un},\swo)\to (\textnormal{Un}^{M_Z},\swo^{M_Z})\) is order preserving.
\begin{proof}
Suppose \(U, W\in \textnormal{Un}\) and \(U\swo W\). By \cref{TechnicalLemma}, as \(Z^-(t_Z(W)) = W\), there is some \[U_*\swo^{M_Z}t_Z(W)\] such that \(Z^-(U_*) = U\). By the minimality of \(t_Z(U)\), we have \[t_Z(U)\wo^{M_Z} U_*\] By the transitivity of the seed order, \(t_Z(U)\swo^{M_Z} t_Z(W)\), as desired.
\end{proof}
\end{prp}

\subsection{Divisibility}
\begin{defn}
Suppose \(U\) and \(W\) are countably complete ultrafilters. We say \(U\) {\it divides} \(W\) and write \(U\D W\) if there is some \(W_*\in \textnormal{Un}^{M_U}\) such that \(M^{M_U}_{W_*} = M_W\) and \(j^{M_U}_{W_*}\circ j_U = j_W\).
\end{defn}

The division order is sometimes called the Rudin-Frolik order, although that name is often reserved for a suborder of the division order. We continue to work with uniform ultrafilters on ordinals here, but notice that many of the facts we prove in this section are invariant under Rudin-Keisler equivalence.

The division order has a number of obvious combinatorial equivalents, one of which we put down below.

\begin{lma}
Suppose \(U,W\in \textnormal{Un}\). Then \(U\) divides \(W\) if and only if there is some \(W_*\in \textnormal{Un}^{M_U}\) such that \(U\oplus W_*\equiv W\).
\end{lma}

The factor ultrafilter \(W_*\) is unique up to equivalence:

\begin{lma}\label{DivisionUniqueness}
Suppose \(U, W\in \textnormal{Un}\). Then there is at most one internal embedding \(i : M_U\to M_W\) such that \(i\circ j_U = j_W\).
\begin{proof}
We may assume \(U\in \textnormal{Un}\). Such an embedding \(i\) is determined by its values on \(j_U[V]\cup \{[\text{id}]_U\}\). Any two such embeddings agree on \(j_U[V]\) by the requirement \(i \circ j_U = j_W\). Moreover they agree on \(\{[\text{id}]_U\}\) by \cref{mindef}.
\end{proof}
\end{lma}

Similarly, we have an absoluteness fact for division.

\begin{lma}\label{DivisionAbsoluteness}
Suppose \(U\in \textnormal{Un}\) and \(D_*,W_*\in \textnormal{Un}^{M_U}\). Then \(D_*\D^{M_U} W_*\) if and only if \(U\oplus D_*\D U\oplus W_*\).
\begin{proof}
The only nontrivial part is proving that if \(i : M_{U\oplus D_*}\to M_{U\oplus W_*}\) is an internal ultrapower embedding with \(i\circ j_{U\oplus D_*} = j_{U\oplus W_*}\) then \(i\circ j^{M_U}_{D_*} = j^{M_U}_{W_*}\). It suffices to show that \(i(j^{M_U}_{D_*}([\text{id}]_U)) = j^{M_U}_{W_*}([\text{id}]_U)\). This follows from \cref{mindef}.
\end{proof}
\end{lma}

Under UA, the division order is closely related to the translation functions of the previous subsection.

\begin{defn}[UA]
For \(U,W\in \textnormal{Un}\), let \(U\vee W\) denote \(U \oplus t_U(W) = W\oplus t_W(U)\).
\end{defn}

We chose this notation because \(U\vee W\) is the least upper bound of \(U\) and \(W\) in the division order.

\begin{thm}[UA]\label{CanonicalComparison}
Suppose \(U,W,Z\in\textnormal{Un}\). Then \(U,W\D Z\) if and only if \(U\vee W\D Z\).
\begin{proof}
For ease of notation let \(D = U\vee Z\). 

We claim that \(t_Z(D)\) is principal in \(M_Z\). To see this, it suffices to show that \([\text{id}]^{M_Z}_{t_Z(D)}\in\text{ran }j^{M_Z}_{t_Z(D)}\).
Let \(i_{UZ} : M_U\to M_Z\) and \(i_{WZ} : M_W\to M_Z\) be internal ultrapower embeddings witnessing that \(U,W\) divide \(Z\). Note that \[j^{M_Z}_{t_Z(D)}\circ i_{UZ}\restriction \text{Ord} =  j^{M_{D}}_{t_{D}(Z)}\circ j^{M_U}_{t_U(W)}\restriction \text{Ord}\] by \cref{mindef}. Hence \[j^{M_{D}}_{t_{D}(Z)}(j^{M_U}_{t_U(W)}([\text{id}]_U))\in\text{ran }j^{M_Z}_{t_Z(D)}\] Similarly, \[j^{M_{D}}_{t_{D}(Z)}(j^{M_W}_{t_W(U)}([\text{id}]_W))\in\text{ran }j^{M_Z}_{t_Z(D)}\]
By \cref{Reciprocity}, 
\begin{align*}[\text{id}]^{M_Z}_{t_Z(D)} &= j^{M_{D}}_{t_{D}(Z)}([\text{id}]_{D}) \\
&= j^{M_{D}}_{t_{D}(Z)}\left(j^{M_U}_{t_U(W)}([\text{id}]_U) \oplus j^{M_W}_{t_W(U)}([\text{id}]_W)\right)
\end{align*}
Hence \([\text{id}]^{M_Z}_{t_Z(D)}\in\text{ran }j^{M_Z}_{t_Z(D)}\), so  \(t_Z(D)\) is principal in \(M_Z\), as claimed.

It follows by \cref{Reciprocity} that \[D\oplus t_{D}(Z) = Z\oplus t_Z(D)\equiv Z\] Hence \(D\D Z\) as desired.
\end{proof}
\end{thm}

Similarly we can characterize the \(M_{U\vee W}\)-ultrafilters that belong to \(M_{U\vee W}\):

\begin{cor}[UA]\label{CanonicalInternal}
Suppose \(U,W\in \textnormal{Un}\) and \(F\) is a countably complete \(M_{U\vee W}\)-ultrafilter. Then \(F\in M_{U\vee W}\) if and only if \(F\in M_U\cap M_W\).
\begin{proof}
For ease of notation let \(D = U\vee Z\). We may assume without loss of generality \(F\) is \(M_D\)-uniform on some ordinal.

Clearly \(F\in M_D\) implies \(F\in M_U\cap M_W\). 

Conversely assume \(F\in M_U\cap M_W\). Then \(U\) and \(W\) divide \(D\oplus F\) via \(j^{M_{D}}_F\circ j^{M_U}_{t_U(W)}\) and \(j^{M_{D}}_F\circ j^{M_W}_{t_W(U)}\). Hence \(D\) divides \(D\oplus F\). 

Let \[i : M_{D}\to M^{M_{D}}_F\] be an internal ultrapower embedding witnessing that \(D\) divides \(D\oplus F\). We must verify that \(i = j^{M_{D}}_F\). Since \(i\circ j_{D} = j^{M_{D}}_F\circ j_{D}\) by the definition of division, it suffices to show that \(i([\text{id}]_{D}) = j^{M_{D}}_F([\text{id}]_{D})\). For this it is enough to show that 
\begin{align*}i(j^{M_U}_{t_U(W)}([\text{id}]_U)) &= j^{M_{D}}_F(j^{M_U}_{t_U(W)}([\text{id}]_U))\\ 
i(j^{M_W}_{t_W(U)}([\text{id}]_W)) &= j^{M_{D}}_F(j^{M_W}_{t_W(U)}([\text{id}]_W))\end{align*}
This follows immediately from the uniqueness of definable embeddings on the ordinals, \cref{mindef}.
\end{proof}
\end{cor}

This has the following useful corollary.

\begin{thm}[UA]\label{DivisionCharacterization}
For \(U,W\in \textnormal{Un}\), the following are equivalent:
\begin{enumerate}[(1)]
\item \(U\) divides \(W\).
\item \(M_W\subseteq M_U\).
\item \(t_W(U)\in M_U\).
\item For some \(U_*\equiv^{M_W} t_W(U)\), \(U_*\in M_U\).
\item \(t_W(U)\) is principal.
\end{enumerate}
\begin{proof}
This is arranged as a round-robbin proof, and the only nontrivial implication is from (4) to (5). 

This is a generalization of the proof that a nonprincipal ultrafilter does not belong to its target model. Suppose \(U_*\in M_U\) for some \(U_*\equiv^{M_W} t_W(U)\). Let \(\delta = \textsc{sp}(U_*)\). Then \(j^{M_W}_{t_W(U)}\restriction \delta \in M_U\) since it can be computed in any inner model containing \(U_*\). Let \(N = M_{U\vee W}\). Consider the class \(C\) of sequences of ordinals \(s\) of the form \(f\circ (j^{M_W}_{t_W(U)}\restriction \delta)\) for some \(f\in N\). Clearly \(C\) is a definable class of both \(M_U\) and \(M_W\). 

We claim that \(C = \text{Ord}^\delta \cap M_W\). To see this, fix \(s\in \text{Ord}^\delta\cap M_W\), and we will show \(s\in C\). Working in \(M_W\), choose \(\langle g_\alpha : \alpha < \delta\rangle\) such that \(s_\alpha = [g_\alpha]_{t_W(U)}\) for all \(\alpha < \delta\). Then let \(\langle h_\alpha : \alpha < j^{M_W}_{t_W(U)}(\delta)\rangle = j_{t_W(U)}^{M_W}(\langle g_\alpha : \alpha < \delta\rangle)\), and let \(f = \langle h_\alpha([\text{id}]_{t_W(U)}) : \alpha < j^{M_W}_{t_W(U)}(\delta)\rangle\). Easily \(f \circ (j^{M_W}_{t_W(U)}\restriction \delta) = s\).

Since \(C\) is a definable class of \(M_U\), one can form in \(M_U\) the ultrapower of \(\text{Ord}\) by \(U_*\) using only functions from \(C\). The ultrapower embedding is precisely \(j^{M_W}_{t_W(U)}\restriction \text{Ord}\). But now repeating the argument of the previous paragraph, it follows that \(P(\text{Ord})\cap M_W\) is a definable subclass of \(M_U\), and hence \(M_W = L(P(\text{Ord})\cap M_W)\) is a definable subclass of \(M_U\).

In particular, \(j^{M_W}_{t_W(U)} : M_W\to N\) is definable over \(M_U\). Let \(i = j^{M_W}_{t_W(U)}\restriction N\). Then \(i\) is definable both over \(M_U\) and over \(M_W\). Since \(i\) is induced by a countably complete \(N\)-ultrafilter (see \cref{PushforwardLemma}), using \cref{CanonicalInternal}, it follows that \(i\) is an internal ultrapower embedding of \(N\).

Note however that \(i\) and \(j^{M_W}_{t_W(U)}(j^{M_W}_{t_W(U)})\) are definable elementary embeddings of \(N\) with the same target model. Hence they agree on the ordinals by \cref{mindef}. Assuming towards a contradiction that \(t_W(U)\) is nonprincipal with completeness \(\kappa\), we have \[\textsc{crt}(i) = \kappa < j^{M_W}_{t_W(U)}(\kappa) = \textsc{crt}(j^{M_W}_{t_W(U)}(j^{M_W}_{t_W(U)}))\]
a contradiction. Therefore \(t_W(U)\) is principal, as desired.
\end{proof}
\end{thm}

The following corollary has the psychological benefit that we never have to check that diagrams of internal ultrapowers commute. 

\begin{cor}[UA]\label{Commutativity}
Suppose \(U\) and \(U'\) are countably complete ultrafilters such that \(M_U = M_{U'}\). Then \(j_U = j_{U'}\).
\end{cor}

As a corollary of \cref{CanonicalComparison}, we have another reciprocity result:

\begin{cor}[UA]\label{Reciprocity2}
Suppose \(U,W,D\in \textnormal{Un}\). Then \(t_U(D)\) divides \(t_U(W)\) in \(M_U\) if and only if \(t_W(D)\) divides \(t_W(U)\) in \(M_W\).
\begin{proof}
This is a calculation using \cref{CanonicalComparison}:
\begin{align*}
t_U(D)\text{ divides }t_U(W)\text{ in }M_U&\iff U\vee D\text{ divides }U\vee W\\
&\iff D\text{ divides }U\vee W\\
&\iff W\vee D\text{ divides }U \vee W\\
&\iff t_W(D)\text{ divides }t_W(U)\text{ in } M_W
\end{align*}
In the second and third equivalences we use \cref{CanonicalComparison}.
\end{proof}
\end{cor}

We can also prove that translation functions preserve division.
\begin{thm}[UA]
Suppose \(U\) is a countably complete ultrafilter and \(D\D W\) are uniform ultrafilters. Then \(t_U(D)\) divides \(t_U(W)\) in \(M_U\).
\begin{proof}
Note that \(D\D W \D U \vee W\) so \(D\D U\vee W\). Therefore by \cref{CanonicalComparison}, \(U\vee D\D U\vee W\). This implies \(t_U(D)\D^{M_U} t_U(W)\) by \cref{DivisionAbsoluteness}.
\end{proof}
\end{thm}

As an immediate corollary, Rudin-Keisler equivalent ultrafilters translate to Rudin-Keisler equivalent ultrafilters.

\begin{cor}[UA]
If \(U\equiv U'\) then \(t_W(U)\equiv^{M_W} t_W(U')\).
\end{cor}

Actually one does not need much machinery to prove the preceding corollary, and it is essentially provable without UA: one can show in ZFC that if \(U_*\) is minimal such that \(W^-(U_*) = U\) and \(f : \textsc{sp}(U)\to \text{Ord}\) is one-to-one on a set in \(U\), then setting \(U' = f_*(U)\) and \(U_*' = j_W(f)_*(U_*)\), \(U_*'\) is minimal such that \(W^-(U_*') = U'\).

\subsection{The Internal Relation}
In this subsection we define a version of the generalized Mitchell order called the internal relation that is compatible with the abstract techniques we have developed so far. The analysis of the internal relation (and similar notions) under UA is instrumental in the analysis of supercompactness. 

Once the supercompactness analysis is carried out, however, we will be able to characterize the precise relationship between the internal relation and generalized Mitchell order assuming UA + GCH; they are essentially interdefinable. The details appear in \cref{LastWord}. One can therefore view the internal relation as no more than a transitory definition aiding in the analysis of the Mitchell order. If one is interested in the ZFC theory, this view probably does not hold up.

\begin{defn}
The internal relation is defined on countably complete ultrafilters \(U\) and \(W\) by setting \(U\I W\) if and only if \(j_U\restriction M_W\) is an amenable class of \(M_W\).
\end{defn}

To help the reader get his or her bearings, we include some immediate observations regarding the internal relation.

\begin{prp}
If \(W\) is \(\delta\)-supercompact and \(U\in \textnormal{Un}_{\leq\delta}\) then \(U\I W\) if and only if \(U\mo W\).
\end{prp}

Unlike the generalized Mitchell order, however, assuming there are two measurable cardinals \(\kappa_0 < \kappa_1\), the internal relation is neither strict or transitive on nonprincipal ultrafilters (which is why it is not called the internal {\it order}). To see this, note that while the internal relation {\it is} irreflexive on nonprincipal ultrafilters, there exist pairs of ultrafilters \(U,W\in \textnormal{Un}\) with \(U\I W\) and \(W\I U\): by the following theorem, any \(\kappa_0\)-complete ultrafilter on \(\kappa_0\) and \(\kappa_1\)-complete ultrafilter on \(\kappa_1\) furnish an example.

\begin{thm}[Kunen]\label{CommutingUltrapowers}
Suppose \(U,W\in \textnormal{Un}\) satisfy \(\textsc{sp}(U) < \textsc{crt}(W)\). Then \(j^{M_U}_{j_U(W)} = j_W\restriction M_U\) and  \(j^{M_W}_{j_W(U)} = j_U\restriction M_W\). Therefore \(U\I W\) and \(W\I U\).
\end{thm}

The conclusion of \cref{CommutingUltrapowers} is often abbreviated by saying ``\(U\) and \(W\) commute," since in particular it implies \(j_U\circ j_W = j_W\circ j_U\). The results \cref{InternalCommutativity} and and more powerfully those of \cref{LastWord} argue that \cref{CommutingUltrapowers} is essentially the only way in which the internal relation fails to be strict.

On the other hand, restricted to \(\textnormal{Un}_\delta\) for a fixed \(\delta\), the internal relation is strict and indeed wellfounded. In fact the seed order extends the internal relation on \(\textnormal{Un}_\delta\):

\begin{prp}\label{SeedExtendInternal}
For any ordinal \(\delta\), the seed order extends the internal relation on \(\textnormal{Un}_\delta\).
\end{prp}

We give a proof right after \cref{SInternal}. It is not hard to prove \cref{SeedExtendInternal} directly, but we are about to introduce notation that makes it transparent.

We introduce an ultrafilter \(s_W(U)\) with the property that \(U\I W\) if and only if \(s_W(U)\in M_W\). (The three functions \(j_W, t_W,s_W\) are right inverse to the operation \(W^-\).)

\begin{defn}
Suppose \(U\in \textnormal{Un}_\delta\) and \(W\) is a countably complete ultrafilter. Then the {\it pushforward of \(U\) by \(j_W\) restricted to \(M_W\)} is the \(M_U\)-ultrafilter \(s_W(U)\) defined by \[s_W(U) = \{A\in P^{M_W}(\sup j_W[\delta]) : j_W^{-1}[A]\in U\}\]
\end{defn}

One could easily define \(s_W(U)\) for an arbitrary ultrafilter \(U\), but we have no need for this here.

\begin{lma}\label{DerivedPushforward}
If \(U\in \textnormal{Un}\) and \(W\in \textnormal{Un}\), then \(s_W(U)\) is the \(M_W\)-uniform ultrafilter derived from \(j_U\restriction M_W\) using \(j^{M_U}_{j_U(W)}([\textnormal{id}]_U)\).
\begin{proof}
For \(A\in P^{M_W}(j_W(X))\),
\begin{align*}
A\in s_W(U)&\iff j_W^{-1}[A]\in U\\
&\iff \{x\in X : j_W(x)\in A\} \in U\\
&\iff j^{M_U}_{j_U(W)}([\text{id}]_U)\in j_U(A)
\end{align*}
with the last equivalence following from Los's Theorem.
\end{proof}
\end{lma}

\begin{prp}\label{PushforwardLemma}
For any \(U,W\in \textnormal{Un}\), \(j^{M_W}_{s_W(U)} = j_U\restriction M_W\).
\begin{proof}
By the previous theorem, there is an embedding \(k : M^{M_W}_{s_W(U)} \to j_U(M_W)\) such that \(k([\text{id}]^{M_W}_{s_W(U)}) = j_U(j_W)([\text{id}]_U)\) and \(k\circ j^{M_W}_{s_W(U)}  = j_U\restriction M_W\). It suffices to show that \(k\) is surjective. Note that \(k[M^{M_W}_{s_W(U)}]\) contains \(j_U\circ j_W[V] = j_U(j_W)\circ j_U[V]\) as well as \(j_U(j_W)([\text{id}]_U)\). Hence \[j_U(j_W)[M_U]\subseteq k[M^{M_W}_{s_W(U)}]\] Moreover \(k[M^{M_W}_{s_W(U)}]\) contains \(k\circ j^{M_W}_{s_W(U)}([\text{id}]_W) = j_U([\text{id}]_W) = [\text{id}]^{M_U}_{j_U(W)}\), so \[  [\text{id}]^{M_U}_{j_U(W)}\in k[M^{M_W}_{s_W(U)}]\] But \(j_U(M_W) = M_{j_U(W)}^{M_U}\) is the definable hull of \(j_U(j_W)[M_U]\cup \{[\text{id}]^{M_U}_{j_U(W)}\}\), and it follows that \(k\) is surjective.
\end{proof}
\end{prp}

\begin{cor}\label{SInternal}
For any \(U,W\in \textnormal{Un}\), \(U\I W\) if and only if \(s_W(U)\in M_W\).
\begin{proof}
The forwards direction is clear from \cref{DerivedPushforward}, and the reverse from \cref{PushforwardLemma}.
\end{proof}
\end{cor}

\begin{proof}[Proof of \cref{SeedExtendInternal}]
Suppose \(U\I W\) belong to \(\textnormal{Un}_\delta\). Then \(s_W(U)\in M_W\) and \(W^-(s_W(U)) = U\). Moreover \(\textsc{sp}(s_W(U)) = \sup j_W[\delta] \leq [\text{id}]_W\). Thus by definition \(U\swo W\).
\end{proof}

Note that even under ZFC, \(U\I W\) implies that \(U\swo W\) in the stronger sense of \cref{Linearity}; that is, both embeddings of the comparison are internal.

\begin{prp}[UA]\label{InternalTranslations}
Suppose \(U,W\in \textnormal{Un}\) and \(U\I W\). Then \(t_W(U) = s_W(U)\) and \(t_U(W) = j_U(W)\).
\begin{proof}
Since \(U\) and \(W\) divide \(U\oplus j_U(W) = W \oplus s_W(U)\), \(U\vee W\) divides \(U\oplus j_U(W)\). 

Let \(i : M_{U\vee W}\to M_{U\oplus j_U(W)}\) witness this. By \cref{mindef}, we have 
\begin{align*}i(j^{M_U}_{t_U(W)}([\text{id}]_U)) &= j^{M_U}_{j_U(W)}([\text{id}]_U)\\
i(j^{M_W}_{t_W(U)}([\text{id}]_W)) &= j^{M_W}_{s_W(U)}([\text{id}]_W)
\end{align*}
Thus \([\text{id}]_{U\oplus j_U(W)}\in \text{ran}(i)\), so \(i\) is the identity. Now the equations above imply \(t_W(U) = s_W(U)\) and \(t_U(W) = j_U(W)\).
\end{proof}
\end{prp}

Using this we can characterize how the internal relation fails to be strict.

\begin{thm}[UA]\label{InternalCommutativity}
Suppose \(U\I W\) and \(W\I U\). Then \(j_U(W) = s_U(W)\) and \(j_W(U) = s_W(U)\). Consequently, \(j^{M_U}_{j_U(W)} = j_W\restriction M_U\) and \(j^{M_W}_{j_W(U)} = j_U\restriction M_W\).
\begin{proof}
By \cref{InternalTranslations}, since \(U\I W\), \(t_U(W) = j_U(W)\) and \(t_W(U) = s_W(U)\). On the other hand since \(W\I U\), \(t_U(W) = s_U(W)\) and \(t_W(U) = j_W(U)\). Equating like terms, \(j_U(W) = s_U(W)\) and \(j_W(U) = s_W(U)\). By \cref{PushforwardLemma}, this implies the last statement of the theorem.
\end{proof}
\end{thm}

\cref{InternalCommutativity} can be a surprisingly powerful tool in proofs by contradiction. Good examples of this technique are \cref{DiscontinuityLemma}, \cref{HardRegularity}, and \cref{UltrafilterBound}. Is \cref{InternalCommutativity} provable in ZFC?

We can also prove the converse of \cref{SInternal}:

\begin{prp}[UA]\label{jInternal}
Suppose \(U,W\in \textnormal{Un}\) and \(t_U(W) = j_U(W)\). Then \(U\I W\).
\begin{proof}
We claim that \(j^{M_W}_{t_W(U)} = j_U\restriction M_W\). Note that \(j_U\restriction M_W\) is the unique elementary embedding \(i : M_W \to j_U(M_W)\) such that \(i([\text{id}]_W) = j_U([\text{id}]_W)\) and \(i\circ j_W = j_U\circ j_W\), since any elementary embedding of \(M_W\) is determined by its target model and its values on \(j_W[V]\cup \{[\text{id}]_W\}\). We claim that \(j^{M_W}_{t_W(U)}\) has these same properties, and hence the claim that  \(j^{M_W}_{t_W(U)} = j_U\restriction M_W\) follows.

Note first that \[M^{M_W}_{t_W(U)} = M^{M_U}_{j_U(W)} = j_U(M_W)\] Note second that \[j^{M_W}_{t_W(U)}([\text{id}]_W) = [\text{id}]^{M_U}_{t_U(W)} = [\text{id}]^{M_U}_{j_U(W)} = j_U([\text{id}]_W\] Note finally that \[j^{M_W}_{t_W(U)}\circ j_W = j^{M_U}_{j_U(W)}\circ j_U = j_U\circ j_W\]
This completes the proof.
\end{proof}
\end{prp}

Finally we will need the following theorem relating fixed points to the internal relation.

\begin{thm}[UA]\label{FixedPointInternal}
Suppose \(\kappa\) and \(\lambda\) are ordinals. Suppose \(U\) is an ultrafilter fixing \(\lambda\). Suppose \(W\) is the \(\swo\)-least ultrafilter such that \(j_W(\kappa) \geq \lambda\). Then \(U\I W\).
\begin{proof}
Since \(U^-(j_U(W)) = W\), by the minimality of \(t_U(W)\), \(t_U(W)\wo^{M_U} j_U(W)\). We will show that \(j_U(W)\wo^{M_U} t_U(W)\), so \(t_U(W) = j_U(W)\) and hence \(U\I W\) by \cref{jInternal}.

Note that \[j^{M_U}_{t_U(W)} (j_U(\kappa)) = j^{M_W}_{t_W(U)} (j_W(\kappa)) \geq j^{M_W}_{t_W(U)}(\lambda) \geq \lambda = j_U(\lambda)\]
In \(M_U\), \(j_U(W)\) is the \(\swo\)-least ultrafilter \(W_*\) such that \(j_{W_*}(j_U(\kappa)) \geq j_U(\lambda)\). Thus \(j_U(W)\wo^{M_U} t_U(W)\), as desired.
\end{proof}
\end{thm}

An important distinction between the internal relation and the Mitchell order is that {\it the internal relation propagates supercompactness}. 

\begin{lma}\label{BasicPropagation}
Suppose \(U,W\in \textnormal{Un}\) and \(U\I W\). Suppose \(W\) is \({<}\kappa\)-supercompact. Then \[\textnormal{Ord}^{<j_U(\kappa)}\cap M_U\subseteq M_W\]
\begin{proof}
Note that \(j_U(M_W)\subseteq M_W\), so \(\textnormal{Ord}^{<j_U(\kappa)}\cap M_U = j_U(\text{Ord}^{<\kappa}) \subseteq M_W\).
\end{proof}
\end{lma}

\begin{cor}\label{SupercompactTransfer1}
Suppose \(U,W\in \textnormal{Un}\) and \(U\I W\). If \(W\) is \({<}\kappa\)-supercompact, \(U\) is \(\lambda\)-supercompact, and \(j_U(\kappa) > \lambda\), then \(W\) is \(\lambda\)-supercompact.
\end{cor}

We will mostly apply a souped up version of \cref{SupercompactTransfer1} whose proof uses the Kunen inconsistency theorem, even though we could often just appeal to the somewhat more natural \cref{SupercompactTransfer1}.

\begin{prp}\label{SupercompactTransfer}
Suppose \(U,W\in \textnormal{Un}\) and \(U\I W\). Let \(\kappa = \textsc{crt}(U)\). Suppose \(W\) is \({<}\kappa\)-supercompact. If \(U\) is \(\lambda\)-supercompact, then \(W\) is \(\lambda\)-supercompact.
\begin{proof}
Let \(\langle \kappa_n : n < \omega\rangle\) be the critical sequence of \(U\). By Kunen's inconsistency theorem, we can fix \(n < \omega\) least such that \(\lambda < \kappa_{n+1}\). We prove by induction that \(W\) is \({<}\kappa_m\) supercompact for \(m\leq n\): if \(W\) is \({<}\kappa_m\)-supercompact and \(m < n\), then since \(U\) is \({<}\kappa_{m+1}\)-supercompact and \(j_U(\kappa_m) = \kappa_{m+1}\), \(W\) is \({<}\kappa_{m+1}\)-supercompact by \cref{SupercompactTransfer1}.

Therefore \(W\) is \({<}\kappa_n\)-supercompact. Since \(j_U(\kappa_n) > \lambda\), one more application of \cref{SupercompactTransfer1} implies \(W\) is \(\lambda\)-supercompact.
\end{proof}
\end{prp}

Similarly, the internal relation propagates strong compactness.

\begin{prp}\label{CompactPropagation}
Suppose \(U,W\in \textnormal{Un}\) and \(U\I W\). If \(W\) is \({<}\kappa\)-supercompact, \(U\) has the \((\delta,\lambda)\)-covering property, and \(\lambda < j_U(\kappa)\), then \(W\) has the \((\delta,\lambda)\)-covering property.
\begin{proof}
Suppose \(A\subseteq [\text{Ord}]^\delta\). Let \(D\in M_U\) be such that \(A\subseteq D\) and \(|D|^{M_U} = \lambda\). Then \(D\in M_W\) by \cref{BasicPropagation}, so it suffices to show that \(|D|^{M_W} = \lambda\). But this follows easily from the fact that \(\text{Ord}^\lambda\cap M_U\subseteq M_W\), again by \cref{BasicPropagation}.
\end{proof}
\end{prp}

\section{The First Strongly Compact Cardinal}\label{FirstStronglyCompact}
\subsection{The Least Ultrafilter}
\begin{defn}
A countably complete ultrafilter on a limit ordinal \(\delta\) is {\it \(0\)-order} if it is weakly normal and concentrates on the set of ordinals that do not carry countably complete uniform ultrafilters. 
\end{defn}

We begin with a trivial lemma that turns out to be useful.

\begin{lma}\label{Trivia}
Suppose \(U_*\) is a \(0\)-order ultrafilter on a singular ordinal \(\delta_*\) of cofinality \(\delta\). Let \(U\) be the weakly normal ultrafilter on \(\delta\) derived from \(U_*\). Then \(U\equiv U_*\), and in fact for any continuous cofinal function \(p: \delta \to \delta_*\), \(p_*(U) = U_*\).
\begin{proof}
Let \(p : \delta\to \delta_*\) be continuous and cofinal. Then \begin{align*}j_{U_*}(p)(\sup j_{U_*}[\delta]) &= \sup j_{U_*}(p)[\sup j_{U_*}[\delta]] \\&= \sup j_{U_*}(p)\circ j_{U_*}[\delta] \\&= \sup j_{U_*}\circ p[\delta]\\& = \sup j_{U_*}[\delta_*]\end{align*}
Since \(U_*\) is weakly normal, \([\text{id}]_{U_*} = \sup j_{U_*}[\delta_*]\), so the calculation implies \(p_*(U)= U_*\), which proves the lemma.
\end{proof}
\end{lma}

Therefore the interesting \(0\)-order ultrafilters lie on regular cardinals, and all other \(0\)-order ultrafilters are reducible to them.

Generalizing an observation due to Solovay, Ketonen introduced the seed order on weakly normal ultrafilters and proved the following fact.

\begin{thm}[Ketonen]
A uniform countably complete ultrafilter on a limit ordinal \(\delta\) is \(0\)-order if and only if it is an \(\swo\)-minimal element of \(\textnormal{Un}_\delta\).
\begin{proof}
Suppose \(U\in \textnormal{Un}_\delta\), \(\alpha\) is an ordinal such that \(\sup j_U[\delta] \leq \alpha \leq [\text{id}]_U\), and \(W_*\in \textnormal{Un}^{M_U}_\alpha\). Then \(U^-(W_*) \swo U\) and \(U^-(W_*)\in \textnormal{Un}_\delta\). Conversely if \(W\swo U\) and \(W\in \textnormal{Un}_\delta\), then for some ordinal \(\alpha\) such that \(\sup j_U[\delta] \leq \alpha \leq [\text{id}]_U\) and \(W_*\in \textnormal{Un}^{M_U}_\alpha\), \(W = U^-(W_*)\).

Thus \(U\) is an \(\swo\)-minimal element of \(\textnormal{Un}_\delta\) if and only if for all ordinals \(\alpha\) such that \(\sup j_U[\delta] \leq \alpha \leq [\text{id}]_U\), \[\textnormal{Un}^{M_U}_\alpha= \emptyset\] Recalling that \(\textnormal{Un}_{\alpha}\) is nonempty whenever \(\alpha\) is a successor ordinal, this holds if and only if \(\sup j_U[\delta] = [\text{id}]_U\) and \([\text{id}]_U\) carries no uniform ultrafilters in \(M_U\), or equivalently if and only if \(U\) is \(0\)-order.
\end{proof}
\end{thm}

In particular, if a limit ordinal \(\delta\) carries a countably complete uniform ultrafilter, it carries a \(0\)-order ultrafilter. If the seed order is linear, then minimal elements of \(\textnormal{Un}_\delta\) are {\it minimum} elements, which yields the following corollary.

\begin{cor}[UA]\label{ZeroUnique}
A limit ordinal \(\delta\) carries at most one \(0\)-order ultrafilter.
\end{cor}

\begin{defn}[UA]
If \(\delta\) is a limit ordinal that carries a countably complete uniform ultrafilter, we call the unique \(0\)-order ultrafilter on \(\delta\) the {\it least ultrafilter} on \(\delta\).
\end{defn}

In the context of UA, if \(U\) is a least ultrafilter, then \(U\) is irreducible in a very strong sense.  

\begin{thm}[UA]\label{GeneralInternal}
Let \(U\) be the least ultrafilter on a limit ordinal \(\delta\). Fix \(W\in \textnormal{Un}\) and let \(\delta_* = \sup j_W[\delta]\). Then one of the following holds: 
\begin{enumerate}[(1)]
\item \(t_W(U)\) is the least ultrafilter on \(\delta_*\) as computed in \(M_W\).
\item \(t_W(U) = P^{M_W}_{\delta_*}\).
\end{enumerate}
\begin{proof}
Let \(D\) be the ultrafilter derived from \(W\) using \(\delta_*\). Then \(U\wo D\), so \[t_W(U) \wo^{M_W} t_W(D)\wo^{M_W} P^{M_W}_{\delta_*}\] with the first inequality following from \cref{op} and the second from definition of translation functions and the fact that \(D = W^{-}(P^{M_W}_{\delta_*})\). On the other hand, \(\textsc{sp}(t_W(U)) \geq \delta_*\) since otherwise \[\textsc{sp}(U) = \textsc{sp}(W^{-}(t_W(U))) < \delta\] contradicting that \(\textsc{sp}(U) = \delta\).

It follows from \cref{SpaceLemma} that either \(t_W(U)\in \textnormal{Un}^{M_W}_{\delta_*}\) or \(t_W(U) = P^{M_W}_{\delta_*}\).

Suppose \(t_W(U)\in \textnormal{Un}^{M_W}_{\delta_*}\). We show \(t_W(U)\) is \(\swo^{M_W}\)-minimal in \(\textnormal{Un}^{M_W}_{\delta_*}\), which proves the theorem. Fix \(Z\in \textnormal{Un}^{M_W}\) with \(Z\swo^{M_W}t_W(U)\). Then \(W^{-}(Z) \swo U\) since \[t_W(W^{-}(Z)) \wo Z \swo t_W(U)\] and \(t_W\) is order-preserving by \cref{op}. Hence \(W^{-}(Z)\in \textnormal{Un}_{<\delta}\), and therefore \(Z\in \textnormal{Un}_{<\delta_*}\), as desired.
\end{proof}
\end{thm}

This leads to a useful characterization of the internal ultrapower embeddings of a least ultrapower.

\begin{thm}[UA]
Let \(U\) be the least ultrafilter on a limit ordinal \(\delta\). Suppose \(W\) is a countably complete ultrafilter and \(k : M_U\to M_W\) is an elementary embedding with \(k\circ j_U = j_W\). Then \(k\) is definable over \(M_U\) if and only if \(k\) is continuous at \(\sup j_U[\delta]\).
\begin{proof}
One direction is obvious. Suppose conversely that \(k\) is continuous at \(\sup j_U[\delta]\). Then \(k(\sup j_U[\delta]) = \sup j_W[\delta]\). Therefore \(\sup j_W[\delta]\) carries no uniform ultrafilters in \(M_W\). It follows that \(t_W(U) = P^{M_W}_{\sup j_W[\delta]}\) by \cref{GeneralInternal}. Therefore by \cref{Reciprocity}, \(j^{M_U}_{t_U(W)}\circ j_U = j_W\) and \(j^{M_U}_{t_U(W)}(\sup j_U[\delta]) = \sup j_W[\delta]\). It follows that \(k = j^{M_U}_{t_U(W)}\), so \(k\) is definable over \(M_U\), as desired.
\end{proof}
\end{thm}

Restated, this is a characterization (reminiscent of \cref{CanonicalInternal}) of the countably complete \(M_U\)-ultrafilters that belong to \(M_U\):

\begin{thm}[UA]\label{ZeroInternal}
Let \(U\) be the least ultrafilter on a limit ordinal \(\delta\). Suppose \(W_*\) is a countably complete \(M_U\)-ultrafilter. Then \(W_*\in M_U\) if and only if \(j^{M_U}_{W_*}\) is continuous at \(\textnormal{cf}^{M_U}(\sup j_U[\delta])\).
\end{thm}

If \(\delta\) is a regular cardinal and \(U\) is a countably complete ultrafilter, a theorem of Ketonen \cite{Ketonen} states that the value of \(\textnormal{cf}^{M_U}(\sup j_U[\delta])\) determines the covering property of \(M_U\) with respect to \(\delta\)-sequences. This makes \cref{ZeroInternal} extremely useful for the purpose of gauging the strong compactness and supercompactness of least ultrafilters under UA. We briefly discuss the relationship between cofinalities and covering properties of ultrapowers.

\begin{defn}
Suppose \(M\) is an inner model, \(\delta\) is a cardinal, and \(\lambda\) is an \(M\)-cardinal. Then \(M\) has the {\it \((\delta,\lambda)\)-covering property} if for every set \(A\subseteq \text{Ord}\) such that \(|A|\leq \delta\), there exists a set \(B\subseteq \text{Ord}\) belonging to \(M\) such that \(A\subseteq B\) and \(|B|^M \leq \lambda\).
\end{defn}

\begin{thm}[Ketonen]\label{CofinalityRegularity}
Suppose \(\delta\) is a regular cardinal and \(U\) is a countably complete ultrafilter. Let \(\lambda = \textnormal{cf}^{M_U}(\sup j_U[\delta])\). Then \(M_U\) has the \((\delta,\lambda)\)-covering property.
\begin{proof}
By a standard argument, it suffices to show that there is a set \(B\in M_U\) with \(j_U[\delta]\subseteq B\) and \(|B|^M\leq \lambda\).

Let \(C\in M_U\) be cofinal subset of \(\sup j_U[\delta]\) of order type \(\lambda\). Since \(\delta\) is regular, we can partition \(\delta\) into bounded intervals \(\langle I_\alpha : \alpha < \delta\rangle\) with the property that for all \(\alpha < \delta\), \(j_U(I_\alpha)\cap C\neq \emptyset\). (This is achieved by setting \[I_\alpha = \left[\sup\left\{\eta + 1 : \eta\in \bigcup_{\beta < \alpha} I_\beta\right\}, \xi\right)\] where \(\xi < \delta\) is least ensuring \(j_U(I_\alpha)\cap C\neq \emptyset\).) 

Let \[\langle J_\alpha : \alpha < j_U(\delta)\rangle = j_U(\langle I_\alpha : \alpha < \delta\rangle)\] and let \[B = \{\alpha < j_U(\delta) : J_\alpha \cap C\neq \emptyset\}\] Clearly \(B\in M_U\). Moreover since \(|C|^{M_U} = \lambda\) and the \(J_\alpha\) are disjoint, it follows that \(|B|^{M_U} \leq \lambda\). Finally \(j_U[\delta]\subseteq B\) since we arranged that \(J_{j_U(\alpha)}\cap C = j_U(I_\alpha) \cap C\neq \emptyset\) for all \(\alpha < \delta\).
\end{proof}
\end{thm}

A second proof proceeds by noting that it suffices to show that for some \(C\) of cardinality \(\delta\), \(j_U[C]\) is covered by a set \(D\in M_U\) with \(|D|^{M_U} = \lambda\). Then let \(D\) be any club of order type \(\lambda\) in \(\sup j_U[\delta]\) and let \(C = j_U^{-1}[D]\).

\begin{defn}
We say \(U\) has the {\it tight covering property} at \(\delta\) if \(M_U\) has the \((\delta,\delta)\)-covering property.
\end{defn}

As a corollary of \cref{ZeroInternal} and \cref{CofinalityRegularity}, we immediately obtain the following dichotomy:

\begin{cor}[UA]\label{CoverDichotomy}
Let \(U\) be the least ultrafilter on a regular cardinal \(\delta\). Either \(U\) has the tight covering property at \(\delta\) or \(U\cap M_U\in M_U\).
\begin{proof}
Of course \(\delta \leq \text{cf}^{M_U}(\sup j_U[\delta])\). If equality holds, then tight covering holds by \cref{CofinalityRegularity}. If instead \(\delta < \text{cf}^{M_U}(\sup j_U[\delta])\), then \(j^{M_U}_{U\cap M_U}\) is continuous at \(\text{cf}^{M_U}(\sup j_U[\delta])\) so \(U\cap M_U\in M_U\) by \cref{ZeroInternal}.
\end{proof}
\end{cor}

It is not clear whether it is consistent with ZFC that there is a countably complete nonprincipal ultrafilter \(U\) on an ordinal such that \(U\cap M_U\in M_U\), but the usual proofs that \(U\notin M_U\) do not give much insight into this question. In any case we can rule out that \(U\cap M_U\in M_U\) in various contexts. (See for example \cref{ZeroTheorem}.)

\cref{ZeroInternal} yields the amenability of many ultrafilters to a least ultrapower. We will use the following classical theorem due to Hausdorff to transform this into the amenability of many sets of ordinals. This is the key to obtaining supercompactness from strong compactness.

\begin{defn}
Suppose \(\kappa\leq \delta\) are cardinals. A family \(\mathcal F\) of subsets of \(\delta\) is {\it \(\kappa\)-independent} if for any subfamilies \(\mathcal F_0,\mathcal F_1\subseteq \mathcal F\) of cardinality less than \(\kappa\), \[|\{\alpha < \delta : \forall X\in \mathcal F_0\ \alpha\in X\text{ and }\forall X\in \mathcal F_1\ \alpha\notin X\}| = \delta\]
\end{defn}

Let \(\mathcal F_\delta\) denote the filter of \(X\subseteq \delta\) such that \(|\delta\setminus X| < \delta\). A family \(\mathcal F\) is \(\kappa\)-independent if and only if for any \(\mathcal G\subseteq \mathcal F\), \[\mathcal G\cup \{\delta\setminus X : X\notin \mathcal G\}\cup \mathcal F_\delta\] generates a \(\kappa\)-complete filter. The official definition has the benefit that it is obviously absolute between \(V\) and any \(\kappa\)-closed inner model.

\begin{thm}[Hausdorff]\label{IndependentFamily}
Suppose \(\kappa\leq \delta\) are cardinals and \(\delta^{<\kappa} = \delta\). Then there is a \(\kappa\)-independent family of subsets of \(\delta\) of cardinality \(2^\delta\).
\end{thm}

We emphasize that the next lemma lies at the center of the relationship between strong compactness and supercompactness. 

\begin{lma}[UA]\label{Strength}
Suppose \(\kappa\) is a cardinal and \(U\) is a \(\kappa\)-supercompact ultrafilter. Suppose \(\gamma\) is an \(M_U\)-cardinal with \(\textnormal{cf}^{M_U}(\gamma)\geq \kappa\) such that for any \(W\in \textnormal{Un}_{{\leq}\gamma}\), \(W\cap M_U\in M_U\). Suppose \(\lambda \leq (2^\gamma)^{M_U}\) and \(\kappa\) is \(\lambda\)-strongly compact. Then \(P(\lambda)\subseteq M_U\).
\begin{proof}
First note that \(\kappa\) is \({<}\gamma\)-strongly compact in \(M_U\). To prove this, it suffices by a theorem of Ketonen \cite{Ketonen} to show that every \(M_U\)-regular \(\iota\in [\kappa,\gamma)\) carries a uniform \(\kappa\)-complete ultrafilter in \(M_U\). Note that \(\text{cf}(\iota)\in [\kappa,\gamma)\) since \(M_U\) is closed under \(\kappa\)-sequences. Hence \(\iota\) carries a uniform ultrafilter \(W\) since \(\kappa\) is \(\gamma\)-strongly compact in \(V\). But by assumption \(W\cap M_U\in M_U\), so \(\iota\) carries a uniform \(\kappa\)-complete ultrafilter in \(M_U\).

Since \(\kappa\) is \({<}\gamma\)-strongly compact in \(M_U\) and \(\text{cf}^{M_U}(\gamma)\geq \kappa\), \((\gamma^{<\kappa})^{M_U} = \gamma\). Applying \cref{IndependentFamily} in \(M_U\), let \(\langle A_\alpha : \alpha < (2^\gamma)^{M_U}\rangle\) be a \(\kappa\)-independent family of subsets of \(\gamma\) relative to \(M_U\). Since \(M_U\) is closed under \(\kappa\)-sequences, \(M_U\) is correct about \(\kappa\)-independence, so \(\langle A_\alpha : \alpha < (2^\gamma)^{M_U}\rangle\) is truly \(\kappa\)-independent. 

Fix \(X\subseteq \lambda\). We will show \(X\in M_U\). The filter \(F\) generated by \[\{A_\alpha : \alpha \in X\}\cup \{\gamma\setminus A_\alpha : \alpha\in \lambda\setminus X\}\] is \(\kappa\)-complete by \(\kappa\)-independence. Since \(\kappa\) is \(\lambda\)-strongly compact, any \(\kappa\)-complete filter on \(\gamma\) generated by \(\lambda\) sets extends to a \(\kappa\)-complete ultrafilter. Therefore let \(D\) be a \(\kappa\)-complete ultrafilter on \(\gamma\) extending \(F\). Since \(D\cap M_U\in M_U\), \(X\in M_U\): \[X = \{\alpha < \lambda : A_\alpha\in D\cap M_U\}\]
This completes the proof.
\end{proof}
\end{lma}

\begin{thm}[UA]\label{ZeroTheorem}
Let \(U\) be the least ultrafilter on a regular cardinal \(\delta\). Let \(\kappa\) be the completeness of \(U\). Assume that \(\kappa\) is \(\delta\)-strongly compact. Then \(U\) is \({<}\delta\)-supercompact and has the tight covering property at \(\delta\). If \(\delta\) is not strongly inaccessible then \(U\) is \(\delta\)-supercompact.
\begin{proof}

Assume towards a contradiction that \(U\) does not have the tight covering property at \(\delta\). Then by \cref{ZeroInternal}, for any \(W\in \textnormal{Un}_{\leq\delta}\), \(W\cap M_U\in M_U\). But now by \cref{Strength}, \(P(\delta)\subseteq M_U\). But then \(U\cap M_U = U\), so \(U\in M_U\), which is impossible. It follows that \(U\) has the tight covering property at \(\delta\).

By \cref{ZeroInternal}, for all \(W\in \textnormal{Un}_{{<}\delta}\), \(W\cap M_U\in M_U\). Therefore by \cref{Strength}, \(\bigcup_{\gamma < \delta} P(\gamma)\subseteq M_U\).

Suppose \(\gamma < \delta\) is a regular cardinal. Note that \(\text{cf}^{M_U}(\sup j_U[\gamma]) \leq \delta\) by the tight covering property. Since \(\delta\) is regular, it follows that \(\text{cf}^{M_U}(\sup j_U[\gamma]) < \delta\). By \cref{CofinalityRegularity}, \(j_U[\gamma]\) can be covered by a set of \(M_U\)-cardinality \(\lambda < \delta\). Since \(P(\lambda)\subseteq M_U\), \(j_U[\gamma]\in M_U\). Hence \(U\) is \(\gamma\)-supercompact.

Finally suppose \(\delta\) is not a strong limit cardinal, and we will show that \(U\) is \(\delta\)-supercompact. It suffices to show that \(P(\delta)\subseteq M_U\): then by the tight covering property, \(j_U[\delta]\) is covered by a set of \(M_U\)-cardinality \(\delta\), and since \(P(\delta)\subseteq M_U\), \(j_U[\delta]\in M_U\).

We split into two cases.

\begin{case}\label{CofCase}For some \(\gamma < \delta\) with \(\text{cf}(\gamma)\geq \kappa\), \(2^\gamma\geq \delta\).\end{case}
Note that \(\delta \leq 2^\gamma \leq (2^\gamma)^{M_U}\), where the final inequality follows from the fact that \(P(\gamma)\subseteq M_U\). We must therefore have \(P(\delta)\subseteq M_U\) by \cref{Strength} with \(\lambda = \delta\).

\begin{case}Otherwise.\end{case}
Since \(\delta\) is not a strong limit cardinal, there must be some \(\lambda < \delta\) with \(2^\lambda \geq \delta\). Since we are not in \cref{CofCase}, \(\lambda^{<\kappa} \geq \delta\): otherwise \(\gamma = \lambda^{<\kappa}\) witnesses the hypotheses of \cref{CofCase}. Note that \(U\) is \(\lambda\)-supercompact since \(\lambda\) is a singular limit cardinal and \(U\) is \(\gamma\)-supercompact for all regular \(\gamma < \lambda\). Since \(j_U\restriction \lambda\in M_U\) and \(U\) is \(\kappa\)-complete, \(j_U\restriction P_\kappa(\lambda)\in M_U\). Thus \(U\) is \(\lambda^{<\kappa}\)-supercompact. So \(U\) is \(\delta\)-supercompact.
\end{proof}
\end{thm}

We now show that in certain circumstances we can obtain the hypotheses of \cref{ZeroTheorem}, leading to a proof from UA that the least strongly compact cardinal is supercompact.

This involves another theorem due to Ketonen, who used the combinatorics of 0-order ultrafilters to give a second proof of his characterization of strongly compact cardinals in terms of uniform countably complete ultrafilters. We include the version that is most relevant to us, since this is not exactly what Ketonen proved.

\begin{thm}[Ketonen]\label{KetonenTheorem}
Suppose \(U\) is a \(0\)-order ultrafilter on a regular cardinal \(\delta\). Then for any \(\gamma \leq \delta\), \(U\) is \((\gamma,\delta)\)-regular if and only if every regular cardinal in the interval \([\gamma,\delta]\) carries a uniform countably complete ultrafilter.
\begin{proof}
If there is a \((\gamma,\delta)\)-regular ultrafilter then easily every regular cardinal in the interval \([\gamma,\delta]\) carries a uniform countably complete ultrafilter. 

Conversely assume every regular cardinal in the interval \([\gamma,\delta]\) carries a uniform countably complete ultrafilter. Since \(U\) is \(0\)-order, \(\sup j_U[\delta]\) carries no uniform countably complete ultrafilter in \(M_U\). Therefore \(\text{cf}^{M_U}(\sup j_U[\delta])\) carries no uniform countably complete ultrafilter in \(M_U\). It follows that \(\text{cf}^{M_U}(\sup j_U[\delta])\notin j_U([\gamma,\delta])\). Hence \(\text{cf}^{M_U}(\sup j_U[\delta])< j_U(\gamma)\), which implies \(U\) is \((\gamma,\delta)\)-regular.
\end{proof}
\end{thm}

\begin{lma}[UA]\label{EasyRegularity}
Let \(U\) be the least ultrafilter on a regular cardinal \(\delta\). Let \(\kappa\) be its completeness, and let \(\bar \kappa\leq \kappa\) be the least ordinal such that for some \(W\), \(j_W(\bar \kappa) > \kappa\). Suppose \(W\) is a countably complete ultrafilter with \(j_W(\bar \kappa) > \kappa\) and \(U\I W\). Then \(U\) and \(W\) are \((\bar \kappa,\delta)\)-regular. Hence \(\bar \kappa = \kappa\).
\end{lma}

For the proof we need a version of the Kunen inconsistency. For this we require a useful lemma that under favorable cardinal arithmetic conditions often allows us to replace arbitrary ultrafilters with ultrafilters on small sets.

\begin{lma}\label{exponential}
Suppose \(U\) is an ultrafilter on a set \(X\) and \(\langle f_i : i \in I\rangle\) is a sequence of functions from \(X\) to a set \(Y\). Then there is a function \(p : X\to Y^I\) such that letting \(W = p_*(U)\) and \(k : M_W\to M_U\) be the factor embedding, \([f_i]_U\in \textnormal{ran}(k)\) for all \(i\in I\).
\begin{proof}
Take \(p: X\to Y^I\) such that \(p(x)(i) = f_i(x)\). Let \(W = p_*(U)\) and let \(h = [\text{id}]_W\in j_W(Y^I)\). Then \begin{align*}k(h(j_W(i))) &= j_U(p)([\text{id}]_U)(j_U(i)) \\
&= j_U(\langle f_i : i \in I\rangle)_{j_U(i)}([\text{id}]_U) \\
&= j_U(f_i)([\text{id}]_U)\\
&= [f_i]_U\end{align*}
so \([f_i]_U\in \text{ran}(k)\).
\end{proof}
\end{lma}

\begin{cor}
Suppose \(U\) is a countably complete ultrafilter and \(\lambda\) is a cardinal. There is a function \(p : \textsc{sp}(U)\to 2^\lambda\) such that letting \(W = p_*(U)\) and \(k: M_W\to M_U\) be the factor embedding, \(\textsc{crt}(k) > \lambda\).
\end{cor}

The notation \(\textnormal{Un}_{{<}\kappa}\I U\) abbreviates the statement that \(W\I U\) for all \(W\in \textnormal{Un}_{<\kappa}\).

\begin{lma}[UA]\label{Kunen2}
Suppose \(U\) is a countably complete ultrafilter and \(\lambda\) is the first fixed point of \(j_U\) above its critical point. Then for some \(\nu < \lambda\), \(\textnormal{Un}_{\leq2^{\nu}}\not\I U\).
\begin{proof}
By Kunen's inconsistency theorem, there is some \(\gamma < \lambda\) such that \(U\) is not \(\gamma\)-supercompact. Let \(\nu = \sup j_U[\gamma]\), so \(\nu < \lambda\). By \cref{exponential} there is some \(W\in \textnormal{Un}_{\leq2^{\nu}}\) such that \(j_W \restriction \gamma = j_U\restriction \gamma\). Note that \(W\not\I U\) since otherwise \(j_U\restriction \gamma \in M_U\), contradicting that \(U\) is not \(\gamma\)-supercompact.
\end{proof}
\end{lma}

\begin{cor}[UA]\label{Kunen}
Suppose \(U\) is a countably complete ultrafilter and \(\lambda\) is the first fixed point of \(j_U\) above its critical point. Suppose \(\kappa\) is a strong limit cardinal and \(\textnormal{Un}_{{<}\kappa}\I U\). Then \(\kappa < \lambda\).
\end{cor}

\begin{cor}[UA]\label{ClosedCompleteness}
Suppose \(\kappa\) is a strong limit cardinal that is closed under ultrapowers. Suppose \(U\) is a countably complete ultrafilter such that \(\textnormal{Un}_{{<}\kappa}\I U\). Then \(U\) is \(\kappa\)-complete.
\end{cor}

\begin{proof}[Proof of \cref{EasyRegularity}]
Let \(\delta_* = \sup j_W[\delta]\). The key point is that we have both that \(t_W(U) = s_W(U)\) by \cref{InternalTranslations} and also that \(t_W(U)\) is the least ultrafilter on \(\delta_*\) by \cref{GeneralInternal}. Hence \begin{equation}\label{Equivalence}j_U\restriction M_W = j^{M_W}_{U_*}\end{equation} where \(U_*\) is the least ultrafilter of \(M_W\) on \(\text{cf}^{M_W}(\delta_*)\). (Here we use \cref{Trivia} in \(M_W\) to see that in \(M_W\), \(U_*\) is equivalent to the least ultrafilter on \(\delta_*\).)

Suppose towards a contradiction that  \(W\) is not \((\bar \kappa,\delta)\)-regular, and therefore \(\text{cf}^{M_W}(\delta_*) \geq j_W(\bar \kappa)\). Work in \(M_W\). Note that \(j_W(\bar \kappa)\) is a strong limit cardinal that is closed under ultrapowers, and moreover \(\textnormal{Un}^{M_W}_{<j_W(\bar \kappa)} \I U_*\) by \cref{ZeroInternal}. Therefore \(U_*\) is \(j_W(\bar \kappa)\)-complete by \cref{ClosedCompleteness}. But \(j_W(\bar \kappa) > \kappa\), contradicting that the critical point of \(j^{M_W}_{U_*}\) is \(\kappa\) by \cref{Equivalence}. Therefore our assumption was false, so \(W\) is \((\bar \kappa,\delta)\)-regular.

Finally we conclude that \(\bar \kappa = \kappa\): first, \(\bar \kappa \leq \kappa\) by definition, and second, by \cref{KetonenTheorem}, \(U\) is \((\bar \kappa,\delta)\)-regular, so \(\kappa \leq \bar \kappa\).
\end{proof}

\begin{cor}[UA]
Suppose there are arbitrarily large regular cardinals carrying countably complete uniform ultrafilters. Suppose \(\kappa\) is the least cardinal mapped arbitrarily high by countably complete ultrafilters. Then \(\kappa\) is supercompact.
\begin{proof}
Note that \(\kappa\) is closed under ultrapowers.

Let \(\delta \geq \kappa\) be a regular cardinal carrying a countably complete uniform ultrafilter. We claim the least ultrafilter \(U\) on \(\delta\) witnesses \(\kappa\) is \({<}\delta\)-supercompact. 

Let \(\lambda\) be the least fixed point of \(U\) above \(\textsc{crt}(U)\). Let \(W\) be the \(\swo\)-least ultrafilter such that \(j_W(\kappa) \geq \lambda\). Then by \cref{FixedPointInternal}, \(U\I W\). Clearly \(j_W(\kappa) > \textsc{crt}(U)\). Since \(\kappa\) is closed under ultrapowers, \(\kappa\) is the least ordinal such that \(j_Z(\kappa) > \textsc{crt}(U)\) for some \(Z\in \textnormal{Un}\). 

Therefore by \cref{EasyRegularity}, \(\textsc{crt}(U) = \kappa\) and \(U\) is \((\kappa,\delta)\)-regular. Now by \cref{ZeroTheorem}, \(U\) witnesses that \(\kappa\) is \({<}\delta\)-supercompact.
\end{proof}
\end{cor}

\begin{cor}[UA]\label{GlobalLeastSupercompact}
The least strongly compact cardinal is supercompact.
\end{cor}

\subsection{The Next Ultrafilter}\label{NextUltrafilter}
We continue our investigation of \(0\)-order ultrafilters, proving some local refinement of the results we have seen so far. 

We first point out that this is much easier if one assumes UA + GCH. We will use the following theorem essentially due to Ketonen.

\begin{thm}[Ketonen]\label{SuccessorDecomposable}
Suppose \(\gamma\) is regular and \(U\) is a countably complete uniform ultrafilter on \(\gamma^+\). Then \(U\) is \(\gamma\)-decomposable.
\begin{proof}
Note that \(\text{cf}(\sup j_U[\gamma^+]) \leq j_U(\gamma)\). If equality holds, then \(\text{cf}(j_U(\gamma)) = \gamma^+\), so \(j_U\) is discontinuous at \(\gamma\). If strict inequality holds, then \(U\) has the \((\gamma^+, \lambda)\)-covering property for some \(\lambda < j_U(\gamma)\), by \cref{CofinalityRegularity}, which implies again that \(j_U\) is discontinuous at \(\gamma\).
\end{proof}
\end{thm}

\begin{prp}[UA]\label{GCHEasy}
Suppose \(U\) is the least ultrafilter on a regular cardinal \(\delta\) and for all \(\gamma\) with \(\gamma < \delta\), \(2^\gamma = \gamma^+\). Then \(\textsc{crt}(U)\) is \(\delta\)-strongly compact.
\begin{proof}[Sketch]
We will use the fact that if \(\gamma\) is regular and \(\gamma^+\) carries a uniform \(\bar \kappa\)-complete ultrafilter, so does \(\gamma\). 

Let \(\bar \kappa\) be least such that for some \(W\), \(j_W(\bar \kappa) > \delta\). One shows \(\bar \kappa\) is a strong limit cardinal that is closed under ultrapowers. Therefore \(\bar \kappa \leq \kappa\). We show \(\bar \kappa\) is \(\delta\)-strongly compact. This will imply the theorem: since \(\bar \kappa\) is \(\delta\)-strongly compact, \(U\) is \((\bar \kappa,\delta)\)-regular by \cref{KetonenTheorem}, so \(\kappa \leq \bar \kappa\), and hence \(\kappa = \bar \kappa\) is \(\delta\)-strongly compact.

Let \(\lambda\) be the largest limit cardinal with \(\lambda \leq \delta\). It suffices by GCH and the first sentence of this proof to show that every regular \(\gamma\) with \(\bar \kappa \leq \gamma < \lambda\) carries a \(\bar \kappa\)-complete ultrafilter. Using \cref{exponential} and GCH, one shows that the space of the least ultrafilter \(W\) sending \(\bar \kappa\) above \(\gamma^+\) is at most \(\gamma^{++}\). On the other hand \(\textsc{sp}(W)\geq \gamma\) since \(\bar \kappa^{<\gamma} = \gamma\). Moreover since every ultrafilter in \(\textnormal{Un}_{<\bar \kappa}\) fixes \(\gamma^+\), \(\textnormal{Un}_{<\bar \kappa} \I W\) by \cref{FixedPointInternal}. Thus \(W\) is \(\bar \kappa\)-complete by \cref{ClosedCompleteness}. Hence \(\gamma\) carries a uniform \(\bar \kappa\)-complete ultrafilter by the first sentence.
\end{proof}
\end{prp}

This result suffices for the analysis under GCH of higher strongly compact cardinals, so the reader who is not interested in the fine structure of least ultrafilters under UA without assuming GCH can skip ahead to \cref{SecondStronglyCompact}.

Without GCH, we will show the following:

\begin{thm}[UA]\label{LocalTheorem}
Suppose \(\gamma < \delta\) are regular cardinals and \(\textnormal{Un}_\delta,\textnormal{Un}_\gamma\neq \emptyset\). Then the least ultrafilter on \(\delta\) is \(\gamma\)-supercompact.
\end{thm}

Before proving \cref{LocalTheorem}, which will take several pages, we give some applications.

\begin{cor}[UA]\label{GCHFree}
Suppose \(U\) is the least ultrafilter on a regular cardinal \(\delta\). Let \(\kappa\) be the completeness of \(U\).
\begin{enumerate}[(1)]
\item If \(\delta\) is the successor of a regular cardinal \(\lambda\), then \(U\) is \(\delta\)-supercompact.
\item If \(\delta\) is the successor of a singular limit \(\lambda\) of regular \(\gamma\) with \(\textnormal{Un}_\gamma\neq \emptyset\), then \(U\) is \(\delta\)-supercompact.
\item Therefore if \(\delta\) is the successor of a singular cardinal \(\lambda\) with \(\textnormal{cf}(\lambda) < \kappa\), then \(U\) is \(\delta\)-supercompact.
\item If \(\delta\) is weakly inaccessible and \(\sup j_U[\delta]\) is not regular in \(M_U\), then \(U\) is \({<}\delta\)-supercompact.
\item Therefore if \(\delta\) is weakly inaccessible but not weakly \({<}\delta^+\)-Mahlo, then \(U\) is \({<}\delta\)-supercompact.
\end{enumerate}
\begin{proof}[Proof of (1)]
Since \(\lambda\) is regular and \(\textnormal{Un}_{\lambda^+}\neq\emptyset\), \(\textnormal{Un}_\lambda\neq\emptyset\). Hence \(U\) is \(\lambda\)-supercompact by \cref{LocalTheorem}. Since \(\kappa\) is \(\lambda\)-supercompact and \(\delta\) carries a \(\kappa\)-complete uniform ultrafilter, \(\kappa\) is \(\delta\)-strongly compact by Ketonen's characterization of strong compactness. Hence \(U\) is \(\delta\)-supercompact by \cref{ZeroTheorem}.
\end{proof}
\begin{proof}[Proof of (2)]
Again \(U\) is \({<}\lambda\)-supercompact by \cref{LocalTheorem}. Since \(\kappa\) is \(\lambda\)-supercompact and \(\delta\) carries a \(\kappa\)-complete uniform ultrafilter, \(\kappa\) is \(\delta\)-strongly compact by Ketonen's characterization of strong compactness. Hence \(U\) is \(\delta\)-supercompact by \cref{ZeroTheorem}.
\end{proof}
\begin{proof}[Proof of (3)]
By a result of Ketonen \cite{Ketonen}, the hypotheses of (2) follow from those of (3) using \cref{CofinalityRegularity}. The proof is similar to the proof of (4) below.
\end{proof}
\begin{proof}[Proof of (4)]
If \(\sup j_U[\delta]\) is not regular in \(M_U\), then for some \(\iota < \delta\), \[\text{cf}^{M_U}(\sup j_U[\delta]) < j_U(\iota)\] In other words \(U\) is \((\iota,\delta)\)-regular, and hence discontinuous at every regular \(\gamma\in [\iota,\delta]\). Hence every regular \(\gamma\in [\iota,\delta]\) carry countably complete uniform ultrafilters. So \(U\) is \({<}\delta\)-supercompact by \cref{LocalTheorem}. 
\end{proof}
\begin{proof}[Proof of (5)]
Suppose \(U\) is not \({<}\delta\)-supercompact. Then by (4), \(\sup j_U[\delta]\) is regular in \(M_U\). Therefore \(U\) concentrates on regular cardinals. Since \(U\) is weakly normal, \(U\) is closed under decreasing diagonal intersections and every set in \(U\) is stationary. 

Moreover \(U\) is closed under the Mahlo operation by a well-known argument. To see this it suffices to show that for any \(X\in U\), \(j_U(X)\) reflects to \(\sup j_U[\delta]\). We may assume that \(X\subseteq \text{Reg}\). Suppose \(C\subseteq \sup j_U[\delta]\) is club. Let \(\bar C = j_U^{-1}[C]\). Then since \(j_U\) is continuous at all sufficiently large regular cardinals, \(\text{lim}(\bar C)\cap \text{Reg}\subseteq \bar C\). Since \(X\) is stationary, there is some \(\gamma\in \text{lim}(\bar C)\cap X\). But \(\gamma\in \bar C\) since \(X\subseteq \text{Reg}\). Thus \(j_U(\gamma) \in C\cap j_U(X)\). Since \(C\) was arbitrary, it follows that \(j_U(X)\) reflects to \(\sup j_U[\delta]\), as desired.

It follows that \(\delta\) is \({<}\delta^+\)-Mahlo.
\end{proof}
\end{cor}

We do not believe the previous corollary exhausts the supercompactness provable from UA alone. For example, if the least ultrafilter on an inaccessible \(\delta\) fails to be \({<}\delta\)-supercompact, the consequences are truly bizarre:

\begin{prp}[UA]
Suppose \(\delta\) is a strongly inaccessible cardinal such that the least ultrafilter \(U\) on \(\delta\) is not \({<}\delta\)-supercompact. Let \(\kappa = \textsc{crt}(U)\) and \(\delta_* = \sup j_U[\delta]\).
\begin{enumerate}[(1)]
\item \(j_U(\kappa) > \delta\).
\item \(\delta\) is \({<}\delta^+\)-Mahlo.
\item \(U\) is the unique countably complete ultrafilter on \(\delta\) extending the club filter.
\item \(\delta\) is not measurable.
\item For all sufficiently large singular strong limits \(\gamma < \delta\) of cofinality less than \(\kappa\), \(2^\gamma \geq \gamma^{+\kappa}\).
\item For all sufficiently large regular \(\gamma < \delta\), \(\textnormal{Refl}(S^\delta_\gamma)\).
\item For any countably complete \(M_U\)-ultrafilter \(W\) with \(\textsc{sp}(W) < \delta_*\), \(W\in M_U\).
\item In particular \(U\cap M_U\in M_U\).
\end{enumerate}
If \(\delta\) is the least such cardinal, then \(U\cap M_U\) is a normal ultrafilter in \(M_U\).
\end{prp}

We omit the proof, parts of which are similar to \cref{GCHEasy} and \cref{GCHFree} (5).

\begin{conj}
It is provable from \textnormal{ZFC + UA} that the least ultrafilter on a strongly inaccessible cardinal \(\delta\) is \({<}\delta\)-supercompact.
\end{conj}

In our first step toward \cref{LocalTheorem}, we prove a very simple weakening of \cref{LocalTheorem} that serves as a local version of \cref{KetonenTheorem}. 

\begin{lma}[UA]\label{DiscontinuityLemma}
Suppose \(U\) is the least ultrafilter on a regular cardinal \(\delta\). Let \(\kappa\) be the completeness of \(U\). Suppose \(\bar \delta\in [\kappa,\delta]\) is regular and \(\textnormal{Un}_{\bar \delta}\) is nonempty. Then \(U\) is discontinuous at \(\bar \delta\).
\begin{proof}
Assume towards a contradiction that \(\delta\) is the least cardinal at which the lemma fails. Let \(\bar \delta\in [\kappa,\delta]\) witness this. Of course \(\bar \delta \in (\kappa,\delta)\). 

Let \(\bar U\) be the least ultrafilter on \(\bar \delta\). Let \(\bar\kappa\) be the completeness of \(\bar U\). Since \(U\) is continuous at \(\bar \delta\), \(U\I \bar U\) by \cref{ZeroInternal}. Since \(\bar \delta < \delta\), \(\bar U\I U\) by \cref{ZeroInternal}. Thus by \cref{InternalCommutativity}, \(U\) and \(\bar U\) commute. It follows in particular that \(j_{\bar U}(\kappa) = \kappa\) and \(j_U(\bar \kappa) = \kappa\).

Since \(\bar \kappa < \delta\) is a strong limit cardinal fixed by \(U\), \(\bar \kappa < \kappa\) by \cref{Kunen}.

Now \(\bar U\) is the least ultrafilter on the regular cardinal \(\bar \delta\), but \(\bar U\) fixes the measurable cardinal \(\kappa \in [\bar \kappa,\bar \delta]\). This contradicts the minimality of \(\delta\).
\end{proof}
\end{lma}

The second step toward \cref{LocalTheorem} is a result that looks like \cref{EasyRegularity} but is really much more complicated.

\begin{thm}[UA]\label{RegularCor}
Let \(U\) be the least ultrafilter on a regular cardinal \(\delta\). Let \(\kappa\) be its completeness. Suppose \(W\) is a countably complete ultrafilter with \(j_W(\kappa) > \kappa\) and \(U\I W\). Then \(\kappa\) is closed under ultrapowers and \(U\) and \(W\) are \((\kappa,\delta)\)-regular.
\end{thm}

To simplify the notation below, we note that to prove \cref{RegularCor}, it suffices to prove the following lemma:

\begin{lma}[UA]\label{HardRegularity}
Suppose \(U\) is the least ultrafilter on a regular cardinal \(\delta\). Let \(\kappa\) be its completeness. Suppose there is a countably complete ultrafilter \(W\) satisfying \(j_W(\kappa) > \kappa\) and \(U\I W\). Then \(\kappa\) is closed under ultrapowers.
\end{lma}

\begin{proof}[Proof of \cref{RegularCor} given \cref{HardRegularity}]
By \cref{HardRegularity}, \(\kappa\) is closed under ultrapowers. Therefore \(\kappa\) itself is the least ordinal \(\bar \kappa\) such that \(j_Z(\bar \kappa) > \kappa\) for some \(Z\in \textnormal{Un}\). Now an application of \cref{EasyRegularity} implies the theorem.
\end{proof}

\begin{proof}[Proof of \cref{HardRegularity}]
Assume towards a contradiction that \(\bar \kappa < \kappa\) is the least ordinal such that for some \(Z\), \(j_Z(\bar \kappa) > \kappa\). Since \(\kappa\) is measurable, \(\bar \kappa\) is also the least ordinal mapped arbitrarily high below \(\kappa\).

Denote by \(W\) the \(\swo\)-least uniform countably complete ultrafilter such that \(j_W(\kappa) > \kappa\) and \(U\I W\).

\begin{clm}\label{InternalClm}
\(\textnormal{Un}_{{<}\kappa}\I W\).
\begin{proof}[Proof of \cref{InternalClm}]
Fix \(Z\in \textnormal{Un}_{{<}\kappa}\). By Kunen's \cref{CommutingUltrapowers}, \(j_Z(j_U) = j_U\restriction M_Z\) and in particular \(U\I Z\). Since \(U \I Z\) and \(U\I W\), we have \(U\I Z\vee W\) by \cref{CanonicalInternal}. This implies \(s_Z(U) \I^{M_Z} t_Z(W)\); that is, \[j_Z(U)\I^{M_Z} t_Z(W)\] Moreover \(j^{M_Z}_{t_Z(W)}(j_Z(\kappa)) = j^{M_W}_{t_W(Z)}(j_W(\kappa)) > \kappa = j_Z(\kappa)\), or more briefly: \[j^{M_Z}_{t_Z(W)}(j_Z(\kappa)) > j_Z(\kappa)\] Therefore in \(M_Z\), \(t_Z(W)\) satisfies the conditions for which \(W\) was minimized with their parameters moved by \(j_Z\). Hence \(j_Z(W)\wo^{M_Z} t_Z(W)\). By the definition of translation functions, \(t_Z(W)\wo^{M_Z} j_Z(W)\), so \(t_Z(W) = j_Z(W)\). By \cref{jInternal}, \(Z\I W\).
\end{proof}
\end{clm}

\begin{clm} \(j_W(\bar \kappa) < \kappa\).\label{KeyClm}\end{clm}
\begin{proof}[Proof of \cref{KeyClm}]
If \(j_W(\bar \kappa) > \kappa\) then \(\bar \kappa= \kappa\) by \cref{EasyRegularity}, a contradiction. Thus \(j_W(\bar \kappa) \leq \kappa\), and we must show the inequality is strict. 

By \cref{InternalClm}, \(\textnormal{Un}_{{<}\kappa}\I W\). Using \cref{exponential}, it follows that \(\bar \kappa\) is mapped arbitrarily high below \(\kappa\) by ultrapowers in \(M_W\). If \(j_W(\bar \kappa) = \kappa\), then by elementarity, in \(V\) some ordinal \(\alpha < \bar \kappa\) is mapped arbitrarily high below \(\bar \kappa\). But \(\alpha\) is then mapped arbitrarily high below \(\kappa\), contradicting the minimality of \(\bar \kappa\).
\end{proof}

We now break into two cases, based on whether or not \(\kappa\) is a limit of regular cardinals carrying uniform countably complete ultrafilters.

\begin{case} \(\kappa\) is {\it not} a limit of regular cardinals that carry uniform countably complete ultrafilters. \label{smallcase}\end{case}
\begin{proof}[Proof in \cref{smallcase}]
We claim \(\kappa = \delta\). Otherwise the normal ultrafilter on \(\kappa\) derived from \(U\) belongs to \(M_U\) by \cref{ZeroInternal} and so \(\kappa\) is easily a limit of measurable cardinals and much more.

We can also show in this case that \(W\) is discontinuous at \(\kappa\): otherwise since \(\delta = \kappa\), \(W\I U\), so \(W\) and \(U\) commute by \cref{InternalCommutativity}. This means \(j_W(j_U) = j_U\restriction M_W\), which contradicts that \(j_W(\textsc{crt}(U)) \neq \textsc{crt}(U)\). 

Since \(U\I W\), \(t_W(U) = s_W(U)\) by \cref{InternalTranslations}. Therefore \(\textsc{sp}(t_W(U)) = \sup j_W[\kappa]\). Since \(W\) is discontinuous at \(\kappa\), we conclude \(\textsc{sp}(t_W(U)) < j_W(\kappa)\). Since \(j_W(\kappa)\) is inaccessible, we can therefore find a fixed point \(\xi\) of \(t_W(U)\) above \(\kappa\) and below \(j_W(\kappa)\). 

Let \(Z\) be \(\swo^{M_W}\)-least such that \(j^{M_W}_Z(j_W(\bar \kappa)) > \xi\), which exists since in \(M_W\), \(j_W(\bar \kappa)\) is mapped arbitrarily high below \(j_W(\kappa)\) by elementarity. By \cref{FixedPointInternal}, we have \(t_W(U)\I Z\). 

Since \(j_W(\bar\kappa)\) is least mapped above \(j_W(\kappa)\) in \(M_W\), \(j_W(\bar \kappa)\) is closed under ultrapowers in \(M_W\), and hence is least mapped above \(\kappa\) in \(M_W\). Thus \(Z\) witnesses the hypotheses of \cref{EasyRegularity} in \(M_W\). It follows that \(j_W(\bar \kappa) = \textsc{crt}(t_W(U))\). But \(\textsc{crt}(t_W(U)) = \textsc{crt}(U) = \kappa\). This contradicts \cref{KeyClm}.
\end{proof}

\begin{case} \(\kappa\) is a limit of regular cardinals that carry uniform countably complete ultrafilters. \label{bigcase}\end{case}
\begin{proof}[Proof in \cref{bigcase}]
Let \(\delta' = \text{cf}^{M_W}(\sup j_W[\delta])\). Assume towards a contradiction that \(j_W(\bar \kappa)\) is \(\delta'\)-strongly compact in \(M_W\). Then by \cref{GeneralInternal} and \cref{KetonenTheorem}, \(t_W(U)\) has critical point less than or equal to \(j_W(\bar \kappa)\). But \(j_W(\bar \kappa) < \kappa\) while \(\textsc{crt}(t_W(U)) = \textsc{crt}(U) = \kappa\), a contradiction. Therefore our assumption was false, so \(j_W(\bar\kappa)\) is not \(\delta'\)-strongly compact in \(M_W\).

Let \(\iota_0 = \kappa\) and for each \(n < \omega\), let \(\iota_{n+1} = j_W(\iota_n)\). 

\begin{clm} For all \(n < \omega\):
\begin{enumerate}[(1)]
\item  \(U\) is \({<}\iota_n\)-supercompact.
\item \(\bar \kappa\) is \({<}\iota_n\)-supercompact.
\item \(W\) is \({<}\iota_n\)-supercompact.
\item \(\iota_n\) is strongly inaccessible.
\end{enumerate}
\label{induction}
\end{clm}
\begin{proof}[Proof of \cref{induction}]
The proof is by induction. 

We begin with the case \(n = 0\). Since \(U\) is \(\kappa\)-complete, \(U\) is \(\kappa\)-supercompact, which yields (1). Since \(\kappa\) is an inaccessible limit of regular cardinals carrying uniform countably complete ultrafilters, \cref{GlobalLeastSupercompact} applied in \(V_\kappa\) implies that \(\bar \kappa\) is \({<}\kappa\)-supercompact, which yields (2). Since \(\textnormal{Un}_{<\kappa}\I W\) by  \cref{InternalClm} and since \(\textsc{crt}(W) \geq \bar \kappa\) by \cref{ClosedCompleteness}, it follows from \cref{SupercompactTransfer} that \(W\) is \({<}\kappa\)-supercompact, which yields (3). Since \(\kappa\) is measurable, \(\kappa\) is inaccessible, which yields (4).

Suppose the claim holds when \(n = k\), and we prove it is true for \(n = k + 1\).

By elementarity, \(j_W(\bar \kappa)\) is \({<}\iota_{k+1}\)-supercompact in \(M_W\). Since \(\bar \kappa\) is not \(\delta'\)-strongly compact in \(M_W\), \(\iota_{k+1} < \delta'\). For any \(M_W\)-regular \(\gamma < \iota_{k+1}\), fix \(Z\in \textnormal{Un}^{M_W}_\gamma\) witnessing \(j_W(\bar \kappa)\) is \(\gamma\)-supercompact. Since \(\gamma < \delta'\), \(Z\I t_W(U)\) by \cref{ZeroInternal}. Since \(\textsc{crt}(Z) = j_W(\bar\kappa) < \kappa = \textsc{crt}(t_W(U))\), by \cref{SupercompactTransfer}, this implies \(t_W(U)\) is \(\gamma\)-supercompact in \(M_W\). Since \(\iota_{k+1}\) is strongly inaccessible in \(M_W\), it follows that \(t_W(U)\) is \({<}\iota_{k+1}\)-supercompact in \(M_W\). 

Thus for all \(\xi < \iota_{k+1}\), \(j^{M_W}_{t_W(U)}\restriction \xi \in M^{M_W}_{t_W(U)}\subseteq M_U\). Since \(j^{M_W}_{t_W(U)}\restriction \text{Ord} = j_U\restriction \text{Ord}\), it follows that \(U\) is \({<}\iota_{k+1}\)-supercompact. This shows (1). 

Since \(\bar\kappa\) is \({<}\kappa\)-supercompact, the fact that \(U\) is \({<}\iota_{k+1}\)-supercompact implies that \(\bar \kappa\) is \({<}\iota_{k+1}\)-supercompact. This shows (2).

Finally since \(W\) is \({<}\kappa\)-supercompact and \(U\I W\), \cref{SupercompactTransfer} implies that \(W\) is \({<}\iota_{k+1}\)-supercompact as well. This shows (3). 

It follows that \(\iota_{k+1}\) is inaccessible: \(\iota_{k+1} = j_W(\iota_k)\) is inaccessible in \(M_W\) by elementarity and our inductive hypothesis. This is absolute to \(V\) since \(\bigcup_{\gamma < \iota_{k+1}} P(\gamma)\subseteq M_W\) by the supercompactness of \(W\). This shows (4).
\end{proof}

Let \(\lambda = \sup_n \iota_n\). Then \(W\) is \({<}\lambda\)-supercompact by \cref{induction}. Since \(M_W\) is closed under countable sequences, \(W\) is \(\lambda\)-supercompact. But \(j_W(\lambda) = \lambda\) and \(\textsc{crt}(W) < \lambda\). This contradicts Kunen's inconsistency theorem.
\end{proof}
We therefore reach contradictions in \cref{smallcase} and \cref{bigcase}. It follows that our assumption was false, which completes the proof of \cref{HardRegularity}.
\end{proof}

In the next theorem we obtain the hypotheses of \cref{HardRegularity} from a large cardinal axiom that appears to be {\it one ultrafilter away from optimal}.

\begin{thm}[UA]\label{QPoints}
Suppose \(\delta\) is a regular cardinal. Suppose there are distinct countably complete ultrafilters extending the club filter on \(\delta\). Let \(U_0 \swo U_1\) be the \(\swo\)-least two. 

Let \(\kappa \leq \delta\) be the least ordinal such that for some countably complete ultrafilter \(Z\), \(j_Z(\kappa) > \delta\). Then the following hold:
\begin{enumerate}[(1)]
\item \(U_0\) is \(\kappa\)-complete and \((\kappa,\delta)\)-regular.
\item \(U_1\) is the weakly normal ultrafilter of a normal fine \(\kappa\)-complete ultrafilter on \(P_\kappa(\delta)\), and \(U_0 \mo U_1\).
\end{enumerate}
\end{thm}

(1) has strong consequences just short of supercompactness by \cref{ZeroTheorem}. We will use this below to show that (1) implies (2).

For the proof of \cref{QPoints}, we use a lemma that is probably part of the folklore, at least in the special case that \(\mathcal F\) is a normal ultrafilter on \(\lambda\). Here we just use the case that \(\mathcal F\) is the club filter on \(\lambda\), but we prove the lemma at a higher level of generality since we need the more general version in \cref{SecondStronglyCompact}. 

\begin{lma}\label{NormalGeneration}
Suppose \(\mathcal F\) is a normal fine filter on \(P(\lambda)\). Suppose \(D\) is a countably complete ultrafilter on \(\lambda\). Let \(B = \{\sigma\in P^{M_D}(j_D(\lambda)) : [\textnormal{id}]_D\in \sigma\}\). Then \(j_D[\mathcal F]\cup \{B\}\) generates \(j_D(\mathcal F)\).
\begin{proof}
Suppose \(X\in j_D(\mathcal F)\). Then \(X = j_D(\langle X_\alpha : \alpha < \lambda\rangle)([\text{id}]_D)\) with \(X_\alpha\in \mathcal F\) for all \(\alpha < \lambda\). By normality \(\triangle_{\alpha < \lambda}X_\alpha\in \mathcal F\). We claim \(j_D(\triangle_{\alpha < \lambda}X_\alpha) \cap B\subseteq X\). Suppose \(\sigma\in j_D(\triangle_{\alpha < \lambda}X_\alpha) \cap B\). Since \(\sigma\in j_D(\triangle_{\alpha < \lambda}X_\alpha)\), we have \(\sigma\in \bigcap_{\alpha\in \sigma} X'_\alpha\) where \(\langle X'_\alpha : \alpha < j_D(\lambda)\rangle = j_D(\langle X_\alpha : \alpha < \lambda\rangle)\). Since \(\sigma \in B\), \([\text{id}]_D\in \sigma\), so \[\sigma \in X'_{[\text{id}]_D} = j_D(\langle X_\alpha : \alpha < \lambda\rangle)([\text{id}]_D) = X\] as desired.
\end{proof}
\end{lma}

\begin{proof}[Proof of \cref{QPoints}]
We first show that for all \(W\swo U_1\), \(W\I U_1\). Suppose \(W\swo U_1\). Then by \cref{NormalGeneration} with \(\mathcal F\) the club filter on \(\delta\), \(t_{W}(U_1)\) extends the club filter on \(j_{W}(\delta)\). Of course \(t_{W}(U_1)\neq j_{W}(U_0)\), since \(W^-(j_{W}(U_0)) = U_0\) while \(W^-(t_{W}(U_1)) = U_1\). So \(j_W(U_0)\swo^{M_W} t_W(U_1)\) and therefore \(j_{W}(U_1)\wo^{M_{W}} t_{W}(U_1)\) since in \(M_W\), \(j_W(U_1)\) is the \(\swo^{M_W}\)-least ultrafilter extending the club filter on \(j_W(\delta)\) apart from \(j_W(U_0)\). By definition \(t_W(U_1)\wo^{M_W}j_W(U_1)\), so \(t_W(U_1) = j_W(U_1)\). This implies \(W \I U_1\) by \cref{jInternal}. 

Since in particular \(U_0\I U_1\), it follows that \(\sup j_{U_1}[\delta]\) carries a weakly normal ultrafilter in \(M_{U_1}\). So the weakly normal ultrafilter on \(\delta\) derived from \(U_1\) is not equal to \(U_0\). Since this derived weakly normal ultrafilter extends the club filter on \(\delta\), it is equal to \(U_1\) by the minimality of \(U_1\). Hence \(U_1\) is weakly normal.

We next show that (1) implies (2). Note that if (1) holds, we may apply \cref{ZeroTheorem} to obtain that \(U_0\) has the tight \(\delta\)-covering property. Note that \(U_1\) is \(\kappa\)-complete since \(\kappa\) is closed under ultrapowers. Since \(U_0\I U_1\), \cref{CompactPropagation} implies that \(U_1\) has the tight covering property at \(\delta\). We repeat the argument of \cref{Strength} to show that \(P(\delta)\subseteq M_{U_1}\). Fix a \(\kappa\)-independent family of subsets of \(\delta\), \(\langle X_\alpha : \alpha < \delta\rangle\in M_{U_1}\), which exists since \(M_{U_1}\) is closed under \(\kappa\)-sequences and \((\delta^{<\kappa})^{M_{U_1}} = \delta\). For any \(A\subseteq \delta\), there is some \(W\leq_\text{RK} U_0\) on \(\delta\) such that \(X_\alpha\in W\) if and only if \(\alpha\in A\). Since \(U_0\I U_1\), we have \(W\I U_1\), and therefore \(A\in M_{U_1}\). It follows that \(P(\delta)\subseteq M_{U_1}\), which combined with the tight covering property implies that \(U_1\) is \(\delta\)-supercompact.

Note that \(j_{U_1}(\kappa) > \kappa\): otherwise \(\kappa\) is \({<}j_{U_1}(\delta)\)-supercompact in \(M_{U_1}\), and this implies that there are many weakly normal ultrafilters \(W\) on \(\delta\) in \(M_{U_1}\), and these are truly weakly normal and internal to \(U_1\) since \(U_1\) is \(\delta\)-supercompact, and this contradicts the \(\swo\)-minimality of \(U_1\) since the seed order extends the internal relation on \(\textnormal{Un}_\delta\). Now \(j_{U_1}(\kappa) > \delta\) since otherwise \(\kappa\) is huge, which contradicts that \(\kappa\) is closed under ultrapowers. Thus \(U_1\) is the weakly normal ultrafilter of a normal fine ultrafilter on \(P_\kappa(\delta)\). Finally \(U_0\mo U_1\) since \(U_0\I U_1\) and \(U_1\) is \(\delta\)-supercompact.

Thus (1) implies (2).

We finally prove (1). Let \(\kappa_0 = \textsc{crt}(U_0)\). If \(j_{U_1}(\kappa_0) > \kappa_0\), then (1) follows from \cref{HardRegularity}. Assume instead \(j_{U_1}(\kappa_0) = \kappa_0\). Then in \(M_{U_1}\), \(j_{U_1}(U_0)\) has critical point \(\kappa_0\). Let \(\delta' = \text{cf}^{M_{U_1}}(\sup j_{U_1}[\delta])\), and let \(U_*\) denote the least ultrafilter on \(\delta'\) as computed in \(M_{U_1}\). By \cref{Trivia}, \(U_*\) is equivalent in \(M_{U_1}\) to \(t_{U_1}(U_0) = s_{U_1}(U_0)\), and hence \(j^{M_{U_1}}_U = j_{U_0}\restriction M_{U_1}\).

Since \(\delta' < j_{U_1}(\delta)\), \cref{ZeroInternal} implies \(U_*\I^{M_{U_1}} j_{U_1}(U_0)\). Hence in \(M_{U_1}\), \(W = j_{U_1}(U_0)\) witnesses the hypothesis of \cref{RegularCor} with respect to \(U\).  It follows that in \(M_{U_1}\), \(\kappa_0\) is closed under ultrapowers and \(U_*\) is \((\kappa_0,\delta')\)-regular. Since for all \(W\swo U_1\), \(W\I U_1\), this implies \(\kappa_0\) is closed under ultrapowers in \(V\). Moreover, since \(U_*\) is \((\kappa_0,\delta')\)-regular in \(M_{U_1}\) and \(j^{M_{U_1}}_{U_*} = j_{U_0}\restriction M_{U_1}\), \(U_0\) is discontinuous at every regular cardinal in the interval \([\kappa_0,\delta]\). Hence \(U_0\) is \((\kappa_0,\delta)\)-regular by \cref{KetonenTheorem}. Therefore (1) holds in this case as well.
\end{proof}

As a corollary, if a least ultrafilter \(U\) interacts nontrivially with an ultrafilter above it, then \(U\) is well-behaved in the sense of \cref{ZeroTheorem}:

\begin{cor}[UA]\label{InternalCompact}
Let \(U\) be the least ultrafilter on a regular cardinal \(\delta\), and let \(\kappa\) be its completeness. Suppose there is a countably complete ultrafilter that is neither divisible by \(U\) nor internal to \(U\). Then \(\kappa\) is \(\delta\)-supercompact and closed under ultrapowers.
\begin{proof}
Let \(W\) be such an ultrafilter. If \(W\) is continuous at \(\delta\), then \(W\I U\) by \cref{ZeroInternal}. Consider the weakly normal ultrafilter \(D\) on \(\delta\) derived from \(W\). To finish, it suffices by \cref{QPoints} to show that \(D\neq U\). But if \(D = U\), then \(U\) divides \(W\), this time by \cref{GeneralInternal}.
\end{proof}
\end{cor}

Using \cref{DiscontinuityLemma} and \cref{InternalCompact}, we can prove \cref{LocalTheorem}.

\begin{proof}[Proof of \cref{LocalTheorem}]
Let \(W\) be the least ultrafilter on \(\gamma\). Let \(\kappa\) be its completeness. Clearly \(W\) does not divide \(U\). By \cref{ZeroInternal}, \(W\I U\). By \cref{DiscontinuityLemma}, \(U\) is discontinuous at \(\gamma\), so since \(W\) is \(0\)-order with respect to \(\gamma\), \(U\not \I W\).

Now by \cref{InternalCompact}, \(\kappa\) is \(\gamma\)-supercompact and closed under ultrapowers. By \cref{ClosedCompleteness}, \(U\) is \(\kappa\)-complete. Since \(\textnormal{Un}_{\gamma}\I U\) by \cref{ZeroInternal}, \(U\) is \(\gamma\)-supercompact by \cref{SupercompactTransfer}.
\end{proof}

\subsection{Some cardinal arithmetic}
Part of the reason for proving these theorems with as few cardinal arithmetic assumptions as we can manage is that it allows us to improve the result that UA implies GCH.

We begin by mentioning a result that suffices to prove GCH above a supercompact.

\begin{defn}
We say \(\gamma\) is \(M\)-commanded if every \(A\subseteq P(\gamma)\) belongs to \(M_W\) for some  \(W\in \textnormal{Un}_{\leq\gamma}\).
\end{defn}

Of course \(M\)-command follows from supercompactness by an argument due to Solovay:

\begin{thm}[Solovay]\label{SolovayMuMeasure}
Suppose that \(U\) is \(\lambda\)-supercompact, \(\textnormal{cf}(\lambda)\geq \textsc{crt}(U)\), and the pre-normal ultrafilter \(D\) on \(\lambda\) derived from \(U\) belongs to \(M_U\). Then \(\lambda\) is \(M\)-commanded in \(M_U\).
\begin{proof}
Suppose not. Let \(k : M_D\to M_U\) be the factor embedding. Since \(k(\lambda) = \lambda\), \(\lambda\) is not \(M\)-commanded in \(M_D\). Take \(A\subseteq P(\lambda)\) with \(A\) in \(M_D\) such that for no \(Z\in  \textnormal{Un}^{M_D}_{\leq\lambda}\) does \(A\) belong to \(M^{M_D}_Z\). Then since \(k(A) = A\), for no \(Z\in \textnormal{Un}^{M_U}_{\leq\lambda}\) does \(A\) belong to \(M^{M_U}_Z\). But \(A\in M^{M_U}_D\): by Kunen's inconsistency theorem there is some inaccessible \(\kappa \leq \delta\) with \(j_D(\kappa) > \delta\), and \[A\in V^{M_D}_{j_D(\kappa)} =  j_D(V_{\kappa}) = j_D(V^{M_U}_{\kappa}) = V^{M^{M_U}_D}_{j_D(\kappa)}\] Since \(D\in \textnormal{Un}^{M_U}_{\leq\lambda}\), this is a contradiction.
\end{proof}
\end{thm}

The first proof of GCH above a supercompact from UA and large cardinals was based on the following fact.

\begin{thm}[UA]
Suppose \(\gamma\) is \(M\)-commanded. Suppose \(\delta > \gamma\) is a regular cardinal that carries a uniform countably complete ultrafilter. Then \(2^\gamma < 2^{\delta}\).
\begin{proof}[Sketch]
Let \(U\) be the least ultrafilter on \(\delta\). Assume towards a contradiction that \(2^\gamma = 2^{\delta}\). We can then code \(U\) by \(A\subseteq P(\gamma)\), so fix \(W\in \textnormal{Un}_{\leq\gamma}\) such that \(U\mo W\). Note that \(W\I U\) by \cref{ZeroInternal} and similarly \(U\I^{M_W} j_W(U)\) since \(j_W(\delta) > \delta\) by Kunen's inconsistency theorem. This implies \(j^{M_W}_U\restriction \text{Ord}\) is amenable to \(M_U\). In particular, \(j_U\restriction \gamma = j^{M_W}_U\restriction\gamma \in M_U\) so \(U\) is \(\gamma\)-supercompact. But then \(W\I U\) implies \(W\mo U\), contradicting the strictness of the Mitchell order.
\end{proof}
\end{thm}

Buried in the reductio was the first hint that UA might prove the supercompactness of least ultrafilters.

In fact, using the theory of the internal relation developed here, one can actually prove the following theorem:
\begin{thm}[UA]\label{UltrafilterBound}
Suppose \(\delta\) is an infinite cardinal. Suppose \(W\in \textnormal{Un}_{\delta}\) and \(\gamma\) is such that \(\textnormal{Un}_{\leq\gamma}\I W\). Then \(|\textnormal{Un}_{\leq\gamma}| \leq 2^\delta\).
\begin{proof}
Assume by induction that the theorem is true for all \(\alpha < \gamma\). Assume towards a contradiction that \(|\textnormal{Un}_{\leq\gamma}| \geq (2^{\delta})^+\). For all \(\alpha < \gamma\), \(|\textnormal{Un}_{\leq\xi}| \leq 2^\delta\), and so since we must have \(\gamma < \delta\),  \(|\textnormal{Un}_{<\gamma}| \leq \gamma\cdot 2^\delta = 2^\delta\). Hence \(|\textnormal{Un}_\gamma| \geq (2^{\delta})^+\).

Note that for any \(U\in \textnormal{Un}_\gamma\), if \(D\swo U\) then there is a sequence \(\langle D_\alpha : \alpha < \gamma\rangle\) such that \(\textsc{sp}(D_\alpha)\leq \alpha\) for all \(\alpha < \gamma\) with the following property: \[D = \left\{X\subseteq \textsc{sp}(D) : \{\alpha <\gamma: X\cap \textsc{sp}(D_\alpha)\in D_\alpha\right\}\in U\}\]
(Using Los's theorem this just says \(D = U^-([D_\alpha]_U)\).) 
Thus \(U\) has at most \(\prod_{\alpha < \gamma} |\textnormal{Un}_{\leq\alpha}|\leq (2^\delta)^\gamma = 2^\delta\) predecessors in the seed order. It follows that the seed order on \(\textnormal{Un}_\gamma\) has order type exactly \((2^\delta)^+\): it is a wellorder of cardinality \((2^\delta)^+\) with initial segments of cardinality \(2^\delta\).

For \(U\in \textnormal{Un}\), let \(|U|_S\) denote the rank of \(U\) in the seed order. For any \(U,D\in \textnormal{Un}\), we claim \(|U|_S \leq |t_D(U)|^{M_D}_S\). To see this, note that \(t_D\) maps the \(\swo\)-predecessors of \(U\) into the \(\swo^{M_D}\)-predecessors of \(t_D(U)\), preserving the seed order, by \cref{op}. In particular if \(U\) is nonprincipal, then \[|U|_S \leq |t_U(U)|^{M_U}_S < |j_U(U)|^{M_U}_S = j_U(|U|_S)\] In other words, {\it every nonprincipal ultrafilter moves its own seed rank}. (More generally if \(j_D(|U|_S) = |U|_S\) then \(D\I U\). In many cases, for example for ultrafilters extending the club filter, we can also show the converse.)

Let \(\eta = |U|_S\) where \(U\) is least on \(\gamma\). If \(\eta \leq |D|_S < (2^\delta)^+\), then \(D\) is a uniform ultrafilter on \(\gamma\), and in particular \(D\) is nonprincipal, so \(D\) moves \(|D|_S\). It follows that every \(\alpha\) such that \(\eta \leq \alpha < (2^\delta)^+\) is moved by some \(D\in \textnormal{Un}_\gamma\). (We remark that this hypothesis can be obtained in ZFC from \(\delta\)-compactness, and this is due to Kunen.)

Note that \(j_W((2^\delta)^+) = (2^\delta)^+\) since \(\textsc{sp}(W) = \delta\). Therefore \(W\) has an \(\omega\)-club of fixed points below \((2^\delta)^+\). Moreover since \(|\textnormal{Un}_{<\gamma}| \leq 2^\delta\) and each \(D\in \textnormal{Un}_{<\gamma}\) fixes \(\delta\), the set of common fixed points of elements of \(\textnormal{Un}_{<\gamma}\) is \(\omega\)-club in \((2^\delta)^+\).  Let \(\xi \in [\eta, (2^\delta)^+)\) be fixed by \(W\) and by all ultrafilters in \(\textnormal{Un}_{<\gamma}\).

Let \(\bar W\) be the least ultrafilter moving \(\xi\). Then \(\bar W\in \textnormal{Un}_\gamma\). By \cref{FixedPointInternal}, \(W\I \bar W\) and \(\textnormal{Un}_{<\gamma}\I \bar W\). By assumption, \(\bar W\I W\). So by \cref{InternalCommutativity}, \(W\) and \(\bar W\) commute.

Let \(\bar \kappa = \textsc{crt}(\bar W)\) and \(\kappa = \textsc{crt}(W)\). Commutativity implies \(j_W(\bar \kappa) = \bar \kappa\) and \(j_{\bar W}(\kappa) = \kappa\). In particular \(\bar \kappa \neq \kappa\).

Also \(\kappa \neq \gamma\) since \(j_{\bar W}\) fixes \(\kappa\) but not \(\gamma\). We cannot have \(\gamma < \kappa\) since \(|\textnormal{Un}_\gamma| > 2^\delta > \kappa\) while \(\kappa\) is strongly inaccessible. It follows that \(\kappa < \gamma\). 

Note that \(\kappa\) is a strong limit cardinal fixed by \(j_{\bar W}\) and \(\textnormal{Un}_{<\kappa}\subseteq \textnormal{Un}_{<\gamma}\I \bar W\). By \cref{Kunen}, it follows that \(\kappa \leq \bar \kappa\).

Similarly, \(\bar \kappa\) is a strong limit cardinal fixed by \(j_{W}\) and \(\textnormal{Un}_{<\bar \kappa}\subseteq \textnormal{Un}_{\leq\gamma}\I W\). By \cref{Kunen}, it follows that \(\bar \kappa \leq \kappa\).

Since \(\bar \kappa \neq \kappa\), we have reached a contradiction.
\end{proof}
\end{thm}

The following generalization of \cref{ZeroInternal} to singular cardinals is useful for obtaining the hypotheses of \cref{UltrafilterBound} (for example in \cref{PredecessorBound} and \cref{GCHTheorem}) at singular cardinals. On the other hand, it seems quite possible that the hypotheses of \cref{SingularIPoints} actually imply \(\lambda\) is regular. In any case, the idea behind \cref{SingularIPoints} might lead to a proof of this.

\begin{thm}[UA]\label{SingularIPoints}
Suppose \(\lambda\) is a cardinal that carries a strongly uniform ultrafilter but \(\lambda\) is not a singular limit of cardinals carrying strongly uniform ultrafilters. Suppose \(W\) is the least strongly uniform ultrafilter on \(\lambda\). Then \(\textnormal{Un}_{<\lambda}\I W\).
\begin{proof}
Suppose \(D\in \textnormal{Un}_{<\lambda}\). Then \(t_D(W)\) is equivalent to a strongly uniform ultrafilter on a cardinal \(\lambda_*\geq \sup j_D[\lambda]\) since \(W\) divides \(D\oplus t_D(W)\). On the other hand \(\lambda_*\leq j_D(\lambda)\) since \(t_D(W)\wo^{M_D} j_D(W)\). 

Assume \(\lambda_* < j_D(\lambda)\) towards a contradiction. Then \(D\) is discontinuous at \(\lambda\), so \(\lambda\) is singular. Moreover by the usual reflection argument, \(\lambda\) is a limit of cardinals carrying strongly uniform ultrafilters, contradicting the minimality of \(\lambda\). Therefore our assumption was false, and \(\lambda_* = j_D(\lambda)\). 

Thus \(t_D(W)\) is equivalent to a strongly uniform ultrafilter on \(j_D(\lambda)\). Since \(t_D(W)\wo^{M_D} j_D(W)\) and \(j_D(W)\) is the least strongly uniform ultrafilter on \(j_D(\lambda)\), \(t_D(W) = j_D(W)\). Thus \(D\I W\) by \cref{jInternal}.
\end{proof}
\end{thm}

The following corollary generalizes a well-known fact regarding the Mitchell order on normal ultrafilters. A direct generalization of this that does not use UA will only show that \(U\in \textnormal{Un}_\delta\) has \(2^{2^{<\delta}}\) many predecessors in the seed order. This is for good reason: suppose \(2^\delta = 2^{\delta^+}\), and \(|\text{Un}_\delta| > 2^\delta\), and \(\text{Un}_{\delta^+}\neq \emptyset\), a hypothesis that is easily proved consistent with ZFC from the existence of a cardinal \(\delta\) that is \(2^\delta\)-supercompact. Then by \cref{SpaceLemma}, any \(W\in \textnormal{Un}_{\delta^+}\) lies above every \(U\in \textnormal{Un}_\delta\) in the seed order, and hence has more than \(2^\delta\) predecessors. So \cref{PredecessorBound} is not provable in ZFC. 

\begin{cor}[UA]\label{PredecessorBound}
Suppose \(\delta\) is a cardinal that carries a strongly uniform ultrafilter. Then any \(U\in \textnormal{Un}_\delta\) has at most \(2^\delta\) predecessors in the seed order.
\begin{proof}
Assume by induction that the theorem holds at all cardinals below \(\delta\).

If \(\delta\) is not a singular limit of cardinals carrying strongly uniform ultrafilters, then using \cref{SingularIPoints}, the least strongly uniform \(W\) on \(\delta\) satisfies the hypotheses of \cref{UltrafilterBound} with respect to any \(\alpha < \delta\), and therefore \(|\textnormal{Un}_\alpha|\leq 2^\delta\) for all \(\alpha < \delta\). Suppose instead that the set \(A\subseteq \delta\) of cardinals carrying strongly uniform ultrafilters is unbounded in \(\delta\). Let \(B\) be the set of successor elements of \(A\). Then every \(\lambda\in B\) satisfies the hypotheses of \cref{UltrafilterBound} by \cref{SingularIPoints}. Now for any \(\alpha \leq \delta\), take \(\lambda\in B\) with \(\alpha < \lambda\), and note that again by \cref{UltrafilterBound}, \(|\textnormal{Un}_\alpha|\leq 2^\lambda \leq 2^\delta\).

In either case, therefore, \(|\text{Un}_{<\delta}|\leq 2^\delta\). By the counting argument from the beginning of the proof of \cref{UltrafilterBound}, it follows that any \(U\in \textnormal{Un}_\delta\) has at most \(\prod_{\alpha < \delta} \textnormal{Un}_{\leq \alpha} \leq 2^\delta\) predecessors.
\end{proof}
\end{cor}

An old observation of Solovay is that the linearity of the Mitchell order on normal ultrafilters on a cardinal \(\kappa\) that carries \(2^{2^\kappa}\) normal ultrafilters implies \(2^{2^\kappa} = 2^{\kappa^+}\). By the previous theorem, one has the following generalization:

\begin{thm}[UA]\label{SolovayBound}
Suppose \(|\textnormal{Un}_\delta| = 2^{2^\delta}\). Then \(2^{2^\delta} = (2^\delta)^+\).
\begin{proof}[Sketch]
Let \(\bar \delta \leq \delta\) be least such that \(|\text{Un}_{\leq\bar \delta}| = 2^{2^\delta}\). Easily \(|\text{Un}_{<\bar \delta}| < 2^{2^\delta}\), so \(|\text{Un}_{\bar \delta}| = 2^{2^\delta}\). If \(\bar \delta\) does not carry a strongly uniform ultrafilter then an easy counting argument implies \(|\text{Un}_{\bar \delta}| \leq [\bar \delta]^{<\bar \delta}\cdot |\text{Un}_{<\bar \delta}| < 2^{2^\delta}\), a contradiction. So \(\bar \delta\) carries a strongly uniform ultrafilter. Therefore by \cref{PredecessorBound}, \(\text{Un}_{\bar \delta}\) is wellordered by the seed order with initial segments of cardinality \(2^{\bar \delta}\). It follows that \(2^{2^\delta} = (2^{\bar \delta})^+\) as desired.
\end{proof}
\end{thm}

In fact \cref{UltrafilterBound} implies much more interesting instances of GCH.

\begin{prp}[UA]\label{OriginalGCH}
Suppose \(\delta\) is a cardinal, \(|\textnormal{Un}_\delta| = 2^{2^\delta}\), and \(|\textnormal{Un}_{\delta^+}| > 2^{\delta^+}\). Let \(\lambda > \delta^+\) be least carrying a strongly uniform ultrafilter. Then \(2^\delta < \lambda\).
\begin{proof}
Since \(\textnormal{Un}_{\leq\delta}\I U\) where \(U\) is least on \(\delta^+\) by \cref{ZeroInternal}, we can apply \cref{UltrafilterBound} to obtain \[2^{2^\delta} \leq |\text{Un}_{\leq\delta}| \leq 2^{\delta^+}\]
Since \(\textnormal{Un}_{<\lambda}\I U\) where \(U\) is the \(\swo\)-least strongly uniform ultrafilter on \(\lambda\) by \cref{SingularIPoints}, we can apply \cref{UltrafilterBound} to obtain 
\[2^{\delta^+} < |\textnormal{Un}_{\delta^+}| \leq 2^\lambda\]
Combining these facts, \(2^{2^\delta} < 2^\lambda\). Hence \(2^\delta < \lambda\).
\end{proof}
\end{prp}

In the case that \(\lambda = \delta^{++}\), we have the following corollary:
\begin{thm}[UA]\label{Suboptimal}
Suppose \(\delta\) is a cardinal, \(|\textnormal{Un}_\delta| = 2^{2^\delta}\), and \(|\textnormal{Un}_{\delta^+}| > 2^{\delta^+}\), and \(\textnormal{Un}_{\delta^{++}}\neq\emptyset\). Then \(2^\delta = \delta^+\).
\end{thm}

A somewhat subtler argument using the techniques of this paper improves the hypotheses above.
\begin{thm}[UA]\label{GCHTheorem}
Suppose \(\delta\) is a cardinal and \(\textnormal{Un}_{(2^\delta)^+}\neq\emptyset\). Then \(2^\delta = \delta^+\).
\begin{proof}
Let \(\delta\) be the least cardinal at which the theorem fails. 

Let \(U\) be the least ultrafilter on \((2^\delta)^+\). Let \(\iota = \text{cf}(2^\delta)\), so \(\iota > \delta\) by Konig's theorem. Note that \(U\) is \(\iota\)-supercompact: if \(U\) is continuous at \(2^\delta\) then \(U\) is \((2^\delta)^+\)-supercompact by the argument of \cref{GCHFree} (3), and if \(U\) is discontinuous at \(2^\delta\) then \(\text{Un}_\iota\neq \emptyset\) so we can appeal to \cref{LocalTheorem} to conclude that \(U\) is \(\iota\)-supercompact.

Assume first that \(\delta\) is a singular cardinal. Then by the \(\iota\)-supercompactness of \(U\), \(\delta\) is a limit of regular cardinals carrying uniform ultrafilters.  By \cref{SolovayMuMeasure}, at all sufficiently large regular \(\bar \delta < \delta\), one has in \(M_U\) the hypotheses of \cref{GCHTheorem} for some \(\lambda < \delta\). Hence \(2^{\bar \delta} < \lambda\). It follows that \(\delta\) is a strong limit cardinal. Therefore by Solovay's theorem \cite{Solovay} on SCH, \(2^\delta = \delta^+\) since \(\delta\) is singular.

We may therefore assume that \(\delta\) is regular. By \cref{SolovayMuMeasure}, \[M_U\vDash |\textnormal{Un}_\gamma| = 2^{2^\gamma}\] for \(\gamma \in \{\delta,\delta^+\}\).

Since \(\delta^{+} < 2^\delta\), if \(U\) is \(\delta^{++}\)-supercompact, then by \cref{SolovayMuMeasure} we have the hypotheses of \cref{Suboptimal} in \(M_U\), so that \(2^\delta = \delta^+\) in \(M_U\), which is absolute to \(V\) since \(P(\delta)\subseteq M_U\). We may therefore assume \(\iota < \delta^{++}\), so that \(\iota = \delta^+\).

If \(2^\delta\) is regular then the fact that \(\iota = \delta^+\) implies the theorem. So we may assume \(2^\delta\) is singular and in particular is a limit cardinal. 

Assume first that \(U\) has the tight covering property at \((2^\delta)^+\). Then since \(j_U(2^\delta) > (2^\delta)^+\), \(U\) is \((2^\delta,(2^\delta)^+)\)-regular, and hence \(U\) is discontinuous at cofinally many regular cardinals below \(2^\delta\). It follows that \(U\) is \((2^\delta)^+\)-supercompact by \cref{GCHFree} (2).

Therefore we may assume that \(U\) does not have the tight covering property at \((2^\delta)^+\).

Assume first that \(((2^\delta)^+)^{M_U} = (2^\delta)^+\). Then the hypotheses of the theorem remain true in \(M_U\), since \(U\cap M_U\in M_U\) by \cref{CoverDichotomy}. Note that \(j_U(\delta) > \delta\) by Kunen's inconsistency theorem, and so \(\delta\) is below the least failure of the theorem in \(M_U\). Therefore \(2^\delta = \delta^+\) in \(M_U\), and so since \(P(\delta)\subseteq M_U\), \(2^\delta = \delta^+\) in \(V\), a contradiction.

Finally assume that \(((2^\delta)^+)^{M_U} < (2^\delta)^+\). Work in \(M_U\). Since \(|\textnormal{Un}_\delta| = 2^{2^\delta}\), we have \(2^{2^\delta} = (2^\delta)^+\). Hence \(2^{\delta^+} \leq (2^\delta)^+\).

Returning to \(V\), since \(P(\delta^+)\subseteq M_U\), the fact that \(M_U\vDash 2^{\delta^+} \leq (2^\delta)^+\) implies \(2^{\delta^+} \leq |((2^\delta)^+)^{M_U}| = 2^\delta\). Therefore \(2^{\delta^+} = 2^\delta\). By Konig's theorem, this contradicts the fact that \(\text{cf}(2^\delta) = \iota = \delta^+\).
\end{proof}
\end{thm}

Assume UA. If \(\kappa\) is \(2^\kappa\)-supercompact, can \(2^\kappa\) be weakly inaccessible? The previous theorem does not give much insight since \((2^\delta)^+\) is a bit of a moving target. Our next theorem rules this out.

\begin{thm}[UA]\label{LocalGCH}
Suppose \(\delta\) is a regular cardinal and \(\delta^{++}\) carries two countably complete ultrafilters extending the club filter. Then \(2^\delta = \delta^+\). 
\begin{proof}
Let \(W\) be the {\it second} such ultrafilter on \(\delta^{++}\). By \cref{QPoints}, \(W\) is \(\delta^{++}\)-supercompact and equivalent to a normal fine \(\kappa\)-complete ultrafilter on \(P_\kappa(\delta^{++})\) for some \(\kappa < \delta^{++}\).

Suppose \(\gamma = \delta\) or \(\gamma = \delta^+\). (The argument that follows works in either case.) Let \(U\) be the pre-normal ultrafilter on \(\gamma\) derived from \(W\); thus \(U\) is equivalent to a normal fine \(\kappa\)-complete ultrafilter on \(P_\kappa(\gamma)\). Note that \(U\mo W\) by the proof of \cref{QPoints}. By an argument due to Solovay, for any \(A\subseteq P(\gamma)\), if \(A\in M_W\), then for some normal fine normal fine \(\kappa\)-complete ultrafilter \(\mathcal D\) on \(P_\kappa(\gamma)\) with \(\mathcal D\mo W\), \(A\in M^{M_W}_\mathcal D\). Such an ultrafilter \(\mathcal D\) satisfies \(|P(P(\gamma)) \cap M^{M_W}_\mathcal D| = 2^\gamma\), and so a simple counting argument implies that \(M_W\) thinks there are \(2^{2^\gamma}\) such ultrafilters. Since each such ultrafilter is equivalent to a unique weakly normal ultrafilter on \(\gamma\) by Solovay's lemma, it follows that \(M_W\) satisfies \(|\textnormal{Un}_\gamma| = 2^{2^\gamma}\).

But now \(M_W\) satisfies the hypotheses of \cref{GCHTheorem} at \(\delta\). Since \(\lambda = \delta^{++}\) carries a uniform ultrafilter in \(M_W\) by \cref{QPoints}, in \(M_W\), \(2^\delta < \delta^{++}\). Hence \(M_W\) thinks \(2^\delta = \delta^+\). Since \(P(\delta)\subseteq M_W\), in fact \(2^\delta = \delta^+\) in \(V\).
\end{proof}
\end{thm}

\begin{cor}[UA]
If \(\delta\) is a regular cardinal and \(2^\delta\) carries a strongly uniform countably complete ultrafilter, then \(2^\delta \leq \delta^{++}\).
\end{cor}

\begin{cor}[UA]\label{LimitGCH}
Suppose \(\lambda\) is a limit of regular cardinals carrying uniform countably complete ultrafilters. Then \textnormal{GCH} holds on a tail below \(\lambda\).
\end{cor}

As a corollary of this, if \(\lambda\) is singular and \(\lambda^+\) carries a two countably complete ultrafilters extending the club filter then \(2^\lambda = \lambda^+\): by \cref{QPoints}, some \(\kappa < \lambda\) is \(\lambda^+\)-supercompact, so by \cref{LimitGCH}, \(\lambda\) is a strong limit cardinal, and hence by Solovay \cite{GCH}, \(2^\lambda = \lambda^+\).

The following theorem appears in \cite{MitchellOrder}:

\begin{thm}[UA]\label{LinearityMO}
Suppose \(\lambda\) satisfies \(2^{<\lambda} = \lambda\). Then the internal relation is linear on normal fine ultrafilters on \(P(\lambda)\).
\end{thm}

As we explained there, this is essentially the same as saying that the Mitchell order is linear. By the results here we can remove the GCH hypothesis in all but two cases. Actually that was the original impetus for this work, although its other applications turned out to be much more interesting. As a corollary of the local GCH results, in many cases the hypothesis \(2^{<\lambda} = \lambda\) can be omitted since it simply follows from the existence of a normal fine ultrafilter on \(P(\lambda)\).

\begin{thm}[UA]
Suppose \(\lambda\) is a limit cardinal, the successor of a singular cardinal, or the double successor of a cardinal of cofinality greater than or equal to the least \(\lambda\)-supercompact cardinal. Then the internal relation is linear on normal fine ultrafilters on \(P(\lambda)\).
\end{thm}

In other words, the only cases we cannot handle by current techniques are successors of inaccessible cardinals and double successors of singulars of small cofinality. 
\section{The Next Strongly Compact Cardinal}\label{SecondStronglyCompact}
The point of this section is to extend the global results of the previous section beyond the least strongly compact cardinal and the local results beyond the least ultrafilters. Our target theorem is the following:
\begin{thm}[UA]\label{Global}
Suppose \(\kappa\) is strongly compact. Then either \(\kappa\) is supercompact or \(\kappa\) is a measurable limit of supercompact cardinals.
\end{thm}
In a sense this is best possible, since every measurable limit of supercompact compact cardinals is strongly compact by a construction of Menas. Better yet, UA + GCH yields a fairly satisfying local analysis of strong compactness that implies that {\it the only way to obtain strong compactness in the absence of supercompactness is Menas's construction.}

\subsection{Factorization into irreducibles}
\begin{defn}
An ultrafilter \(W\) is {\it irreducible} if for all \(U\D W\), either \(U\) is principal or \(U\equiv W\).
\end{defn}

A key structural consequence of UA, which relatively easy compared to the results of this paper, is an ultrafilter factorization theorem:

\begin{thm}[UA]\label{Factorization}
Every countably complete ultrafilter factors as a finite iteration of irreducible ultrafilters.
\begin{proof}
Suppose towards a contradiction that \(U\in \textnormal{Un}\) is \(\swo\)-least at which the theorem fails. Then \(U\) is reducible. Let \(D\) be a proper divisor of \(U\). Then without loss of generality \(D\swo U\). Moreover by \cref{CanonicalComparison}, \(U\equiv D \oplus t_D(U)\). Since \(D\D U\) and \(D\) is nonprincipal, \(D\not \I U\): otherwise \(j_D\restriction \text{Ord}\) is amenable to \(M_U\subseteq M_D\), a contradiction. Hence \(t_D(U) \neq j_D(U)\) by \cref{jInternal}. It follows that \(t_D(U) \swo^{M_D} j_D(U)\) by the definition of translation functions and the linearity of the seed order.

By the minimality of \(U\), \(D\) factors as a finite iteration of irreducible ultrafilters. In \(M_D\), by the minimality of \(j_D(U)\), \(t_D(U)\) factors as a finite iteration of irreducible ultrafilters. Composing the two factorizations yields a factorization of \(U\) as a finite iteration of irreducible ultrafilters, a contradiction.
\end{proof}
\end{thm}

We acknowledge that the argument in \cref{Factorization} that \(t_D(U)\neq j_D(U)\) does not really require using the internal relation, let alone UA. We are ultimately reproving with our notation the standard fact that if \(U\) and \(D\) are countably complete ultrafilters and \(D\) is nonprincipal then \(U\neq D\times U\). Obvious as it may appear, \cref{Factorization} itself is not provable in ZFC by a theorem of Gitik \cite{Gitik}.

The factorization theorem is only useful if one can analyze irreducible ultrafilters. In the next section we will show that they are supercompact up to their spaces. For this we need a different sort of factorization lemma, which states, assuming enough GCH, that {\it ultrafilters factor at their continuity points}. Its proof is the source of most of our cardinal arithmetic woes.

\begin{thm}[UA]\label{ContinuityFactor}
Suppose \(U\) is a \(\kappa\)-complete ultrafilter and \(\delta\) is a regular cardinal such that \(j_U(\delta) = \sup j_U[\delta]\). Let \(\lambda = \textnormal{ot}(\textnormal{Un}_{<\delta},<_S)\). Suppose \(2^{\lambda} < \delta^{+\kappa}\). Then \(U\) factors as \(D\oplus Z\) where \(\textsc{sp}(D) < \delta\) and \(\textsc{crt}(Z) > j_D(\delta)\).
\begin{proof}
Take \(p : \textsc{sp}(U)\to 2^{\lambda}\) such that letting \(W = p_*(U)\) and \(k: M_W\to M_U\) be the factor embedding, \(\textsc{crt}(k) > \lambda\). Note that if \(U\) projects to a uniform ultrafilter on a cardinal between \(\delta\) and \(\delta^{+\kappa}\), then \(U\) is discontinuous at \(\delta\). Hence we may replace \(W\) with an equivalent uniform ultrafilter \(D\) on some cardinal \(\gamma < \delta\). 

Note that the map \(Z\mapsto D\oplus Z\) is an order embedding from \((\textnormal{Un}^{M_D}_{<j_D(\delta)},<^{M_D}_S)\) to \((\textnormal{Un}_{<\delta},<_S)\) by \cref{SumSeed}. Therefore  \(j_D(\lambda) = \text{ot}(\textnormal{Un}^{M_D}_{<j_D(\delta)},<^{M_D}_S) \leq \lambda\), so \(j_D(\lambda) = \lambda\).  Thus since \(\textsc{crt}(k) > \lambda\), \(\text{ot}(\textnormal{Un}^{M_U}_{<j_U(\delta)},<^{M_U}_S) = \lambda\). It follows that
\(\textnormal{Un}^{M_U}_{<j_U(\delta)}\subseteq \text{ran}(k)\). By \cref{Reciprocity}, \(t_U(D) \wo P^{M_U}_{k([\text{id}]_D)}\). But since \(\textsc{sp}(D) = \gamma\), we have \([\text{id}]_D < j_D(\gamma)\), so \(k([\text{id}]_D) < j_D(\delta)\). So \(t_U(D)\in \textnormal{Un}^{M_U}_{<j_U(\delta)}\).

Hence \(t_U(D)\in \text{ran}(k)\). Let \(D_* = k^{-1}(t_U(D))\). It is easy to see that \(D^{-}(D_*) = D\). So by the definition of translation functions, \[D_*\lwo^{M_D} t_D(D) = P^{M_D}_{[\text{id}]_D}\] On the other hand since \(t_U(D) \wo^{M_U} P^{M_U}_{k([\text{id}]_D)}\), \[D_* \wo^{M_D} P^{M_D}_{[\text{id}]_D}\] It follows that \(D_* = P^{M_D}_{[\text{id}]_D}\). Hence \(t_U(D) = P^{M_U}_{k([\text{id}]_D)}\), so \(D\) divides \(U\) and moreover \(k = j^{M_D}_{t_D(U)}\) by \cref{Reciprocity}. 

Taking \(Z = t_D(U)\) we therefore have \(U \equiv D\oplus Z\) where \(\textsc{sp}(D) < \delta\) and \(\textsc{crt}(Z) > \lambda = j_D(\lambda) > j_D(\delta)\). This proves the theorem.
\end{proof}
\end{thm}

\subsection{The structure of irreducible ultrafilters}
The main theorem of this subsection is a structure theorem for irreducible ultrafilters. This structure leads almost immediately to \cref{Global}.

\begin{thm}[UA + GCH]\label{IrredGCH}
Suppose \(W\) is an irreducible strongly uniform ultrafilter on \(\lambda\). Then for every successor cardinal \(\delta \leq \lambda\), \(W\) is \(\delta\)-supercompact.
\end{thm}

Before proving \cref{IrredGCH}, we prove a well-known fact that is a version of Solovay's Lemma.

\begin{lma}[Solovay]\label{LocalSolovay}
There is a formula \(\varphi(x)\) in the language of set theory with an extra predicate \(\dot S\) with the following property:

Suppose \(\delta\) is a regular cardinal and \(\vec S = \langle S_\alpha : \alpha < \delta\rangle\) is a partition of \(S^\delta_\omega\) into stationary sets. Set \(A = \{\sigma : (P_\delta(\delta),\in,\vec S)\vDash \varphi(\sigma)\}\). Then the \(\sup\) function is one-to-one on \(A\) any normal fine ultrafilter on \(P_\delta(\delta)\) contains \(A\).
\begin{proof}
We let \(\varphi(x)\) be the following formula: \[x = \{\alpha \in \text{Ord}: \dot S_\alpha\text{ meets every closed cofinal subset of }\sup (x)\}\] This works by the proof of Solovay's Lemma.
\end{proof}
\end{lma}

\begin{defn}
Suppose \(\delta\) is a regular cardinal and \(\vec S\) is a partition of \(S^\delta_\omega\) into stationary sets. The {\it Solovay set defined from \(\vec S\) at \(\delta\)} is the set \(A\) defined by the formula \(\varphi\) of \cref{LocalSolovay}.
\end{defn}

\begin{proof}[Proof of \cref{IrredGCH}]
We may assume without loss of generality that \(\textsc{crt}(W) < \delta\).

We first prove that \(W\) is discontinuous at \(\delta\); assume towards a contradiction that it is not. Otherwise by \cref{ContinuityFactor}, \(W\equiv D \oplus Z\) where \(D\in \textnormal{Un}_{<\delta}\) and \(Z\in \textnormal{Un}^{M_D}\) is \({\leq}j_D(\delta)\)-complete. Since \(W\) is irreducible, either \(D\) or \(Z\) is principal. Since \(W\) is strongly uniform on \(\lambda \geq \delta\), \(D\) is principal. Thus \(\textsc{crt}(W) > \delta\), a contradiction.

Let \(U\) be the least ultrafilter on \(\delta\). Note that \(W\) is neither divisible by \(U\) nor internal to \(U\) (since \(W\) is discontinuous at \(\delta\)), so by \cref{InternalCompact}, \(\textsc{crt}(U)\) is \(\delta\)-supercompact. By \cref{ZeroTheorem}, since \(\delta\) is a successor, \(U\) is \(\delta\)-supercompact. 

Let \[W_* = t_U(W)\] We will show \(\textsc{crt}(W_*) > \delta\). This implies the theorem: since \(\text{Ord}^\delta\subseteq M_U\), \(\textsc{crt}(W_*) > \delta\) implies \(\text{Ord}^\delta \subseteq M^{M_U}_{W_*}\subseteq M_W\), which implies \(W\) is \(\delta\)-supercompact, as desired. 

Since \(\delta\) carries no uniform ultrafilters in \(M_U\), by \cref{ContinuityFactor}, \(W_*\) factors in \(M_U\) as \(D\oplus Z\) for \(D\in \textnormal{Un}^{M_U}_{{<}\delta}\) and \(Z\in \textnormal{Un}^{(M_D)^{M_U}}\) with \(\textsc{crt}(Z) > j_D(\delta)\). We will prove \(D\) is principal, and hence \(\textsc{crt}(W_*) > \delta\). Assume towards a contradiction that \(D\) is nonprincipal.

Note that \(D\) is a (total) ultrafilter since \(U\) is \(\delta\)-supercompact. Moreover \(D\I U\), so \(t_U(D) = s_U(D)\) divides \(W_*\) in \(M_U\). Let \[U_* = t_W(U)\] Since \(t_U(D)\) divides \(W_*\) in \(M_U\), \(t_W(D)\) divides \(U_*\) in \(M_W\) by \cref{Reciprocity2}. 

Let \[\gamma = \text{cf}^{M_W}(\sup j_W[\delta])\]  By \cref{GeneralInternal}, \(U_*\) is either principal or is the least ultrafilter on \(\sup j_W[\delta]\) in \(M_W\), and therefore by \cref{Trivia} is equivalent to the least ultrafilter on \(\gamma\), which is irreducible in \(M_W\) by \cref{GeneralInternal} applied in \(M_W\). Since \(D\) is nonprincipal by assumption, and since \(W\) is irreducible, \(t_W(D)\) is nonprincipal in \(M_W\). Since \(U_*\) is either principal or irreducible in \(M_W\), and \(t_W(D)\) divides \(U_*\) in \(M_W\), \(t_W(D) \equiv^{M_W} U_*\). It follows in particular that \(U_*\) is nonprincipal and hence equivalent to the least ultrafilter on \(\gamma\). (This is easy to prove directly.)

We next show that \(\gamma = j_D(\delta)\). Since \(U_*\) has the tight covering property at \(\gamma\) in \(M_W\) and \(\text{cf}^{M_W}(\sup j^{M_W}_{U_*}\circ j_W[\delta]) = \text{cf}^{M_W}(\sup j_W[\delta]) = \gamma\),
\[\gamma = \text{cf}^{M^{M_W}_{U_*}}(\sup j^{M_W}_{U_*}\circ j_W[\delta]) \]
Now we calculate:
\begin{align*}
\text{cf}^{M^{M_W}_{U_*}}(\sup j^{M_W}_{U_*}\circ j_W[\delta]) &= \text{cf}^{M^{M_U}_{W_*}}(\sup j^{M_U}_{W_*}\circ j_U[\delta]) \\
&= \text{cf}^{M^{M_U}_{W_*}}(\sup j^{M^{M_U}_{D}}_Z\circ j^{M_U}_{D}\circ j_U[\delta])\\
&= \text{cf}^{M^{M_U}_{W_*}}( j^{M^{M_U}_{D}}_Z\circ j^{M_U}_{D}(\sup j_U[\delta]))\\
&= j^{M^{M_U}_{D}}_Z\circ j^{M_U}_{D}(\text{cf}^{M_U}(\sup j_U[\delta]))\\
&= j^{M^{M_U}_{D}}_Z(j^{M_U}_{D}(\delta))\\
&= j^{M_U}_{D}(\delta)\\
&= j_D(\delta)
\end{align*}

Let \(U'\) be the least ultrafilter on \(\gamma\) in \(M_W\), so \(U'\), \(U_*\), and \(t_W(D)\) are all equivalent in \(M_W\). 

\begin{clm}\(U' = j_D(U)\).\end{clm}

The claim leads immediately to a contradiction: since \(U' \equiv^{M_W} t_W(D)\), if \(U' \in M_D\) then \(U'\) is principal by \cref{DivisionCharacterization}, while clearly \(j_D(U)\) is nonprincipal. Therefore to complete the proof of the theorem, we just need to prove the claim.

Note that \(\text{Ord}^{\gamma}\cap M_D = \text{Ord}^{\gamma}\cap M_W\) since \(t_D(W)\equiv Z\) and \(t_W(D) \equiv U'\) are \(\gamma\)-supercompact in \(M_D\) and \(M_W\) respectively. Let \(\mathcal U\) be the normal fine ultrafilter on \(P(\delta)\) derived from \(U\) and \(\mathcal U'\) the \(M_W\)-normal fine ultrafilter on \(P^{M_W}_\gamma(\gamma)\) derived from \(U'\). Since \(\text{Ord}^{\gamma}\cap M_D = \text{Ord}^{\gamma}\cap M_W\), we have \(P^{M_W}_\gamma(\gamma) = P^{M_D}_\gamma(\gamma)\) and \(P^{M_W}(P^{M_W}_\gamma(\gamma)) = P^{M_D}(P^{M_D}_\gamma(\gamma))\). Thus \(\mathcal U'\) is an \(M_D\)-normal fine ultrafilter on \(P^{M_D}_\gamma(\gamma)\). 

We claim that \(D^-(\mathcal U') = \mathcal U\). It then follows from \cref{NormalGeneration} that \(\mathcal U' = j_D(\mathcal U)\): we have \(j_D[\mathcal U]\subseteq \mathcal U'\) since \(D^-(\mathcal U') = \mathcal U\) and \(B = \{\sigma\in P^{M_D}_\gamma(\gamma) : [\text{id}]_D\in \sigma\}\in \mathcal U'\) since \(\mathcal U'\) is fine.  But then the weakly normal ultrafilter on \(\gamma\) derived from \(\mathcal U'\) is equal to \(j_D(U)\), or in other words \(U' = j_D(U)\), as claimed.

To show that \(D^-(\mathcal U') = \mathcal U\), we first show that \(D^-(U') = U\). It suffices by \cref{ZeroUnique} to show that \(D^-(U')\) is weakly normal and concentrates on the set \(A\subseteq \delta\) of ordinals that do not carry uniform countably complete ultrafilters. Note that \(j_D(A)\) is the set of ordinals that do not carry uniform countably complete ultrafilters in \(M_D\), which is the same as the set of ordinals that do not carry uniform countably complete ultrafilters in \(M_W\). Thus \(j_D(A)\in U'\) since \(U'\) is the least ultrafilter on \(\gamma\) in \(M_W\). To show \(D^-(U')\) is weakly normal, note that \([\text{id}]^{M_D}_{U'} = [\text{id}]^{M_W}_{U'} = \sup j^{M_W}_{U'}[\gamma] = \sup j^{M_D}_{U'}\circ j_D[\delta]\) and hence \([\text{id}]_{D^-(U')} = \sup j_{D^-(U')}[\delta]\).

We finally show \(D^-(\mathcal U') = \mathcal U\). Fix a stationary partition \(\vec S\) of \(S^\delta_\omega\). Let \(A\) be the Solovay set defined from \(\vec S\) at \(\delta\). Then in \(M_D\), \(j_D(A)\) is a the Solovay set defined from \(j_D(\vec S)\) at \(\gamma\). Note that \(j_D(\vec S)\in M_W\) and \(j_D(\vec S)\) is a stationary partition of \(S^\gamma_\omega\) in \(M_W\), since \(P^{M_D}(\gamma) = P^{M_W}(\gamma)\). Since \(P^{M_W}_\gamma(\gamma) = P^{M_D}_\gamma(\gamma)\), \(j_D(A)\) is a the Solovay set defined from \(j_D(\vec S)\) in \(M_W\) at \(\gamma\). It follows from \cref{LocalSolovay} applied in \(M_W\) that \(j_D(A)\in \mathcal U'\). 

For any \(X\subseteq P_\delta(\delta)\),
\begin{align*}X\in D^-(\mathcal U')&\iff j_D(X)\in \mathcal U'\\
&\iff \{\sup \sigma : \sigma\in j_D(X)\cap j_D(A)\}\in U'\\
&\iff j_D(\{\sup \sigma : \sigma\in X\cap A\})\in U'\\
&\iff \{\sup \sigma : \sigma\in X\cap A\}\in D^-(U')\\
&\iff \{\sup \sigma : \sigma\in X\cap A\}\in U\\
&\iff X\in \mathcal U
\end{align*}
The second and the last equivalences follow from \cref{LocalSolovay}.
\end{proof}

\subsection{More supercompact cardinals}
In order to prove \cref{Global}, we prove a special case of \cref{IrredGCH} that requires a more manageable form of GCH. The proof is a minor variant on the proof of \cref{IrredGCH}. We prove only what is needed for \cref{Global} and leave it to the reader to figure out exactly the optimal local result one can get out of this variant argument.

\begin{prp}[UA]\label{IrredSingular}
Suppose \(W\) is an irreducible strongly uniform ultrafilter of completeness \(\kappa\) on a cardinal \(\lambda\). Suppose \(\nu\) is a strong limit singular cardinal of countable cofinality and \(2^\nu = \nu^+ < \lambda\). Then \(W\) is \(\nu^+\)-supercompact.
\end{prp}

We use the following version of \cref{ContinuityFactor}.

\begin{lma}[UA]\label{ContinuityFactor2}
Suppose \(W\in \textnormal{Un}\) is continuous at \(\nu^+\) where \(\nu\) is a strong limit singular cardinal of cofinality less than the critical point \(\kappa\) of \(W\). Assume \(2^\nu = \nu^+\). Then \(W\) factors as \(D\oplus Z\) where  \(D\in \textnormal{Un}_{<\nu}\) and \(\textsc{crt}(Z) > j_D(\nu^+)\).
\begin{proof}
By \cref{exponential}, there is a function \(p : \textsc{sp}(W)\to \nu^+\) such that the factor embedding \(k: M_{p_*(W)}\to M_W\) has critical point greater than \(\nu\). Since \(W\) is continuous at \(\nu\) and \(\nu^+\), there is some \(D\in \textnormal{Un}_{<\nu}\) with \(D \equiv p_*(W)\). Note that \(j_D(\nu^+) = \nu^+\) so in fact \(\textsc{crt}(k) > j_D(\nu^+)\). Moreover \[t_W(D)\in \textnormal{Un}^{M_U}_{<j_U(\nu)} = \textnormal{Un}^{M_U}_{<\nu}\] Since \(\nu\) is a strong limit and \(\textsc{crt}(k) > \nu\), \(t_W(D)\in M_D\). Hence \(t_W(D)\) is principal by \cref{DivisionCharacterization}. This implies \(D\) divides \(W\), so fix \(Z\) witnessing this. Then by \cref{mindef}, \(j_Z\restriction \text{Ord}\leq k\restriction \text{Ord}\). It follows that \(\textsc{crt}(Z) > j_D(\nu^+)\).
\end{proof}
\end{lma}

\begin{proof}[Sketch of \cref{IrredSingular}]
The proof is very similar to that of \cref{IrredGCH}, with \(\delta = \nu^+\), so we only highlight the differences.

We may assume without loss of generality that \(\kappa < \nu\). 

The only difference in showing that \(\nu^+\) carries a uniform ultrafilter and that the least \(U\in \textnormal{Un}_{\nu^+}\) is \(\nu^+\)-supercompact lies in replacing \cref{ContinuityFactor} with \cref{ContinuityFactor2}.

As in \cref{IrredGCH}, we factor \(W_* = t_U(W)\) across \(\nu^+\): since \(U\) is \(0\)-order, by \cref{ContinuityFactor2} applied in \(M_U\), \(t_U(W)\) factors as \(D\oplus Z\) where \(D\in \textnormal{Un}^{M_U}_{<\nu}\) and \(Z\in \textnormal{Un}^{M^{M_U}_D}\) is \(j^{M_U}_D(\nu^+)\)-complete. To apply \cref{ContinuityFactor2} here, we need that \(2^\nu = \nu^+\) in \(M_U\), but this follows from the fact that \(2^\nu = \nu^+\) in \(V\) combined with the \(\nu^+\)-supercompactness of \(U\).

We now prove by contradiction that \(D\) is principal, which implies \(W_*\) has critical point above \(\nu^+\), yielding the theorem. The only part of what remains of \cref{IrredGCH} that uses GCH is the fact that \(U_* = t_W(U)\) has the tight covering property at \(\gamma = \text{cf}^{M_W}(\sup j_W[\nu^+])\) in \(M_W\), which we would like to use to prove that \(\gamma = j_D(\nu^+)\). But actually the situation is a bit easier in the current context: since \(\nu\) is a strong limit singular cardinal and \(D\in \textnormal{Un}_{<\nu}\), \(j_D(\nu^+) = \nu^+\). The argument there establishes \(\gamma \leq j_D(\nu^+)\), but obviously \(\nu^+\leq \gamma\), and hence \(\gamma = \nu^+\).

The remainder of the proof is identical to that of \cref{IrredGCH}.
\end{proof}

\begin{proof}[Proof of \cref{Global}]
Suppose \(\kappa\) is strongly compact. By a theorem due to Solovay \cite{Solovay}, SCH holds above \(\kappa\), which is enough to justify all our uses of \cref{IrredSingular} below.

Let \(\gamma > \kappa\) be a strong limit cardinal of uncountable cofinality and let \(\lambda = \gamma^+\). Let \(U\) be the \(\swo\)-least \(\kappa\)-complete ultrafilter on \(\lambda\). By a variant of \cref{KetonenTheorem} due to Ketonen \cite{Ketonen}, \(U\) is \((\kappa,\lambda)\)-regular.

If \(U\) is irreducible, then by \cref{IrredSingular}, \(U\) is \({<}\gamma\)-supercompact and hence \(\kappa\) is \({<}\gamma\)-supercompact. Suppose \(U\) is not irreducible. By \cref{Factorization}, let \(D\) be a divisor of \(U\) such that \(U_* = t_D(U)\) is irreducible. Of course \(U \equiv D\oplus U_*\) by \cref{CanonicalComparison}. Since \(U\) is weakly normal, \(D\in \textnormal{Un}_{<\delta}\). Note that  \(U_*\swo j_D(U)\) since \(U\not \I D\). 

By \cref{NormalGeneration}, \(U_*\) extends the club filter on \(j_D(\lambda)\), and hence is uniform on \(j_D(\lambda)\). Let \(\kappa_* = \textsc{crt}(U_*)\). Since \(U\equiv D \oplus U_*\), \(\kappa_*\geq \kappa\). On the other hand \(\kappa_* < j_D(\kappa)\) since \(j_D(U) \slwo U_*\) is the \(\swo^{M_D}\)-least \(j_D(\kappa)\)-complete uniform ultrafilter on \(j_D(\lambda)\). By \cref{IrredSingular} applied in \(M_D\), \(U_*\) witnesses \(\kappa_*\in [\kappa,j_D(\kappa))\) is \({<}j_D(\gamma)\)-supercompact in \(M_D\).

Obviously \(\textsc{crt}(D) \geq \kappa\), but since as we have seen \(\kappa \leq \kappa_* < j_D(\kappa)\), in fact \(j_D(\kappa)\neq \kappa\), so \(\textsc{crt}(D) = \kappa\). Since in \(M_D\) there is a \({<}j_D(\gamma)\)-supercompact cardinal \(\kappa_* \in [\kappa,j_D(\kappa))\), the usual reflection argument implies \(\kappa\) is a limit of \({<}\gamma\)-supercompact cardinals. 

The theorem is proved by taking \(\gamma\) to absolute infinity and using a simple pigeonhole argument.
\end{proof}
\subsection{Some applications}
In this section we give a few applications of \cref{IrredGCH}: a characterization of weakly normal ultrafilters and an application to huge cardinals.

Our first application is essentially a restatement of \cref{IrredGCH} that clarifies how it is related to Solovay's program described in the introduction.

\begin{defn}
A countably complete strongly uniform ultrafilter \(U\) on a cardinal \(\lambda\) is {\it pre-normal} if it is Rudin-Keisler minimal among all strongly uniform ultrafilters on \(\lambda\).
\end{defn}

If \(\lambda\) is regular, then \(U\) is pre-normal if and only if \(U\) is weakly normal. On the other hand, if \(\lambda\) is singular, then no weakly normal ultrafilter on \(\lambda\) is strongly uniform by \cref{Trivia}. 

\begin{defn}
Suppose \(\lambda\) is a cardinal and \(U\) is a \(\lambda\)-decomposable countably complete ultrafilter. The {\it pre-normal ultrafilter on \(\lambda\) derived from \(U\)} is the ultrafilter derived from \(U\) using the least generator \(\theta\) of \(U\) such that \(\theta \geq \sup j_U[\lambda]\).
\end{defn}

The name is inspired by the following theorem from \cite{MitchellOrder}:

\begin{thm}\label{DerivedSuper}
Suppose \(U\) is a \(\lambda\)-supercompact, \(\lambda\)-decomposable ultrafilter. Then the pre-normal ultrafilter on \(\lambda\) derived from \(U\) is equivalent to a normal fine ultrafilter on \(P(\lambda)\).
\end{thm}

This should be seen as a generalization of Solovay's Lemma to singular cardinals.

Returning to the discussion in the introduction, suppose one wanted to generalize the proof that if \(\kappa\) carries a \(\kappa\)-complete ultrafilter, then \(\kappa\) carries a \(\kappa\)-complete normal ultrafilter. That proof really shows:

\begin{prp}\label{Normal}
If \(U\) is a \(\kappa\)-complete, \(\kappa\)-decomposable ultrafilter then the pre-normal ultrafilter on \(\kappa\) derived from \(U\) is normal.
\end{prp}

One might attempt to generalize this by starting with an arbitrary countably complete \(\lambda\)-decomposable ultrafilter, deriving its pre-normal ultrafilter \(D\) on \(\lambda\), and trying to prove that \(D\) is equivalent to a normal fine ultrafilter on \(P(\lambda)\).

This cannot work for several reasons. One is that Magidor's independence result shows that such a statement cannot be provable from ZFC alone. Another is that there are pre-normal ultrafilters that are not equivalent to normal ultrafilters. All the known examples, for example those produced by Menas, are built by hitting a small ultrafilter to produce a failure of supercompactness and then hitting a large ultrafilter to produce strong uniformity, and finally showing that the resulting iterated ultrapower is equivalent to its derived pre-normal ultrafilter. There is no known provable example that is irreducible. 

There is a good reason for this: as a corollary of our results, assuming UA + GCH, the naive attempt to generalize \cref{Normal} succeeds under the simplest condition that rules out Menas's counterexamples:

\begin{thm}[UA + GCH]
Suppose \(\lambda\) is an accessible cardinal. Suppose \(U\) is a \(\lambda\)-decomposable ultrafilter that is not divisible by any \(W\in \textnormal{Un}_{<\lambda}\). Then the pre-normal ultrafilter on \(\lambda\) derived from \(U\) is equivalent to a normal fine ultrafilter on \(P(\lambda)\).
\begin{proof}
By \cref{DerivedSuper}, it suffices to show that \(U\) is \(\lambda\)-supercompact.

The proof of \cref{IrredGCH} shows that \(U\) is \(\lambda\)-supercompact if \(\lambda\) is a successor and \({<}\lambda\)-supercompact if \(\lambda\) is a limit. In the latter case since \(\lambda\) is accessible, \(\lambda\) is singular, so in fact \(U\) is \(\lambda\)-supercompact in this case as well.
\end{proof}
\end{thm}

Our next application derives large cardinal strength in the realm of huge cardinals from a simpler ultrafilter theoretic statement.

\begin{lma}[UA]
Suppose \(\delta\) is a regular cardinal. Suppose \(W\) is the \(\swo\)-least countably complete weakly normal ultrafilter on \(\delta\) concentrating on \(S_\kappa^\delta\) for some \(\kappa < \delta\). Then \(W\) is irreducible.
\begin{proof}
Suppose \(D\) divides \(W\) and \(D\swo W\). Let \(W_* = t_D(W)\). Since \(W\) is pre-normal, \(\textsc{sp}(D) < \delta\). It follows from \cref{NormalGeneration} that \(W_*\) lies on \(j_D(\delta)\). Note that \(W_*\) is weakly normal since \([\text{id}]_{W_*} = [\text{id}]_W = \sup j_W[\delta] = \sup j^{M_D}_{W_*}[j_D(\delta)]\). Moreover \(j_W(S^\delta_\kappa) \in W_*\) since \(D^-(W_*) = W\). Therefore \(j_D(W)\wo W_*\). It follows that \(W_* = j_D(W)\), so \(D\I W\) by \cref{jInternal}. Hence \(D\) is principal.
\end{proof}
\end{lma}

\begin{cor}[UA + GCH]\label{HyperHugeGCH}
Suppose \(\delta\) is a regular cardinal carrying a countably complete ultrafilter concentrating on \(S_\kappa^\delta\) for some \(\kappa < \delta\). The \(\swo\)-least such ultrafilter \(W\) is \({<}\delta\)-supercompact. If \(\delta\) is a successor, then \(W\) is \(\delta\)-supercompact. In any case, \(W\) has the tight covering property at \(\delta\).
\begin{proof}
This is all immediate except for the tight covering property, which is only relevant when \(\delta\) is inaccessible. In fact, in this case, it is not hard to show that {\it every} ultrafilter has the tight covering property at \(\delta\):
\begin{lma}[UA + GCH]
Suppose \(\delta\) is inaccessible or a successor of a strong limit cardinal of countable cofinality. Then every countably complete ultrafilter has the tight covering property at \(\delta\).
\begin{proof}
Suppose \(\delta\) carries no countably complete ultrafilter. Then using \cref{ContinuityFactor} or \cref{ContinuityFactor2}, it is easy to see that every ultrapower fixes \(\delta\) and hence has the tight covering property.

Otherwise let \(U\) be the least ultrafilter on \(\delta\). Suppose \(W\in \textnormal{Un}\). In \(M_U\), \(\delta\) carries no countably complete ultrafilter, so \(t_U(W)\) fixes \(\delta\). As in \cref{IrredGCH}, one calculates that \(\text{cf}^{M_W}(\sup j_W[\delta]) = j^{M_U}_{t_U(W)}(\delta)\), but \( j^{M_U}_{t_U(W)}(\delta) = \delta\), so we are done by \cref{CofinalityRegularity}.
\end{proof} 
\end{lma}
This completes the proof of \cref{HyperHugeGCH}.
\end{proof}
\end{cor}

Similarly one can show the following fact for singular cardinals:

\begin{cor}[UA+GCH]
Suppose there is a cardinal \(\delta\) as in \cref{HyperHugeGCH}. Then there is an almost huge cardinal.
\begin{proof}
Let \(\delta\) be such a cardinal and \(W\) the least countably complete weakly normal ultrafilter concentrating on \(S^\delta_\kappa\) for some \(\kappa < \delta\). We then have \(j_W(\kappa) = \text{cf}^{M_W}(\sup j_W[\delta])= \delta\), where the first equality follows from Los's theorem and the second follows from the tight covering property, which holds by \cref{HyperHugeGCH}.

Note that \(\textsc{crt}(W) \leq \kappa\). Since \(W\) is \({<}\delta\)-supercompact and \(j_W(\textsc{crt}(W))\leq j_W(\kappa) = \delta\), \(W\) witnesses that \(\kappa\) is almost huge.
\end{proof}
\end{cor}

If one assumes \(\delta\) is weakly inaccessible, \cref{HyperHugeGCH} is provable without assuming GCH:

\begin{thm}[UA]
Suppose \(\delta\) is a weakly inaccessible cardinal carrying a countably complete ultrafilter concentrating on \(S_\kappa^\delta\) for some \(\kappa < \delta\). The \(\swo\)-least such ultrafilter \(W\) is \({<}\delta\)-supercompact.
\begin{proof}
Let \(W\) be the \(\swo\)-least such ultrafilter. As above it is easy to show that \(W\) is irreducible. 

Note that \(\text{cf}^{M_U}(\sup j_U[\delta]) = j_U(\kappa) < \sup j_U[\delta]\). Therefore \(\delta\) is a limit of ultrafiltered regular cardinals, so by \cref{GCHFree}, the least ultrafilter \(U\) on \(\delta\) is \({<}\delta\)-supercompact and has critical point less than or equal to \(\textsc{crt}(W)\). By \cref{LimitGCH}, we have GCH on a tail below \(\delta\). (If one assumes \(\delta\) is strongly inaccessible, one can just use Solovay's theorem on SCH here.) It follows that we can apply \cref{IrredSingular} to conclude that \(W\) is \({<}\delta\)-supercompact.
\end{proof}
\end{thm}

As our last application, we show one can rule out the existence of cardinality-preserving embeddings of the universe into an inner model assuming UA; this sort of embedding is considered in \cite{Caicedo}. This is somewhat interesting in that the hypothesis does not explicitly involve ultrafilters.

Caicedo observed that the following lemma is useful in this context: 

\begin{lma}\label{SuccessorLemma} If \(j:V\to M\) is elementary, \(\tau\) is a successor cardinal, \(j(\tau)\) is a cardinal, and \(j(\tau) \neq\tau\), then \(\sup j[\tau] < j(\tau)\).
\begin{proof}
Since \(\tau\) is a successor cardinal, \(j(\tau)\) is a successor cardinal of \(M\), and hence since \(j(\tau)\) is a cardinal, it must be a successor cardinal in \(V\). Therefore \(j(\tau)\) is regular, so \(\sup j[\tau] < j(\tau)\).
\end{proof}
\end{lma} 

We will use this several times. The following improves a lemma in Caicedo's paper, using \cref{SuccessorDecomposable}.

\begin{lma}
Suppose \(j : V\to M\) is an elementary embedding and \(\lambda\) is its first fixed point above its critical point \(\kappa\). Suppose that for all \(\gamma \in [\kappa,\lambda]\), \(j(\gamma)\) is a cardinal. Then \(j\) is discontinuous at every regular cardinal in \([\kappa,\lambda]\).
\begin{proof}
Note that \(\lambda\) is a limit cardinal since the elements of the critical sequence of \(j\) are cardinals. By \cref{SuccessorLemma}, \(j\) is continuous at every successor cardinal in \([\kappa,\lambda]\). Suppose \(\delta\in [\kappa,\lambda]\) is regular. Let \(U\) be the weakly normal ultrafilter derived from \(j\) on \(\delta^+\). Then \(U\) is \(\delta\)-decomposable. So since \(j\) factors through \(j_U\), \(j\) is discontinuous at \(\delta\).  
\end{proof}
\end{lma}

\begin{lma}[UA]\label{Assumption}
Suppose \(j : V\to M\) is an elementary embedding and \(\lambda\) is its first fixed point above its critical point \(\kappa\). Suppose \(\lambda\) is a strong limit cardinal and \(2^\lambda = \lambda^+\). Then \(j\) is continuous at \(\lambda^{+\kappa+1}\).
\begin{proof}
Suppose not, and let \(U\) be the weakly normal ultrafilter on \(\lambda^{+\kappa+1}\) derived from \(j\). By \cref{Factorization}, \(U\) can be factored as \(D\oplus Z\) where \(D\in \textnormal{Un}_{<\lambda^{+\kappa+1}}\) and \(Z\in j_D(\textnormal{Un}_{\lambda^{+\kappa+1}})\) is irreducible in \(M_Z\). 

Note that \(j_D(\lambda) = \lambda\) and \(j_Z(\lambda) = \lambda\) since \(j\) factors through \(j_Z\circ j_D\) and fixes \(\lambda\). By \cref{IrredSingular}, \(Z\) is therefore \(\lambda^+\)-supercompact in \(M_D\). We have \(\textsc{crt}(Z) < \lambda\) since in \(M_D\) there are no measurable cardinals in the interval \([\lambda, j_D(\lambda^{+\kappa + 1})]\). But the existence of such a \(Z\) contradicts Kunen's inconsistency theorem applied in \(M_Z\).
\end{proof}
\end{lma}

\begin{cor}[UA]
Suppose \(j : V\to M\) is an elementary embedding and \(\lambda\) is its first fixed point above its critical point \(\kappa\). Then for some cardinal \(\gamma \leq \lambda^{+\kappa+1}\), \(j(\gamma)\) is not a cardinal.
\begin{proof}
Suppose not. By our assumptions, every element of the critical sequence of \(j\) is a cardinal. So \(\lambda\) is a limit cardinal. Since \(\kappa\) is \(\lambda\)-strongly compact, it follows by \cref{LimitGCH} that \(\lambda\) is a strong limit cardinal. Let \(U\) be the ultrafilter on \(\lambda^{+\kappa+1}\) derived from \(j\).  Again since \(\kappa\) is \(\lambda\)-strongly compact, the least ultrafilter on \(\lambda^{+\kappa+1}\) is \(\lambda\)-supercompact by \cref{GCHFree}. Since \(\lambda\) has countable cofinality, this ultrafilter is \(\lambda^+\)-supercompact. Therefore \(2^\lambda = \lambda^+\) by Solovay's theorem.

Applying \cref{Assumption}, \(j\) is continuous at \(\lambda^{+\kappa+1}\). Since \(j(\lambda^{+\kappa+1})\) is a cardinal and \(j(\lambda^{+\kappa+1}) > j(\lambda^{+\kappa}) \geq \lambda^{+\kappa+1}\), this contradicts \cref{SuccessorLemma}.
\end{proof}
\end{cor}
\subsection{Pathological Ultrafilters}\label{LastWord}
We begin this subsection by showing quite easily that assuming UA + GCH, the internal relation and the Mitchell order essentially coincide.

\begin{lma}\label{InternalInterval}
Suppose \(\textnormal{Un}_{<\lambda}\I U\), \(W\in \textnormal{Un}_\lambda\), and \(Z\swo W \I U\). Then \(Z\I U\).
\begin{proof}[Sketch]
Fix \(Z_*\in \textnormal{Un}^{M_W}_{<j_W(\lambda)}\) with \(Z = W^-(Z_*)\). Since \(\textnormal{Un}_{<\lambda}\I U\), \[Z_*\in \textnormal{Un}^{M_W}_{<j_W(\lambda)} = j_W(\textnormal{Un}_{<\lambda})\I j_W(U)\]
Therefore \(Z_*\I^{M_W} j_W(U)\). Since \(M^{M_W}_{j_W(U)}\subseteq M_U\), it follows that \(W\oplus Z_*\I U\). Since \(Z\leq_{\text{RK}} W\oplus Z_*\), \(Z\I U\).
\end{proof}
\end{lma}

\begin{prp}[UA]\label{IUniformSuper}
Suppose \(\delta\) is a regular cardinal, \(\textnormal{Un}_{<{\delta}}\I U\), and for some \(W\in \textnormal{Un}_{\delta}\), \(W\I U\). Then \(U\) is \(\delta\)-supercompact.
\begin{proof}
If \(\textsc{crt}(U) \geq \delta\), we are done, so assume not.

Let \(Z\) be the least ultrafilter on \(\delta\). By \cref{InternalInterval}, \(Z\I U\). Assume towards a contradiction that \(U\I Z\). Then \(U,Z\) commute by \cref{InternalCommutativity}, and therefore fix each other's critical points, contradicting \cref{Kunen}. Therefore \(U\not \I Z\). By \cref{InternalCompact}, it follows that \(Z\) is \(\kappa\)-complete and \((\kappa,\delta)\)-regular for some \(\kappa\leq \delta\) closed under ultrapowers. Note that it follows from \cref{Kunen} that \(\textsc{crt}(U)\geq \kappa\) Using this we can show by the argument of \cref{Strength} that \(P(\delta)\subseteq M_U\). Moreover by \cref{CompactPropagation}, \(U\) has the tight covering property at \(\delta\). Hence \(U\) is \(\delta\)-supercompact.
\end{proof}
\end{prp}

The following gives in many cases a converse to Kunen's commuting ultrapowers lemma, \cref{CommutingUltrapowers}.

\begin{lma}\label{CommutingConverse}
Suppose for \(i = 0,1\) that \(U_i\in \textnormal{Un}_{\delta_i}\) with completeness \(\kappa_i\). Suppose \(U_i\) fixes no measurable cardinals in \([\kappa_i,\delta_i]\). Assume \(j_{U_0}(j_{U_1})= j_{U_1}\restriction M_{U_0}\) and \(j_{U_1}(j_{U_0})= j_{U_0}\restriction M_{U_1}\). Then \(\delta_0 < \kappa_1\) or \(\delta_1 < \kappa_0\). 
\begin{proof}
If \(j_{U_0}(j_{U_1})= j_{U_1}\restriction M_{U_0}\) and \(j_{U_1}(j_{U_0})= j_{U_0}\restriction M_{U_1}\) then in particular \(j_{U_0}(\kappa_1) = \kappa_1\) and \(j_{U_1}(\kappa_0) = \kappa_0\). It follows that \(\kappa_1\notin [\kappa_0,\delta_0]\) and \(\kappa_0\notin [\kappa_1,\delta_1]\). It follows that \(\delta_0 < \kappa_1\) or \(\delta_1 < \kappa_0\). 
\end{proof}
\end{lma}

\begin{thm}[UA + GCH]\label{InternalIrred}
Suppose \(U\) and \(W\) are irreducible strongly uniform ultrafilters. Then \(U\I W\) if and only if \(U\mo W\) or \(\textsc{sp}(W) < \textsc{crt}(U)\).
\begin{proof}
Let \(\lambda = \textsc{sp}(U)\). 

For the reverse direction, assume \(U\mo W\) or \(\textsc{sp}(W) < \textsc{crt}(U)\). Obviously if the latter holds then \(U\I W\). We need to show that if \(U\mo W\) then \(W\) is \(\lambda\)-supercompact. Note that if \(U\mo W\) then \(P(\lambda)\subseteq M_W\). In particular by GCH, \(\textsc{sp}(W)\geq \lambda\). If \(\lambda\) is not inaccessible then \(W\) is \(\lambda\)-supercompact by \cref{IrredGCH}. If \(\lambda\) is inaccessible then \(W\) has tight covering at \(\lambda\), but since \(U\mo W\), \(P(\lambda)\subseteq M_W\), and therefore \(W\) is \(\lambda\)-supercompact.

We now prove the forward direction. Assume \(U\I W\) and we will show that \(U\mo W\) or \(\textsc{sp}(W) < \textsc{crt}(U)\).

Suppose first that \(U\I W\) and \(W\not \I U\). We claim \(U\mo W\). It suffices to show that \(W\) is \(\lambda\)-supercompact. Note that since \(W\not \I U\), we must have \(\textsc{sp}(W)\geq \lambda\). If \(\lambda\) is not inaccessible then by \cref{IrredGCH}, this implies \(W\) is \(\lambda\)-supercompact. When \(\lambda\) is inaccessible, we can only conclude that \(W\) is \({<}\lambda\)-supercompact and has the tight covering property at \(\lambda\). In this case we can apply \cref{IUniformSuper} to conclude that \(W\) is \(\lambda\)-supercompact.

Suppose instead that \(U\I W\) and \(W \I U\). An irreducible ultrafilter fixes no regular cardinal between its critical point and its space by \cref{IrredGCH}. Therefore by \cref{CommutingConverse}, either \(\textsc{sp}(W) < \textsc{crt}(U)\) or \(\textsc{sp}(U) < \textsc{crt}(W)\). If the former holds, we are done. If the latter holds then \(U\mo W\), so we are done again. This concludes the proof of the forwards direction.
\end{proof}
\end{thm}

Since every ultrafilter factors into irreducible strongly uniform ultrafilters, this leads to a somewhat complicated characterization of the internal relation in terms of the Mitchell orders of ultrapowers. Even the statement is a bit tedious, so we omit it. 

We now turn to the main subject of this section: pathological ultrafilters.

\begin{qst}[UA + GCH]
Suppose \(U\) is an irreducible ultrafilter on an inaccessible cardinal \(\delta\). Must \(U\) be \(\delta\)-supercompact?
\end{qst}

We note that this is the same as the question of whether the {\it least} ultrafilter on \(\delta\) is \(\delta\)-supercompact. This uses a very easy version of the argument of \cref{IrredGCH}. 

\begin{prp}[UA + GCH]\label{InternalCrt}
Suppose \(\textnormal{Un}_{{<}\delta}\I W\) and \(U\) is the least ultrafilter on \(\delta\). Then \(\textsc{crt}(t_U(W)) > \delta\).
\begin{proof}
Since \(\text{Un}_{{<}\delta}\I U,W\), by \cref{CanonicalInternal}, \(\text{Un}_{{<}\delta}\I U\vee W\). It follows that \(\text{Un}^{M_U}_{{<}\delta} \I^{M_U} t_U(W)\). But \(t_U(W)\) is continuous at \(\delta\) since \(\delta\) carries no uniform ultrafilter in \(M_U\). Therefore by \cref{ContinuityFactor}, in \(M_U\), \(t_U(W)\) factors as \(D\oplus Z\) where \(D\in \textnormal{Un}_{<\delta}\) and \(\textsc{crt}(Z) > \delta\). But \(\text{Un}_{{<}\delta}\I t_U(W)\) implies \(D\I t_U(W)\), which implies \(D\) is principal.
\end{proof}
\end{prp}

\begin{prp}[UA + GCH]\label{IrredCrt}
Suppose the least ultrafilter \(U\) on \(\delta\) is \(\delta\)-supercompact. Then every irreducible strongly uniform ultrafilter \(W\) on \(\lambda\geq \delta\) is \(\delta\)-supercompact.
\begin{proof}[Sketch]
By \cref{IrredGCH}, \(W\) is \({<}\delta\)-supercompact. It follows that \(\text{Un}_{{<}\delta}\I W\). Therefore by \cref{InternalCrt}, \(\textsc{crt}(t_U(W)) > \delta\). Hence \(\text{Ord}^\delta = \text{Ord}^\delta\cap M_U\subseteq M_W\).
\end{proof}
\end{prp}

\begin{defn}
Suppose \(\delta\) is a regular cardinal. An ultrafilter \(U\in \textnormal{Un}_\delta\) is called {\it pathological} if \(U\) is \({<}\delta\)-supercompact and has the tight covering property at \(\delta\) but is not \(\delta\)-supercompact.
\end{defn}

An immediate consequence of our main theorems place stringent constraints on the type of pathologies that can arise under UA + GCH. 

\begin{prp}[UA + GCH]
Suppose \(\delta\) is a regular cardinal. An irreducible ultrafilter on \(\delta\) is either \(\delta\)-supercompact or pathological. In the latter case \(\delta\) must be strongly inaccessible.
\end{prp}

The following variant of our main question is open in ZFC.

\begin{qst}
Is it consistent (relative to large cardinals) that there is a pathological ultrafilter?
\end{qst}

A very interesting special case of this question asks whether one can in fact prove in ZFC that a \(\kappa\)-complete uniform ultrafilter on \(\kappa^+\) with the tight covering property is \(\kappa^+\)-supercompact. Of course this follows from UA by \cref{GCHFree}.

Under UA, to understand pathological ultrafilters it to some extent suffices to understand weakly normal pathological ultrafilters:

\begin{thm}[UA+GCH]
Suppose that \(U\) is a pathological ultrafilter on \(\delta\) and \(W\) is the weakly normal ultrafilter on \(\delta\) derived from \(U\). Then \(W\) is a pathological ultrafilter and \(\textnormal{Ord}^\delta\cap M_W = \textnormal{Ord}^\delta\cap M_U\).
\begin{proof}
Since \(\delta\) carries a pathological ultrafilter. Let \(k : M_W\to M_U\) be the factor embedding. Note that \(k(\delta) = \delta\) since \(k(\sup j_W[\delta]) = \sup j_U[\delta]\) and \(\text{cf}^{M_W}(\sup j_W[\delta]) = \text{cf}^{M_U}(\sup j_U[\delta]) = \delta\) by tight covering.

We first show that \(W\) is irreducible. By \cref{Factorization}, there is some divisor \(D\in \text{Un}_{<\delta}\) of \(W\) such that \(M_D\cap V_\delta = M_W\cap V_\delta\): one obtains \(D\) by ``iterating the least factor" until one reaches a strongly uniform ultrafilter on \(\delta\). To see that \(W\) is irreducible it suffices to show that \(D\) is principal. Note that \(k\circ j_D\restriction V_\delta\) is an elementary embedding from \(V_\delta\) to itself. By Kunen's inconsistency theorem, it is the identity. Hence \(D\) is principal. So \(W\) is irreducible, and therefore \(W\) is pathological.

Since \(W\) is irreducible, \(V_\delta \cap M_W = V_\delta\). Therefore the same argument shows \(k\restriction \delta\) is the identity. Since \(W\) has the tight covering property at \(\delta\), \(\delta^{+M_W} = \delta^+\). It follows that \(\textsc{crt}(k) > \delta^+\). Since we assume GCH, it follows that \(P(\delta)\cap M_W = P(\delta) \cap M_U\). Combined with the tight covering property, it follows that \(\text{Ord}^\delta\cap M_U = \text{Ord}^\delta \cap M_W\).
\end{proof}
\end{thm}

Must every irreducible pathological ultrafilter be weakly normal? One can show assuming UA + GCH that if \(U\) is \(\delta\)-supercompact, then either \(U\) is divisible by its derived pre-normal ultrafilter \(D\) on \(\delta\) or else \(D\I U\). Since pathological ultrafilters have no uniform ultrafilters on \(\delta\) internal to them, if one could generalize this fact to pathological ultrafilters, one would rule out pathology that is not essentially reducible to a weakly normal ultrafilter. But perhaps an alternate hierarchy of pathological ultrafilters emerges on inaccessible cardinals, distinct from the familiar hierarchy of supercompact ultrafilters under the Mitchell order. 

By \cref{QPoints}, if a regular cardinal \(\delta\) carries two weakly normal ultrafilters, the second is not pathological. Denote the least two weakly normal ultrafilters by \(U_0\swo U_1\). In \(M_{U_1}\), \(\delta\) carries a unique weakly normal ultrafilter, while in \(M_{U_0}\), \(\delta\) carries no ultrafilter whatsoever. By \cref{ZeroInternal}, any \(W\) with the property that \(\delta\) carries no weakly normal ultrafilter in \(M_W\) is divisible by \(U_0\). Does this generalize to \(U_1\): suppose \(\delta\) carries a unique weakly normal ultrafilter in \(M_W\). Is \(W\) divisible by \(U_1\)? 

If the answer is positive, then the least ultrafilter is the only source of pathology on \(\delta\):

\begin{prp}[UA + GCH]
Suppose \(U_0\swo U_1\) are the least weakly normal ultrafilters on \(\delta\). Assume that every \(W\in \textnormal{Un}_\delta\) such that \(\delta\) carries a unique weakly normal ultrafilter in \(M_W\) is divisible by the \(U_1\). Then every pathological ultrafilter on \(\delta\) is divisible by \(U_0\).
\begin{proof}
We first show that given such an ultrafilter \(W\), \(\delta\) does not carry a unique weakly normal ultrafilter in \(M_W\). Assume towards a contradiction that it does. Then \(U_1\) divides \(W\). Note that \(t_{U_1}(W)\) is \(\delta\)-strong and continuous at \(\delta\) (since otherwise \(U_0\D^{M_{U_1}} t_{U_1}(W)\) contradicting that \(\delta\) carries an ultrafilter in \(M_W\)). Therefore \(\textsc{crt}(t_{U_1}(W)) > \delta\). This contradicts that \(W\) is not \(\delta\)-supercompact.

We now finish the proof in general. Suppose \(\delta\) carries a second weakly normal ultrafilter \(U^*_1\) in \(M_W\). By assumption, \(U_1\) divides \(W\oplus U^*_1\). Moreover since \(U^*_1\) is \(\delta\)-supercompact in \(M_W\), \(W' = W\oplus U^*_1\) is pathological and \(\delta\) carries a unique weakly normal ultrafilter in \(M_{W'}\). But the first paragraph rules out the existence of such a  \(W'\).
\end{proof}
\end{prp}

The Mitchell order yields perhaps the most comprehensive abstract analysis of supercompact ultrafilters, especially given that it is linear on normal ultrafilters. \cref{IUniformSuper} shows that we cannot use the Mitchell order (or the internal relation) to understand pathological ultrafilters: they are all incomparable. It is tempting to try to use a weaker Mitchell relation: if \(W\) is a pathological ultrafilter, for which \(U\) can \(U\cap M_W\) belong to \(M_W\)? Note that if some cardinal is strong compact past \(\delta\), then every irreducible ultrafilter on \(\delta\) in \(M_W\) is of the form \(U\cap M_W\) for some \(U\) (and perhaps many).

As we have remarked, it is not even clear that this relation is irreflexive on nonprincipal ultrafilters in ZFC. But our main point here is that assuming UA + GCH, even under this weaker Mitchell relation, there is an enormous amount of incomparability.
\begin{lma}
Suppose \(U\) is a \(\kappa\)-complete ultrafilter with the tight covering property at \(\delta\). Then there is a function \(f : \delta \to \sup j_U[\delta]\) such that \(f\in M_U\) and \(\{\alpha < \delta : f(\alpha) = j_U(\alpha)\}\) is \({<}\kappa\)-club.
\begin{proof}
In fact if \(f : \delta \to \sup j_U[\delta]\) is any increasing continuous cofinal function with \(f\in M_U\) then since \(j_U[\delta]\) is \({<}\kappa\)-club, \(f(\alpha) = j_U(\alpha)\) on a \({<}\kappa\)-club of \(\alpha\). 
\end{proof}
\end{lma}

\begin{cor}
Suppose \(U\) is a \(\kappa\)-complete ultrafilter with the tight covering property at \(\delta\). Suppose \(W\) is a countably complete ultrafilter extending the \({<}\kappa\)-club filter on \(\delta\). Then \(j^{M_U}_{W\cap M_U} = j_W\restriction M_U\).
\begin{proof}
Let \(k : M^{M_U}_{W\cap M_U}\to j_W(M_U)\) be the factor embedding. By \cref{PushforwardLemma} it suffices to show that \(j_W(j_U)([\text{id}]_W)\in \text{ran}(k)\). Take \(f\in M_U\) such that \(\{\alpha < \delta : f(\alpha) = j_U(\alpha)\}\) is \({<}\kappa\)-club. Then \(k(j^{M_U}_{W\cap M_U}(f)([\text{id}]^{M_U}_{W\cap M_U})) = j_W(f)([\text{id}]_W)\) by the definition of the factor embedding, and \( j_W(f)([\text{id}]_W) = j_W(j_U)([\text{id}]_W)\) since \(f = j_U\) almost everywhere with respect to \(W\), since \(W\) extends the \({<}\kappa\)-club filter.
\end{proof}
\end{cor}

By \cref{IUniformSuper}, we have the following corollary.

\begin{cor}[UA + GCH]
Suppose \(U\) is a \(\kappa\)-complete pathological ultrafilter. Then for any countably complete ultrafilter \(W\) extending the \({<}\kappa\)-club filter on \(\delta\), \(W\cap M_U\notin M_U\). In particular, for any weakly normal ultrafilter \(W\) on \(\delta\) such that \(j_W(\kappa) > \delta\), \(W\cap M_U\notin M_U\).
\end{cor}

\section{Questions}
This paper probably represents only the beginning of the structure theory past strong compactness that can be established assuming the Ultrapower Axiom. There are many combinatorial questions we expect to be solvable above the least strongly compact.

\begin{qst}[UA + GCH]
Let \(\kappa\) be the least strongly compact cardinal.
\begin{enumerate}[(1)]
\item Do pathological ultrafilters exist? Or is the Mitchell order linear on irreducible ultrafilters? Is every irreducible ultrafilter Dodd solid?
\item Is there a nonreflecting stationary subset of \(S^{\kappa^{++}}_\kappa\)?
\item The linearity of the Mitchell order yields \(\diamondsuit(S^{\delta^{++}}_\delta)\) whenever \(\textsc{cf}(\delta)\geq \kappa\). Does \(\diamondsuit^+(\kappa^+)\) hold? What about \(\diamondsuit(S^{\kappa^+}_\kappa)\)?
\item A theorem of Usuba \cite{Usuba} states that if there is a hyperhuge cardinal, the generic multiverse has a minimum element. Does this follow from the existence of a strongly compact cardinal assuming UA?
\item A theorem from \cite{SO} states that for any \(X\subseteq \kappa\) coding \(V_\kappa\), \(V = \text{HOD}_X\). Letting \(\delta = \kappa^{++}\), Vopenka's theorem implies that \(V\) is a generic extension of HOD by a partial order of cardinality \(\delta\). Is \(V\) a \(\kappa\)-cc extension of \(\text{HOD}\)? Does GCH hold above \(\kappa\) in \(\text{HOD}\)?
\item Suppose \(\gamma\) is the least tall cardinal above \(\kappa\). Is \(\gamma\) a strong cardinal?
\item Suppose \(\lambda\) is least such that there is an elementary embedding \(j : V\to M\) such that \(j(\lambda) = \lambda\), \(\textsc{crt}(j) < \lambda\), and \(M\) computes cofinalities correctly below \(\lambda\). Does \(I_2\) hold at \(\lambda\)?
\item Suppose \(U,W\in \textnormal{Un}\). Must \(\textsc{crt}(U\vee W) = \min\{\textsc{crt}(U),\textsc{crt}(W)\}\)?
\end{enumerate}
\end{qst}

These are just the first questions that come to mind. There are many more raised implicitly in this paper. Of course there are the rather technical questions in the style of \cref{NextUltrafilter} of what can be proved from UA without assuming GCH, most interestingly whether there is a GCH free proof of \cref{IrredGCH}, which seems extremely likely.

\bibliography{seedorder}{}
\bibliographystyle{unsrt}

\end{document}